\documentclass[10pt]{report}
\usepackage[all]{xy}
\usepackage{amsfonts,amsmath,oldgerm,amssymb,amscd}
\newcommand{\by}[1]{\stackrel{#1}{\rightarrow}}
\newcommand{\longby}[1]{\stackrel{#1}{\longrightarrow}}

\newcommand{\ooplus}{\displaystyle{\mathop\oplus}}
\newcommand{\ol}{\overline}
\newcommand{\wt}{\widetilde}
\newcommand{\iso}{\by \sim}
\newcommand{\ra}{\rightarrow}

\newcommand{\surj}{\ra\!\!\!\ra}

\newtheorem{theorem}{Theorem}[chapter]
\newtheorem{proposition}[theorem]{Proposition}
\newtheorem{lemma}[theorem]{Lemma}

\newtheorem{corollary}[theorem]{Corollary}

\newcommand{\Z}{\mbox{$\mathbb Z$}}	
\newcommand{\Q}{\mbox{$\mathbb Q$}} 	
\newcommand{\R}{\mbox{$\mathbb R$}}     
\newcommand{\C}{\mbox{$\mathbb C$}}     

\newcommand{\M}{\mbox{$\mathcal M$}}
\newcommand{\n}{\mbox{$\mathcal N$}}

\newcommand{\gj}{\blacksquare}

\newcommand{\q}{\mbox{$\mathfrak q$}}
\newcommand{\p}{\mbox{$\mathfrak p$}}
\newcommand{\m}{\mbox{$\mathfrak m$}}
\newcommand{\G}{\mbox{$\mathfrak a$}}

\newcommand{\hh}{\mbox{\rm ht\,}}	\newcommand{\Hom}{\mbox{\rm Hom}}
  	
\newcommand{\Um}{\mbox{\rm Um}}		\newcommand{\SL}{\mbox{\rm SL}}
\newcommand{\GL}{\mbox{\rm GL}}		
	
\newcommand{\id}{\mbox{\rm Id}}

\def\Com{\UseComputerModernTips}

\begin{document}

\thispagestyle{empty}
\begin{center}
	\vspace*{.5in}
	{\LARGE \bf Euler Class Group of a Noetherian Ring}
\end{center}
\begin{center}
	\vspace*{1in}
	{\large   A thesis}\\ 
	{\large   Submitted to the University of Mumbai}\\
	{\large   for the Degree of Master of Philosophy}
\end{center}
\begin{center}	
	\vspace*{.6in}
	by\\
   	{\Large \bf  Manoj Kumar Keshari}
\end{center}
\begin{center}	
	\vspace*{1.3in}
	{\large   Tata Institute of Fundamental Research}\\
	{\large  Mumbai}\\
	{2001}
\end{center}
\newpage
\begin{center}
	\vspace*{.5in}
	{\Large \bf   CERTIFICATE}
\end{center}
	\vspace*{.2in}
\begin{center}
\begin{minipage}{4.5in}
\parindent=0pt Certified that the work contained in the thesis
entitled \linebreak {\bf Euler class group of a Noetherian ring}, by
{\bf Manoj Kumar Keshari}, has been carried out under my supervision
and that this work has not been submitted elsewhere for a degree.
\end{minipage}
\end{center}
\vspace{2in}
\hspace{2cm}
\hbox to \hsize{
 	\vbox{\hsize 2in\hrule 
	\vspace{.3cm}{(Prof. S. M. Bhatwadekar)} }
	\hspace{.4in}
	\vbox{\hsize 2in\hrule 
	\vspace{.3cm}{(Manoj Kumar Keshari)} }
	\hfill
    	}
\newpage

\begin{center}
{\Large \bf ACKNOWLEDGEMENTS} 
\end{center}
\vspace*{.5in}
It gives me great pleasure to express my deep sense of
gratitude and thanks to my supervisor Prof. S. M. Bhatwadekar, for
his valuable guidance and  suggestions throughout my thesis work. This
work is an outcome of the reading course I undertook with him over the
past one-and-half years. 

My sincere thanks to Dr. Raja Sridharan for helpful discussions. 
My thanks are also due to Prof. R. Sridharan 
for his guidance and help during my first year in TIFR. 
    
I would like to thank my friends for their
encouragement. Finally, I thank my parents for the support they have
given me all these years.

\setcounter{chapter}{-1}
\tableofcontents
\newpage


\chapter{Introduction}
	Let $A$ be a commutative Noetherian ring of dimension $n$ and
	let $P$ be a projective $A$-module. Then $P$ is said to have a
	unimodular element if there exists a surjective $A$-linear map
	$\phi : P \surj A$ (in other words $P \iso Q\oplus A$).  A
	classical theorem of Serre (\cite{SM}, Theorem 4.2.1) asserts
	that if $P$ is a projective $A$-module of rank $> n$, then $P$
	has a unimodular element.  This result is the best possible in
	the general. A standard example to show this is:

	Let $A = \R[X,Y,Z]/(X^2+Y^2+Z^2-1) = \R[x,y,z]$ be the
	coordinate ring of the real 2-sphere. Let $P = A^3/A(x,y,z)$. 
Then $P$ is a projective
	$A$-module of rank 2 and is associated to the tangent bundle
	of the real 2-sphere.  We have $P\oplus A \iso A^3$ and it is
	well known that $P \not\simeq A^2$. Hence $P$ does not have a
	unimodular element. Thus, Serre's result is not valid in
	general if rank $P = \dim A$. 

	Therefore, it is natural to ask:  \medskip

\noindent{\bf Main Question:} 
{\it Let $A$ be a commutative Noetherian ring
of dimension $n$ and let $P$ be a projective $A$-module of rank
$n$. Can we associate an invariant to $P$, the vanishing of which
would ensure that $P$ has a unimodular element?}
\medskip

	Let $A$ be a smooth affine domain over a field $k$.  Let $F^n
	K_0(A)$ denote the subgroup of $K_0(A)$ generated by the
	images of the residue fields of all the maximal ideals of
	$A$. Let $P$ be a projective $A$-module of rank $n$. Then the
	$n^{th}$ {\it Chern class} of $P$, $C_n(P) = \sum
	\,(-1)^i\,(\wedge^i P^\ast)$ (where $P^\ast$ is the dual of
	$P$) is an element of $F^n K_0(A)$. It is easy to see that if
	$P \iso Q\oplus A$, then $C_n(P) = 0$.

	In the above setting, if $k$ is an algebraically closed, then
	Murthy (\cite{Mu}, Theorem 3.8) proved that $P$ has a
	unimodular element if and only if $C_n(P) = 0$. Thus, the only
	obstruction for $P$ to have a unimodular element is the
	possible non-vanishing of its ``top Chern class'' $C_n(P)$.
	However, if $k$ is not algebraically closed, then the
	vanishing of the invariant top Chern class is not sufficient,
	as is shown by the example of the projective module associated
	to the tangent bundle of the real 2-sphere. 

	It is natural to ask, whether in the case of affine domains
	$A$ of dimension $\geq 2$ over arbitrary base fields, if one
	can attach a different invariant to a projective $A$-module
	$P$ of rank $= \dim A$, the vanishing of which would ensure
	that $P$ has a unimodular element.  To tackle this question,
	Nori defined the notion of the ``Euler class group'' of a
	smooth affine variety $X =$ Spec $(A)$ over an infinite field,
	attached to any projective $A$-module $P$ of rank $= \dim A$,
	an element in this group, called the ``Euler class'' of $P$
	and asked whether the vanishing of the Euler class of $P$
	would ensure that $P$ has a unimodular element.  In 
	\cite{B-RS1}, Bhatwadekar and Raja Sridharan settled this
	question of Nori in the affirmative for projective modules of
	trivial determinant.  In \cite{B-RS1}, an explicit
	description of the Euler class group is given, which appeared
	amenable for plausible generalization to arbitrary Noetherian
	rings. Indeed such a generalization is possible. In order to
	answer the Main Question, in \cite{B-RS3}, to any
	Noetherian ring $A$ of dimension $n\geq 2$ containing the
	field of rational numbers, an abelian group $E(A)$ is
	attached, defined roughly as follows:

	First, one takes the free abelian group on the pairs
	$(J,w_J)$, where $J\subset A$ is an ideal of height $n = \dim
	A$ such that $J/J^2$ is generated by $n$ elements and $w_J$
	a set of $n$ generators of $J/J^2$. The group $E(A)$ is a
	quotient of this group by the subgroup generated by $(J,w_J)$,
	where $J = (a_1,\ldots,a_n)$ and $w_J$ is the induced set of
	generators of $J/J^2$.  It is proved in \cite{B-RS3}, that
	if $A$ is a Noetherian ring containing the field of
	rationales, then the group $E(A)$ detects the obstruction for
	a projective $A$-module $P$ of rank $n$ with trivial
	determinant to have a unimodular element, thus answering the
	Main Question in the affirmative. \medskip

	The aim of this thesis is to give a self contained account of
	the proof of this result and give some applications. The
	layout of this thesis is as follows: In chapter 1, we recall
	some basic definitions and some well known theorems.  In
	chapter 2, we prove some preliminary results.  In chapter 3,
	we prove some addition and subtraction principles which are
	the main ingredients for the proofs of the main theorems. In
	chapter 4, we define the notion of Euler class group $E(A)$
	and show how to attach to the pair $(P,\chi)$ (where $P$ is a
	projective $A$-module of rank $n$ and $\chi : A \iso
	\wedge^n(P)$ an isomorphism), an element $e(P,\chi)$ of $E(A)$
	called the Euler class of $(P,\chi)$.  We show that $P$ has a
	unimodular element if and only if $e(P,\chi)$ vanishes. In
	chapter 5, we use the above result to prove some theorems
	about projective modules over real affine varieties.  In the
	last chapter, we define the notion of the weak Euler class
	group $E_0(A)$, which is obtained as a certain canonical
	quotient of $E(A)$.  We also define the weak Euler class of a
	projective $A$-module of rank $n = \dim A$.  It is proved
	that if $A$ is a Noetherian ring of even dimension $n$, and
	$P$ is a projective $A$-module of rank $n$ with trivial
	determinant, then the weak Euler class $e(P)$ of $P$ vanishes
	in $E_0(A)$ if and only if $[P] = [Q\oplus A]$ in $K_0(A)$
	for some projective $A$-module $Q$ of rank $n-1$.

\chapter{Some Basic Definitions}

{\bf In this thesis we assume that all rings are commutative Noetherian
with unity and all modules are finitely generated unless otherwise
stated}. 
	We assume that the multiplicative closed sets with respect to
which we localize do not contain $0$. We begin with a few definitions
and subsequently state some basic and useful results without proof.

\begin{define}
Let $A$ be a ring. The supremum of the lengths $r$, taken over all
strictly increasing chains 
$\p_0 \subset \p_1 \subset \ldots\subset \p_r$ 
of prime ideals of $A$, is called the {\it Krull dimension} of
$A$ or simply the dimension of $A$ and is denoted by $\dim A$.

	For a prime ideal $\p$ of $A$, the supremum of the lengths
$r$, taken over all strictly increasing chains 
$\p_0 \subset \p_1 \subset \ldots\subset \p_r = \p$ 
of prime ideals of $A$, is called the
the height of $\p$ and is denoted by $\hh \p$. Note that for a
Noetherian ring $A$, $\hh \p < \infty$.

	For an ideal $I \subset A$, the infimum of the heights of
$\p$, taken over all prime ideals $\p \subset A$ such that $I \subset
\p$, is defined to be {\it height} of $I$ and is denoted by $\hh I$.

	For a prime ideal $\p$ of $A$, the supremum of the lengths
$r$, taken over all strictly increasing chains 
$\p = \p_0 \subset \p_1 \subset \ldots\subset \p_r$ 
of prime ideals of $A$ starting from $\p$,
is called the the {\it coheight} of $\p$ and is denoted by coht $\p$.

	For an ideal $I \subset A$, the supremum of the coheights of
$\p$, taken over all prime ideals $\p \subset A$ such that $I \subset \p$, 
is defined to be {\it coheight} of $I$ and is denoted by coht $I$.

	It follows from the definitions that

$\hh \p = \dim A_{\p}$, coht $\p = \dim (A/\p)$ and 
$\hh \p$ + coht $\p \leq \dim A$.
\end{define}

\begin{define}
	An $A$-module $P$ is said to be {\it projective} if it
satisfies one of the following equivalent conditions:

	(i) Given $A$-modules $M,N$ and an $A$-linear surjective map
$\alpha : M \surj N$, the canonical map from $\Hom_A(P,M)$ to
$\Hom_A(P,N)$ sending $\theta$ to $\alpha\theta$ is surjective.

	(ii) Given an $A$-module $M$ and a surjective $A$-linear map
$\alpha : M \surj P$, there exists an $A$-linear map $\beta : P \ra M$
such that $\alpha \beta = 1_P$.

	(iii) There exists an $A$-module $Q$ such that $P \oplus Q
\simeq A^n$ for some positive integer $n$, i.e. $P \oplus Q$ is free.
\end{define} 

\begin{lemma}\label{Nakayama}
({\bf Nakayama Lemma}) 
	Let $A$ be a ring and let $M$ be a finitely
generated $A$-module.  Let $I\subset A$ be an ideal such that $I M
=M$.  Then, there exists $a\in I$ such that $(1+a)M = 0$.  In
particular, if $I$ is contained in Jacobson radical of $A$, then
$(1+a)$ is a unit and hence $M = 0$.
\end{lemma}

\begin{corollary}
	Let $A$ be a ring and let $M$ be a finitely generated
$A$-module. Let $I$ be an ideal contained in the Jacobson radical of
$A$ and let $N$ be a submodule of $M$. If $N + I M = M$, then $N=M$.
\end{corollary}

\begin{corollary}\label{400}
Let $A$ be a local ring with $\m$ its maximal ideal. Let $M$
be a finitely generated $A$-module. Then $\mu (M)$ (the minimum number
of generators of $M$) $= \dim_{A/\m}(M/\m M)$.
\end{corollary}

\begin{lemma}\label{manoj3}
	Let $I$ be an ideal of $A$ contained in the Jacobson radical of
$A$. Let $P,\;Q$ be projective $A$-modules such
that projective $A/I$-modules $P/IP$ and $Q/IQ$ are isomorphic. Then
$P$ and $Q$ are isomorphic as $A$-modules.
\end{lemma}

\begin{proof}
	Let $\ol \alpha : P/IP \iso Q/IQ$ be an isomorphism. 
Since $P$ is projective, $\ol \alpha$ can be lifted to an $A$-linear
map $\alpha : P \ra Q$. We claim that $\alpha$ is an isomorphism.

	Since $\ol \alpha$ is surjective, $Q = \alpha (P) + IQ$. 
As $I$ is contained in the Jacobson radical of $A$, by Nakayama lemma,
we get $Q = \alpha(P)$. Hence $\alpha$ is surjective.

Since $Q$ is projective, there exists an $A$-linear map 
	$\beta : Q \ra P$ such that $\alpha\beta = \id_Q$. 
Let $\ol \beta : Q/IQ \ra P/IP$ be the map induced by $\beta$. 
Then, we have $\ol \alpha \ol \beta = \id_{Q/IQ}$. 
	As $\ol \alpha$ is an isomorphism, we get
that $\ol \beta$ is also an isomorphism and in particular, $\ol \beta$
is surjective. Therefore $P = \beta(Q) +IP$. Hence as before, we see
that $\beta$ is surjective. Now, injectivity of $\alpha$ follows from
the fact that $\alpha\beta = \id$. $\hfill \gj$
\end{proof}

\begin{corollary}\label{manoj4}
	Let $A$ be a local ring. Then every  projective
$A$-module is free.
\end{corollary}

\begin{proof}
	Let $\mathfrak m$ be the maximal ideal $A$ and let $k =
A/\mathfrak m$ be the residue field of $A$. Let $P$ be a projective
$A$-module and let $n = \dim_k(P/\m P)$. Now, applying
(\ref{manoj3}) to the projective modules $P$ and $A^n$, we see that
$P \iso A^n$.  
$\hfill \gj$
\end{proof}

\begin{define}
({\bf Zariski Topology}) 
	For an ideal $I\subset A$, we denote by
$V(I)$, the set of all prime ideals of $A$ containing $I$. For $f\in
A$, we denote by $D(f)$, the set of all prime ideals of $A$ not
containing the element $f$. The {\it Zariski topology} on Spec$\;(A)$
is the topology for which all the closed sets are of the form $V(I)$
for some ideal $I$ of $A$ or equivalently the basic open sets are of
the form $D(f),\;f\in A$. 
\end{define}

\begin{define}
	Let $P$ be a projective $A$-module. In view of
(\ref{manoj4}), we define the rank function rank$_P$ as follows:

	rank$_P$ : Spec$\;(A) \ra \Z$ is the function defined by
rank$_P(\q) =$ rank of the free $A_{\q}$-module $P\otimes_A
A_{\q}$. If rank$_P$ is a constant function taking the value $n$, then
we define the rank of $P$ to be $n$ and denote it by rank$(P)$.
\end{define}

\begin{rem}
	rank$_P$ is a continuous function (with the discrete topology
on $\Z$ and Zariski topology on Spec $A$). Moreover, rank$_P$ is a
constant function for every finitely generated projective $A$-module
$P$ if $A$ has no non trivial idempotent elements.
\end{rem}

\begin{rem}
	As in corollary (\ref{manoj4}), one can show that if $A$ is
a semi-local ring and $P$ is a projective $A$-module of constant rank
$n$, then $P$ is free of rank $n$.
\end{rem}

\begin{define}
	Given a projective $A$-module $P$ and an element $p \in P$, we
define ${\cal O}_P(p) = \{ \alpha (p) |\alpha \in P^\ast\}$. We
say that $p$ is {\it unimodular} if ${\cal O}_P(p) = A$. The set of
all unimodular elements of $P$ is denoted by $\Um(P)$. If $P = A^n$,
then, we write $\Um_n(A)$ for $\Um(A^n)$.
\end{define}

\begin{rem}
	${\cal O}_P(p)$ is an ideal of $A$ and $p$ is unimodular if
and only if there exists $\alpha \in P^\ast$ such that $\alpha (p) =
1$. An element $(a_1,\ldots,a_n) \in A^n$ is unimodular if and only if
there exists elements $b_1,\ldots,b_n \in A$ such that
$\sum_{i=1}^n\,a_ib_i = 1$. If $(a_1,\ldots,a_n)$ is unimodular, then
we say that the row $[a_1,\ldots,a_n]$ is a {\it unimodular row}. 
\end{rem}
\medskip

We now state the  classical stability theorem of Serre.

\begin{theorem}\label{Serre}
({\bf Serre}) 
Let $A$ be a Noetherian ring of dimension $n$ and let $P$ be a projective
$A$-module of rank $> n$. Then $P \simeq Q\oplus A$.  (\cite{SM},
p. 41).
\end{theorem}

\begin{define}
	Let $A$ be a ring. Let $\GL_n(A)$ be the subset of M$_n(A)$
consisting of matrices having determinant equal to a unit in $A$.  Let
$\SL_n(A)$ be the subset of M$_n(A)$ consisting of matrices of
determinant $1$.  Let $e_{ij},\;i\neq j$ denote the $n\times n$ matrix
with 1 in the $(i,j)$ coordinate and having zeros elsewhere and 
	${\rm E}_{ij}(a) = I_n+ae_{ij},\; a\in A$. 
We denote by ${\rm E}_n(A)$ the
subgroup of $\SL_n(A)$ generated by matrices of the type ${\rm
E}_{ij}(a)$, $a\in A$.
\end{define}
\medskip

	Let $A$ be a ring. Then $\GL_n(A)$ acts on $\Um_n(A)$. If two
rows $f,g\in$ $\Um_n(A)$ are conjugate under this action, then, we shall
write $f \sim g$. This defines an equivalence relation on
Um$_n(A)$. The equivalence classes of $\Um_n(A)$ under $\sim$ are just
the orbits of the $\GL_n(A)$ action. The next proposition shows how to
associate a projective module to a unimodular row.

\begin{proposition}\label{111}
	The orbits of $\Um_n(A)$ under the $\GL_n(A)$ action
are in $1-1$ correspondence with the isomorphism classes of
$A$-modules $P$ for which $P \oplus A \simeq A^n$. Under this
correspondence, $(1,0,\ldots,0)$ corresponds to the free module
$A^{n-1}$.
\end{proposition}

\begin{proof}
	To any $[b_1,\ldots,b_n] \in  \Um_n(A)$, we can associate $P =
P(b_1,\ldots,b_n)$, the kernel of 
	$[b_1,\ldots,b_n] : A^n \surj A$. 
Such $P$ is a typical module for which $P\oplus A \simeq
A^n$. Suppose $\beta : P(b_1,\ldots,b_n) \iso P(c_1,\ldots,c_n)$ is an
isomorphism for another $[c_1,\ldots,c_n] \in \Um_n(A)$. Then, we can
complete the following commutative diagram 
\Com
$$\xymatrix{ 
	0 \ar[r] & P(b_1,\ldots,b_n) \ar[r] \ar [d]_\beta & A^n
	\ar [r]^{[b_1,\ldots,b_n]} \ar@{-->} [d] \;\;\;\; & 
	A\ar[r] \ar@{=}[d] & 0 \\ 
	0 \ar[r] & P(c_1,\ldots,c_n) \ar[r] & A^n \ar
	[r]^{[c_1,\ldots,c_n]} \;\;\;\;& A \ar[r] & 0 
	}
$$ 
	with a suitable isomorphism $A^n \iso A^n$ (note that the rows
are split exact). If $\sigma \in \GL_n(A)$ denotes the matrix of this
isomorphism, we will have 
	$[b_1,\ldots,b_n] = [c_1,\ldots,c_n]\sigma$ 
and hence $[b_1,\ldots,b_n] \sim [c_1,\ldots,c_n]$. 
	Conversely, if this equation holds for some $\sigma
\in \GL_n(A)$, then the automorphism $A^n \iso A^n$ defined by
$\sigma$ induces an isomorphism of the two kernels :
$P(b_1,\ldots,b_n) \simeq P(c_1,\ldots,c_n)$. 
$\hfill \gj$ 
\end{proof}
\medskip

We now state some well known theorems on unimodular rows.

\begin{theorem}\label{Suslin}
{(\bf Swan, Towber)} 
	Let $A$ be a commutative ring and let $v=[a^2,b,c] \in
A^3$ be a unimodular row. Then $v$ can be completed to a matrix
in $\SL_3(A)$ (\cite{S-T}, Theorem 2.1). 
\end{theorem}

\begin{theorem}
{\bf (Suslin)} Let $A$ be a commutative ring and
let	$[x_0,x_1,\ldots,x_n]\in A^{n+1}$ be a unimodular row. Let
	$r_0,\ldots,r_n$ be positive integers such that the product
	$r_0 r_1\ldots r_n$ is divisible by $n!$. Then the unimodular
	row $[x_0^{r_0},x_1^{r_1},\ldots,x_n^{r_n}]$ is completable to
	a matrix in $\GL_{n+1}(A)$ (\cite{SM}, Theorem 5.3.1).
\end{theorem}

\begin{theorem}\label{Ravi}
{\bf (Ravi A. Rao)} 
	Let $A$ be a Noetherian ring  of dimension $n$. If
$1/n! \in A$, then any unimodular row $v \in \Um_{n+1}(A[X])$ is
extended from $A$, i.e. $v \sim_{\GL_{n+1}(A[X])} v(0)$,
i.e. there exists a matrix in $\GL_{n+1}(A[X])$ which takes $v$
to $v(0)$ (\cite{Ra}, Corollary 2.5).
\end{theorem}

	The following proposition is analogous to (\ref{111}).

\begin{proposition}
For a projective $A$-module $P$, the following are equivalent:

(i) For any projective $A$-module $Q$, if $P\oplus A \iso Q\oplus A$,
then $P\iso Q$.

(ii) Given a unimodular element $(p,a) \in P\oplus A$, there exists an
automorphism $\Delta$ of $P\oplus A$ such that $\Delta (p,a) = (0,1)$.
\end{proposition}

\begin{proof}
\underline{$(i)\Rightarrow  (ii)$.}  Since $(p,a)$ is unimodular
	element of $P\oplus A$, there exists an element $\alpha \in
	(P\oplus A)^\ast$ such that $\alpha(p,a) = 1$. Let $Q$ = ker
	$(\alpha)$. Then, we get the following short exact sequence of
	$A$-modules:
$$0 \ra Q \ra P\oplus A \by \alpha A \ra 0$$ 
	Let $\beta : A \ra P\oplus A$ be an $A$-linear map such that
	$\beta(1) = (p,a)$. Then $\alpha\beta = 1_A$. Hence the
	cyclic submodule $A(p,a)$ of $P\oplus A$ is isomorphic to $A$
	and $P\oplus A = Q \oplus A(p,a)$. Therefore, by assumption,
	there exists an isomorphism $\sigma : Q \iso P$.

	Let $\Delta : Q \oplus A(p,a) \ra P\oplus A$ be an
endomorphism of $P\oplus A$ defined by $\Delta(q,0) = (\sigma(q),0)$
for $q \in Q$ and $\Delta(p,a) = (0,1)$. Then as $\sigma$ is an
isomorphism and $A(p,a) \iso A$, it follows that $\Delta$ is an
automorphism of $P\oplus A$ which sends $(p,a)$ to (0,1).
\medskip

\underline{$(ii) \Rightarrow (i)$.}  
	Let $\psi : Q \oplus A \iso P \oplus A$
be an isomorphism and let $\psi(0,1) = (p,a)$. Then as $\psi$ is an
isomorphism, $(p,a)$ is a unimodular element of $P\oplus
A$. Therefore, by assumption, there exists an automorphism $\Delta $
of $P\oplus A$ such that $\Delta (p,a) = (0,1)$. Hence the isomorphism
$\Delta \psi : Q\oplus A \iso P\oplus A$ sends the element $(0,1)$
of $Q\oplus A$ to $(0,1)$ of $P\oplus A$. Note that $Q \iso (Q\oplus
A)/A(0,1)$ and $P\iso (P\oplus A)/A(0,1)$. Hence $P\iso Q$. This
proves the result.  $\hfill \gj$
\end{proof}

\begin{define}
	Let $f_1 : M_1 \ra N$ and $f_2 : M_2 \ra N$ be homomorphisms
of $A$-modules. The {\it fiber product} of $M_1$ and $M_2$ over $N$ is
a triple $(M,g_1,g_2)$, where $M$ is an $A$-module, $g_1 : M \ra M_1$
and $g_2 : M \ra M_2$ are $A$-linear maps such that $f_1g_1 =
f_2g_2$ and the triple is universal in the sense that given any
other triple $(M^\prime,g_1^\prime,g_2^\prime)$ of this kind with
$f_1g_1^\prime = f_2g_2^\prime$, there is a unique homomorphism $h
: M^\prime \ra M$ such that $g_1 h = g_1^\prime$ and $g_2h =
g_2^\prime$.
\end{define}

\begin{example}
	Let $A$ be a commutative ring let $M$ be an $A$-module. Let 
$s,t \in A$ be
such that $As + At = A$. Then
\Com
$$\xymatrix{
	A \ar [r] \ar [d] & A_s \ar [d] \\ 
	A_t \ar [r] & A_{st}}~
\xymatrix{ 
	M \ar[r] \ar[d] & M_s \ar[d] \\
	M_t\ar[r]       & M_{st}
	}$$
are  fiber product diagrams of commutative rings and $A$-modules
respectively. 
\end{example} 

\begin{lemma}
	Let $A$ be a commutative ring and let $s,t \in A$ be such that
$(s,t) = A$. Suppose $M$ and $M^\prime$ are two $A$-modules. Let $f_1
: M_s \ra M^\prime_s$ be an $A_s$-linear map and $f_2 : M_t \ra
M^\prime_t$ be an $A_t$-linear map such that $(f_1)_t = (f_2)_s$.

(1) Then, there is an $A$-linear map $f : M \ra M^\prime$ such that
$(f)_s = f_1$ and $(f)_t = f_2$.

(2) Further, if $f_1$ and $f_2$ are injective (respectively
surjective, isomorphisms), then so is $f$. 
\Com
$$\xymatrix{ 
	M \ar@{->}[rr]\ar@{->}[dd]\ar@{->}[dr]^f & & M_s
	\ar@{->}'[d][dd] \ar@{->}[dr]^{f_1} \\ 
	& M' \ar@{->}[rr] \ar@{->}[dd] & & 
	M^\prime_s \ar@{->}[dd] \\ 
	M_t \ar@{->}'[r][rr] \ar@{->}[dr]_{f_2} & & 
	M_{st} \ar@{->}[dr]^{f_0}\\ 
	& M^\prime_t \ar@{->}[rr] & & M^\prime_{st} 
	}$$ 
\end{lemma}

\begin{define}
	For a projective $A$-module $P$, we write $(P)$ for the
isomorphism class of $P$. The {\it Grothendieck group $K_0(A)$} is an
additive abelian group generated by the symbols $(P)$ with certain
natural relations. To be precise, we let:

$G = $ free abelian group generated by $(P) : P$ is a projective
$A$-module,

$H =$ subgroup of $G$ generated by $(P\oplus Q)- (P) -(Q) : P,Q$ are
projective $A$-modules,

$K_0(A) = G/H$ and $[P]=$ image of $(P)$ in $K_0(A)$.

Thus we have $[P\oplus Q] = [P]+[Q]$ in $K_0(A)$. 
\end{define}

\begin{proposition}
Let $A$ be a  ring and
let $P$ and $Q$ be projective $A$-modules. Then the
following are equivalent :

(1) $[P] = [Q] \in K_0(A)$,

(2) there exists a projective $A$-module $T$ such that $P\oplus T
\simeq Q\oplus T$, 

(3) there exists a positive integer $t$ such that $P\oplus A^t \simeq
Q\oplus A^t$.
\end{proposition}

\begin{define}
A projective $A$-module $P$ is said to be {\it stably free} if
$[P]=[A^n]$ in $K_0(A)$ for some $n$.
\end{define}

\begin{theorem}\label{Bass}
({\bf Bass Cancellation Theorem}) 
	Let $A$ be a Noetherian ring of
dimension $n$ and let $P$ be a projective $A$-module of rank $> n$.
Suppose that $P\oplus Q \iso P^\prime \oplus Q$ for some projective
$A$-modules $P^\prime$ and $Q$. Then $P \iso P^\prime$ i.e. if rank
$P$ = rank $P' > \dim A$ and $[P]=[P']$ in $K_0(A)$, then
$P\iso P'$ (\cite{SM}, p. 42).
\end{theorem}


\chapter{Some Preliminary Results}

We begin with some lemmas on general position that are proved using
prime avoidance arguments.  

\begin{lemma}\label{manoj5}
	Let $A$ be a Noetherian ring, 
$I\subset A$ be an ideal and let $\p_1,\ldots,\p_r$ be prime
ideals of $A$. Let $I = (a_1,\ldots,a_n) \nsubseteq
\bigcup_1^r\,\p_i$. Then, there exists $b_2,\ldots,b_n \in A$ such that
$c = a_1 + b_2a_2+\ldots+b_na_n \notin \bigcup_1^r\, \p_i$.
\end{lemma}

\begin{proof}
Without loss of generality, we may assume that there are no inclusion
relations between the various prime ideals $\p_i$. We prove the lemma
by induction on the number of prime ideals. Suppose by induction, we
have chosen $c_2,\ldots,c_n\in A$ such that 
	$d_1 = a_1 +c_2a_2+\ldots+c_na_n \notin \bigcup_1^{r-1}\,
\p_i$. If $d_1 \notin \p_r$, we set $c = d_1$. We assume therefore,
that $d_1\in \p_r$. Since $I \nsubseteq \p_r$, it follows that one of
the elements $a_2,\ldots,a_n \notin \p_r$. Without loss of generality,
we assume that $a_2\notin \p_r$. We choose an element $g\in A$ such
that $g\in \bigcap_1^{r-1} \, \p_i$ and $g\notin \p_r$. Such a choice
of $g$ is possible, since there are no inclusion relations between the
various prime ideals $\p_i$. The element $c = d_1 + ga_2$ is of the
form $a_1+e_2a_2+\ldots+e_na_n$ and $c \notin \bigcup_1^r\,\p_i$.
$\hfill \gj$
\end{proof}

\begin{lemma}\label{manoj6}
Let $A$ be a Noetherian ring and let $I = (a_1,\ldots,a_n) \subset A$ be
an ideal of height $\geq n$. Then, there exists an elementary matrix
$\theta \in {\rm E}_n(A)$ such that $[a_1,\ldots,a_n] \theta =
[b_1,\ldots,b_n]$, $I = (b_1,\ldots,b_n)$ and 
$\hh (b_1,\ldots,b_i) \geq i,\; 1\leq i \leq n$.
\end{lemma}

\begin{proof}
By lemma (\ref{manoj5}), we find elements $b_2,\ldots,b_n \in A$
such that the element
	$d_1 = a_1 + b_2a_2+\ldots+b_na_n$ does
not belong to the minimal prime ideals of $A$. Hence $\hh (d_1)\geq 1$.
	The element $d_1$ is
contained in only finitely many height one prime ideals of $A$, say
$\p_1,\ldots,\p_r$. Note that $I=(d_1,a_2,\ldots, a_n)$. Applying
lemma (\ref{manoj5}) once more, we find 
	$c_1,c_3,\ldots,c_n \in A$
such that the element 
	$d_2 = a_2 + c_1d_1+c_3a_3+\ldots+c_na_n$
does not belong to any $\p_i$ for $1\leq i\leq r$. Hence 
$\hh (d_1,d_2)\geq 2$. Note that $I=(d_1,d_2,a_3,\ldots,a_n)$. Proceeding as
above, we obtain a set of generators $d_1,\ldots,d_n$ of $I$ with the
required properties. We note that the transformations we have
performed are all elementary. Hence a matrix $\theta$ exists with the
required property.  
$\hfill \gj$
\end{proof}

\begin{lemma}\label{Evans}
Let $A$ be a Noetherian ring and 
	$[a_1,\ldots,a_n,a] \in A^{n+1}$.
Then, there exists $[b_1,\ldots,b_n] \in A^n$ such that $\hh I_a
\geq n$, where 
	$I = (a_1+a b_1,\ldots,a_n+a b_n)$, i.e. if $\p
\in$ {\rm Spec} $(A)$, $I \subset \p$ and $a\notin \p$, then 
$\hh \p \geq n$. In particular, if the ideal $(a_1,\ldots,a_n,a)$ has
height $\geq n$, then $\hh I \geq n.$ Further, if
$(a_1,\ldots,a_n,a)$ is an ideal of height $\geq n$ and $I$ is a
proper ideal of $A$, then $\hh I = n$.
\end{lemma}

\begin{proof}
The only prime ideals of $A$ which survive in $A_a$ are those which do
not contain $a$. If every minimal prime ideal of $A$ contains $a$, 
then $a$ is a nilpotent element and every prime ideal contains
$a$. Hence 
there is nothing to prove. Assume that $a\in A$ is not a  nilpotent
element.  Let $\p_1,\ldots,\p_r$ be the minimal prime
ideals of $A$ which do not contain $a$.
	Applying (\ref{manoj5}), we can find $b_1\in A$ such that
 	$(a_1+a b_1)A_a \not\subset \bigcup_1^r \p_i$.
Assume that $b_1,\ldots,b_{n-1} \in A$ are chosen so that 
$\hh (a_1+ab_1,\ldots,a_{n-1}+ab_{n-1}) A_a \geq n-1$. 
Let $\q_1,\ldots, \q_s$ be the minimal prime ideals of
$(a_1+ab_1,\ldots,a_{n-1}+ab_{n-1})$ which do not contain $a$. 
Applying (\ref{manoj5}), we can find $b_n\in A$ such that 
$a_n+a b_n \notin \bigcup_1^s \q_i$ and hence
$\hh (a_1 +ab_1,\ldots,a_n + ab_n) A_a \geq n$.

Now, assume that $\hh (a_1,\ldots,a_n,a) \geq n$. We show that 
$\hh I \geq n$, where 
	$I = (a_1+ ab_1,\ldots,a_n + ab_n)$.
Assume $\hh I = r < n$. Let $\p$ be a prime ideal of $A$ containing
$I$ such that $\hh \p = r$. If $a\notin \p$, then $I_a\subset \p_a$
and $\hh \p_a = r$, a contradiction. If $a\in \p$, then $(I,a) =
(a_1,\ldots,a_n,a) \in \p$, a contradiction as 
$\hh (a_1,\ldots,a_n,a) \geq n$ and $\hh \p = r < n$. 
Hence, we have $\hh I \geq n$.

Assume that $\hh (a_1,\ldots,a_n,a) \geq n$ and $I$ is a proper ideal.
Then $\hh I\leq n$, since $I$ is generated by $n$ elements. Hence 
$\hh I = n$. 
$\hfill \gj$
\end{proof}

\begin{lemma}\label{com}
Let $[a_0,a_1,\ldots,a_n]$ be a unimodular row. If $[a_1,\ldots,a_n]$
is also unimodular, then $[a_0,a_1,\ldots,a_n]$ can be taken to
$[1,0,\ldots,0]$ by an elementary transformation.
\end{lemma}

\begin{proposition}\label{Action}
Let $A$ be a Noetherian ring of dimension $d$. If $n\geq d+2$, then 
${\rm E}_n(A)$ acts transitively on $\Um_n(A)$
\end{proposition}

\begin{proof}
Let $(a_1,\ldots,a_n) \in \Um_n(A)$. By (\ref{Evans}), there exist
	$b_1,\ldots,b_{n-1} \in A$ such that
	$(a_1+b_1a_n,\ldots,a_{n-1}+b_{n-1}a_n)$ is a unimodular
row. Hence, by (\ref{com}), $(a_1,\ldots,a_n)$ can be taken to
$(1,0,\ldots,0)$ by elementary transformations.
$\hfill \gj$
\end{proof}

\begin{lemma}
Let $A$ be a Noetherian ring of dimension $n$ and
let $[a_0,a_1,\ldots,a_n]$ be a unimodular row. Then, we can
elementarily transform  $[a_0,a_1,\ldots,a_n]$ to $[b_0,\ldots,b_n]$
such that $(1)~ 
\hh (b_1,\ldots,b_n) \geq n$ and $(2)$ if $J\subset A$ is
an ideal of height $n$, then, we can choose the elementary 
transformations
so that in  addition we have $\;(b_1,\ldots,b_n) + J = A$.
\end{lemma}

\begin{proof}
Since $\dim (A/J) = 0$, by (\ref{Action}), we may perform elementary
transformations to obtain $[b_0,b_1,\ldots,b_n]$ such that
$[b_0,b_1,\ldots,b_n]  = [0,\ldots,0,1]$
modulo $J$. Further, adding suitable multiple of $b_0$ to
$b_1,\ldots,b_n$, we may assume by (\ref{Evans}) that 
$\hh (b_1,\ldots,b_n) \geq n$, and in addition that 
$(b_1,\ldots,b_n)+J =A$. 
$\hfill \gj$
\end{proof}

\begin{lemma}\label{lem11}
Let $A$ be a Noetherian ring and let $J \subset A$ be an ideal. Let
$J_1 \subset J$ and $J_2 \subset J^2$ be two ideals of $A$ such that
$J_1 + J_2 = J.$ Then $J = J_1 + (e)$ for some $e \in J_2$ and $J_1 =
J \cap J^\prime$, where $J_2 + J^\prime = A$.
\end{lemma}

\begin{proof}
Since $J/J_1$ is an idempotent ideal of $A/J_1$, it is generated by an
idempotent element. Let $J/J_1 = (\ol e)$. Since $J_1 + J_2 = J$, we
can assume that $e \in J_2$. Since $\ol e$ is an idempotent element,
we have $e-e^2 \in J_1$. Take $J^\prime = J_1+(1-e)$. Then 
$J_2 +J^\prime = A$, since $e \in J_2$. We claim that 
$J \cap J^\prime = J_1$.

Let $x\in J\cap J^\prime$. Then $x = y+e z = y_1+(1-e)z_1$, where
$y,y_1 \in J_1$ and $z,z_1 \in A$.  This implies 
$e z-(1-e)z_1 \in J_1$.  But $e-e^2 \in J_1$, so 
$e^2z \in J_1$ and hence $e z \in J_1$. Hence $x \in J_1$. 
This proves $J \cap J^\prime = J_1$.  
$\hfill \gj$
\end{proof}

\begin{corollary}\label{222}
Let $A$ be a Noetherian local ring. Let $J\subset A$ be an ideal such
that $J=(f_1,\ldots,f_n)+J^2$. Then $J=(f_1,\ldots,f_n)$.
\end{corollary}

\begin{lemma}\label{manoj10}
Let $A$ be a Noetherian ring and $I\subset A$ an ideal. Let
$f_1,\ldots,f_n \in I$ and $J = (f_1,\ldots,f_n)$. Then $I = J$ if and
only if $I = J + I^2$ and $V(J) = V(I)$ in {\rm Spec} $(A)$.
\end{lemma}

\begin{proof}
This follows from (\ref{lem11}), however we give an independent proof.
In order to show that $J = I$, it is
enough to show that $J_{\p} = I_{\p}$ for all $\p \in$ Spec $(A)$. 
If $\p \nsupseteq J$, then $\p \nsupseteq I$ and 
$J_{\p} = I_{\p} = A_{\p}$. If
$\p \supset J$, then by hypothesis, $\p \supset I$. 
Since $I =(f_1,\ldots,f_n) + I^2$ by (\ref{222}), 
we have $J_{\p} = I_{\p}$. This proves the lemma. 
$\hfill \gj$
\end{proof}

\begin{lemma}\label{Mohan}
({\bf Mohan Kumar}) 
	Let $A$ be a Noetherian ring and let $I$ be an
ideal of $A$. Let $I/I^2$ is generated by $n$ elements as an
$A/I$-module. Let $x$ be any element of $A$. Then the ideal $(I,x)
\subset A$ is generated by $n+1$ elements \cite{MK1}.
\end{lemma}

\begin{proof}
Let $a_1,\ldots,a_n$ be elements of $I$ such that they
generate $I$ modulo $I^2$. In the ring $A/(a_1,\ldots,a_n)$, the ideal
$\ol I = I/(a_1,\ldots,a_n)$ has the property that 
 $\ol I = \ol {I^2}$. Hence $\ol I$ is generated by an
idempotent. Let $h \in I$ be any lift of this idempotent.
We see that $I = (a_1,\ldots,a_n,h)$ and $h(1-h) \in
(a_1,\ldots,a_n)$. So $(I,x) = (a_1,\ldots,a_n,h,x)$. 
We claim that the ideal 
$J = (a_1,\ldots,a_n,h+(1-h)x) \subset (I,x)$ 
is actually equal to $(I,x)$.

By multiplying $h+(1-h)x$ by $h$, we have $h^2 \in J$ (since $h(1-h)
\in J$). Since $h = h^2 + h(1-h)$, we have $h\in J$. Also 
$h + (1-h)x \in J$, hence $x\in J$. 
Thus $J = (I,x)$ which proves the claim.
$\hfill \gj$
\end{proof}

\begin{rem}
Implicit in the above proof is a proof of the following assertion.
Let $A$ be a ring, $e\in A$ be an idempotent.  Then, for any $x\in A$,
the ideals $(e,x)$ and $(e+(1-e)x)$ are equal.
\end{rem}

The following is a theorem of Eisenbud and Evans \cite{E-E} and this
is a version proved in (\cite{P}, p. 1420). This was proved in
(\ref{Evans}) when $P$ is free.

\begin{lemma}\label{cor13}
Let $A$ be a Noetherian ring and let $P$ be a projective $A$-module of
rank $n$. Let $(\alpha,a) \in (P^\ast \oplus A)$. Then, there exists an
element $\beta \in P^\ast$ such that $\hh I_a \geq n$, where $I
= (\alpha + a\beta)(P).$ In particular, if the ideal $(\alpha(P),a)$
has height $\geq n$, then $\hh I \geq n.$ Further, if
$(\alpha(P),a)$ is an ideal of height $\geq n$ and $I$ is a proper
ideal of $A,$, then $\hh I = n.$
\end{lemma}

\begin{lemma}\label{manoj}
Let $A$ be a Noetherian ring of dimension $d$ and let $P$ be a
projective $A$-module of rank $n > d$. Let $J \subset A$ be an ideal
and let $\ol \alpha : P/J P \surj J/J^2$ be a surjection. 
Then $\ol\alpha$ can be lifted to a surjection from $P$ to $J$.
\end{lemma}

\begin{proof}
Let $\delta : P \ra J$ be a lift of $\ol\alpha$. 
Then $\delta(P) + J^2 = J$ and hence, by
(\ref{lem11}), there exists $c^\prime \in J^2$ such that
$\delta(P) + (c^\prime) = J$. Now, applying (\ref{cor13}) to
the element $(\delta,c^\prime)$ of $P^\ast \oplus A$, we see
that there exists $\gamma \in P^\ast$ such that the height of the
ideal $N_{c^\prime} > n$, where 
$N = (\delta + c^\prime \gamma)(P)$. Since $\dim A = d$ and $n>d$,
 it follows that
	$({c^\prime})^r \in N$ for some
positive integer $r$. As $N + (c^\prime) = J$ and $c^\prime \in
J^2$, we have $N =  J$, by (\ref{manoj10}). Since
$\delta + c^\prime \gamma$ is also a lift of $\ol \alpha$, 
we get the result.  
$\hfill \gj$
\end{proof}

\begin{corollary}\label{cor14}
({\bf Moving Lemma}) 
	Let $A$ be a Noetherian ring of dimension $n \geq
	2$ and let $P$ be a projective $A$-module of rank $n$. 
Let $J	\subset A$ be an ideal of height $n$ and let 
$\ol \alpha : P/JP \surj J/J^2$ be a surjection. 
Then, there exists an ideal $J^\prime \subset A$ and a surjection 
$\beta : P \surj J \cap	J^\prime$ such that:

$(i)~ J + J^\prime = A$, 
$(ii)~ \beta \otimes A/J = \ol \alpha$, 
$(iii)~ \hh J^\prime \geq n$, and 

$(iv)$ Further, given finitely many ideals
$J_1,J_2,\ldots,J_r$ of height $\geq 1$, $J^\prime$ can be chosen with
the additional property that $J^\prime$ is comaximal with
$J_1,J_2,\ldots,J_r$.
\end{corollary}

\begin{proof}
Let $K = J^2 \cap J_1 \ldots \cap J_r$. Then, by the assumption,
$\hh K \geq 1$. Therefore, there exists an element $a \in K$ such that
$\hh Aa = 1$ and hence $\dim (A/Aa) \leq n-1.$ 
By (\ref{manoj}), the surjection $\ol\alpha$ can be lifted to a surjection 
$\delta : P/aP \ra J/Aa$.

Let $\theta \in$ $\Hom_A(P,J)$ be a lift of $ \delta$.  Then, as 
$J/Aa = \delta (P/aP)$, we have $\theta (P) + Aa = J$. 
Applying (\ref{cor13}) to the element $(\theta,a)$ of 
$P^\ast \oplus A$, we see that there exists $\psi \in P^\ast$ 
such that $\hh \wt J_a \geq n$, where $\wt J = (\theta + a \psi)(P)$. 
	But $(\theta(P),a) = J$ has
height $n$ and $\wt J$ is a proper ideal ($\wt J \subset J$). Hence, by
(\ref{cor13}), $\hh \wt J = n$.  Since $\wt J + Aa = J$ and $a \in J^2$,
by (\ref{lem11}), there exists an ideal $J^\prime $ of $A$ such that
$\wt J = J \cap J^\prime$ and $Aa + J^\prime = A$. Now, setting 
$\beta = \theta + a \psi$, we get

$(i)$ $\beta : P \surj \wt J = J \cap J^\prime$,

$(ii)$ $\beta \otimes A/J = \ol \alpha$, 

$(iii)~ \hh J^\prime\geq n$, since $\hh \wt J_a \geq n$ and

$(iv)~ J^\prime$ is comaximal with $J_1,\ldots,J_r$,
since $Aa + J^\prime = A$.  $\hfill \gj$
\end{proof}

\begin{lemma}\label{lem3}
Let $A$ be a  ring and let $J$ be a proper ideal of $A$. 
Let $J = (a,b) = (c,d)$. Suppose $[a,b] = [c,d]$
modulo $J^2$. Then, there exists an
automorphism $\triangle$ of $A^2$ such that 
(1) $[a,b]\triangle =[c,d]$ and 
(2) det $(\triangle) = 1.$
\end{lemma}

\begin{proof}
We have $a - c,b-d \in J^2$. So, we can write 
$a - c = aa_1 +ba_2$ and $b - d = aa_3 + ba_4$, 
where $a_i\in J$ for $1\leq i \leq 4$. Let 
$u = 1 - a_1, v = - a_2, w = - a_3,$ and $x = 1 -a_4$. 
Then, we have the following equation
$$
\left(\begin{matrix} 
	u & v   \\ w & x 
	\end{matrix} 
\right).  
	\left(\begin{matrix}  a \\ b
	\end{matrix} \right) = 
\left(\begin{matrix} c \\ d \end{matrix} 
\right),
$$
Now, we see that $ux - vw = 1-f,$ for some $f \in J$. There exists
$t_1,t_2 \in A$ such that $f = dt_2 - ct_1$.  The endomorphism
$\Delta$ of $A^2$ given by
$$\left(\begin{matrix} 
	u + bt_2  & v - a t_2  \\ 
	w + bt_1  & x - a t_1 
	\end{matrix} 
\right)$$ 
is an automorphism of determinant 1 with $[a,b] \Delta = [c,d].$
$\hfill \gj$
\end{proof}

\begin{lemma}\label{lem2}
Let $A$ be a Noetherian ring of dimension $n$ and let $J \subset A$ be an
ideal of height $n$. Let $P$ and $P_1 $ be projective $A$-modules of
rank $n$. Let $\alpha : P \ra J$ and $\beta : P_1 \ra J$ be maps such
that $\alpha \otimes A/J$ and $\beta \otimes A/J$ are surjective. Let
$\psi : P \ra P_1 $ be a homomorphism such that $\beta \psi =
\alpha$. Then $\psi \otimes A/J : P/J P \ra P_1/J P_1$ is an
isomorphism.
\end{lemma}

\begin{proof}
By Nakayama lemma, it is enough to prove that if $K = \sqrt J$, then
$\ol \psi : P/KP \ra P_1/K P_1$ is an isomorphism. Let ``bar'' denote
reduction modulo $K$.  Note that $\ol \alpha$ and $\ol \beta$ are
surjections. We prove that $\ol \alpha$ and $\ol \beta$ are
isomorphisms. Since $\ol \beta \ol \psi =
\ol \alpha$, it will follow that $\ol \psi$ is an isomorphism.
 	Since $A/K$ is semi-local, $P/KP$ and $P_1/KP_1$ are
free $A/K$-modules of rank $n$, by (\ref{manoj4}). 
Hence, in order to prove that $\ol \alpha$ and $\ol \beta$ 
are isomorphisms, it is enough to prove that $J/KJ$ is a free 
$A/K$-module of rank $n$. 
Note that $J/KJ = \oplus J/\m_i J$, where $\m_i$ are the maximal
ideals containing $K$. 
	We prove that $J/\m_i J$ is a free
$A/\m_i$-module of rank $n$.  Since $\alpha \otimes A/J$ is
surjective, $J/J^2$ is generated by $n$ elements. 
Hence, by (\ref{222}), $J_{\m_i}$ is generated by $n$ elements. 
Since $\hh J = n$, $J_{\m_i}$ cannot be generated by less than $n$
elements.  Hence $\mu (J_{\m_i}) = n$. 
Hence, by (\ref{400}), $J/\m_i J$ is a free
$A/\m_i$-module of rank $n$. This proves the lemma.  
$\hfill \gj$
\end{proof}

\begin{lemma}\label{lem9}
Let $A$ be a Noetherian ring and let $P$ be a finitely generated projective
$A$-module. Let $P[T]$ denote the projective $A[T]$-module $P\otimes_A
A[T].$ Let $\alpha(T) : P[T] \surj A[T]$ and $\beta (T) : P[T] \surj
A[T] $ be two surjections such that $\alpha(0) = \beta(0).$ 
	Suppose
further that the projective $A[T]$-modules {\rm ker} $\alpha(T)$ and
{\rm ker} $\beta(T)$ are extended from $A$. Then, there exists an
automorphism $\sigma(T)$ of $P[T]$ with $\sigma(0) = \id$ such that
$\beta(T) \sigma(T) = \alpha(T).$
\end{lemma}   

\begin{proof}
First, we show that there exists an automorphism $\theta(T)$ of $P[T]$
such that $\theta(0)=\id$ and $\alpha(T)\theta(T)= \alpha(0) \otimes
A[T]$. Let $Q=$ ker $(\alpha(T))$ and $L=$ ker $(\alpha(0))$.  Since
$Q$ is extended from $A$, there exists an isomorphism $\mu : L[T] \iso
Q$.  Since the rows of the following diagram 
\Com
$$
\xymatrix{ 
	0 \ar [r] & L[T] \ar [r] \ar [d]^\mu & P[T] \;\;\;\ar
	[r]^{\alpha(0)\otimes A[T]} \ar@{-->} [d]^{\rho(T)} & 
	A[T] \ar [r] \ar [d]^{\id} & 0  \\ 
	0 \ar [r] & Q \ar [r] & P[T] \ar [r]_{\alpha(T)} &
	A[T] \ar [r] & 0
	}$$ 
are split, we can find an automorphism $\rho(T)$ of $P[T]$ such that
the above diagram is commutative. We  have 
$\alpha(T)\rho(T)= \alpha(0)\otimes A[T]$ and hence
$\alpha(0)\rho(0)=\alpha(0)$. Consider an automorphism 
$\theta(T)= \rho(T)(\rho(0)\otimes A[T])^{-1}$ of $P[T]$. 
Then $\alpha(T) \theta(T)= ( \alpha(0)\otimes A[T])
	(\rho(0)\otimes A[T])^{-1}= 
	(\alpha(0)\otimes A[T])$ and $\theta(0)= \id$.

Similarly, we have an automorphism $\delta(T)$ of $P[T]$ such that
$\beta(T)\delta(T)=\beta(0)\otimes A[T]$ and $\delta(0)= \id$.
Consider the automorphism $\sigma(T)=\delta(T)(\theta(T))^{-1}$ of
$P[T]$.  As $\alpha(0) = \beta(0)$, we have
	$\beta(T)\sigma(T)=(\beta(0)\otimes
	A[T])(\theta(T))^{-1}=(\alpha(0)\otimes A[T])
	(\theta(T))^{-1}= \alpha(T)$
and $\sigma(0)=\id$. This proves the lemma.  
$\hfill \gj$
\end{proof}

\begin{lemma}\label{121}
Let $A$ be a ring (not necessarily commutative) and let $S\subset A$ be a
multiplicative closed set which is contained in the center of $A$. Let
$u(T)\in A_S[T]$ be a unit such that $u(0)=1$. Then, there exists $s\in
S$ such that $u(sT)$ is a unit of $A[T]$.
\end{lemma}

\begin{lemma}\label{lem10}
{\bf (Quillen)} 
	Let $A$ be a ring and let $s,t \in A$ be such that $A s +A
t = A$. Let $\sigma(T) \in \GL_n(A_{st}[T])$ be such that
$\sigma(0) = \id$. Then $\sigma(T) =
(\psi_2(T))_t(\psi_1(T))_s$, where $\psi_1(T)\in
\GL_n(A_t[T])$ such that $\psi_1(0) = \id$ and $\psi_1(T)=
\id$ modulo $(s)$ and $\psi_2(T)\in \GL_n(A_s[T])$ such
that $\psi_2(0) =\id$ and $\psi_2(T)= \id$ modulo $(t)$.
\end{lemma}

\begin{proof}
Since $\sigma (0)=\id$, $\sigma = \id + T \tau(T)$. Therefore, 
by (\ref{121}), we can choose large enough $k_1$ such 
that for all $k \geq k_1$ and	for all 
	$\lambda \in A$, $\sigma (\lambda s^k T)\in \GL_n(A_t[T])$
and $\sigma (\lambda s^k T )=\id$ modulo $(sT)$. 
Hence, we can write 
	$\sigma (\lambda s^k T)=
	(\psi_1(T))_s$ 
where $\psi_1(T) \in \GL_n(A_t[T])$ and $\psi_1(T) =
\id$ modulo $(sT)$.

Let $X$ and $Y$ be variables. Write 
	$\delta (X,T,Y) = \sigma((X+Y)T)\sigma(XT)^{-1}$. 
Clearly $\delta	(X,T,Y) \in \GL_n(A_{st}[X,T,Y])$, 
	$\delta (X,T,0) = \id$ and
	$\delta(X,0,Y)= \id$. Hence 
	$\delta(X,T,Y)= \id +YT \wt\tau(X,T,Y)$.
 We can choose large enough $k_2$ such that for all $k \geq k_2$ and
for all $\mu \in A$, $\delta(X,T,t^k \mu Y)
\in \GL_n(A_{s}[X,T,Y])$ and is identity modulo  $(tTY)$. 
	Hence, we can write $\delta(X,T,t^k\mu Y) =
(\psi_2(X,T,Y))_t$, where $\psi_2(X,T,Y) \in \GL_n(A_s[X,T,Y])$ and 
$\psi_2(X,T,Y)= \id$ modulo $(tT)$.

Take $k \geq $ max $(k_1,k_2)$. Since $A s + A t = A$, 
we have $\lambda s^k +\mu t^k = 1$ for some $\lambda,\mu \in A$. 
Now $\sigma(T) = \sigma(T) \sigma(\lambda s^k T)^{-1}
\sigma(\lambda s^k T)$. 
	We have $\sigma(\lambda s^k T) = (\psi_1(T))_s $
and $\sigma(T) \sigma(\lambda s^k T)^{-1} 
	= \sigma((\lambda s^k +\mu t^k)T) \sigma(\lambda s^k T)^{-1} 
	=\delta(\lambda s^k,T,\mu t^k)
	= (\psi_2(\lambda s^k,T,1))_t = (\psi_2(T))_t$.
	Hence, we have $\sigma(T)= (\psi_2(T))_t(\psi_1(T))_s$.
This proves the lemma.  
$\hfill \gj$
\end{proof}

\begin{define}
Let $A$ be a ring and let $M,N$ be $A$-modules. Suppose
$f,g : M \iso N$ be two isomorphisms. We say that $f$ is {\it isotopic
to} $g$ if there is an isomorphism $\phi : M[X] \iso N[X]$ such that
$\phi(0) = f$ and $\phi(1) = g$. 
	A matrix $\theta \in \GL_n(A)$
is said to be isotopic to identity if the corresponding automorphism
of $A^n$ is isotopic to identity, i.e. there exists a matrix $\alpha(X)
\in \GL_n(A[X])$ such that $\alpha(0)= \id$ and $\alpha(1)=\theta$.
\end{define}

\begin{corollary}\label{depu}
Let $A$ be a ring and $s,t \in A$  such that $As + At = A$. Let
$\theta \in \GL_n(A_{st})$ be isotopic to identity. Then
$\theta$ splits as $\theta = (\theta_1)_t (\theta_2)_s$, where
$\theta_1\in \GL_n(A_s)$ such that $\theta_1 = \id$ modulo
$(t)$ and $\theta_2\in \GL_n(A_t)$ such that $\theta_2 =
\id$ modulo $(s)$.
\end{corollary}

\begin{example}\label{333}
Elementary automorphisms are isotopic to identity. If $\sigma = \prod
(1+\lambda e_{ij})$ is an elementary automorphism of $A^n$, then $\gamma
(T) = \prod (1+\lambda T e_{ij})$ is an automorphism of
$(A[T])^n$ (in-fact elementary) 
such that $\gamma (0) = \id$ and $\gamma (1) = \sigma$.
\end{example}

\begin{define}
Let $P$ be a projective $A$-module of
rank $n$. Let $\wedge^n(P)$ denote the $n^{\rm th}$ exterior power of
$P$. Then $\wedge^n(P)$ is a projective $A$-module of rank 1 and is
called the determinant of $P$. An
$A$-linear endomorphism $\alpha$ of $P$ gives rise, in a natural way,
to an endomorphism $\wedge^n(\alpha)$ of $\wedge^n(P)$. Since rank of
$\wedge^n(P) = 1$, we have End$_A(\wedge^n(P)) = A$ and hence
$\wedge^n(\alpha) \in A$. 
	Note that $\alpha$ is an automorphism if and
only if $\wedge^n(\alpha)$ is an invertible element of $A$.

Let $P, \alpha$ be as in the above paragraph. We define the
determinant of $\alpha$ to be $\wedge^n(\alpha)$ and denote it by det
$(\alpha)$. We denote the group of automorphisms of $P$ of determinant
1 by $\SL(P)$. 
\end{define}

\begin{define}
Let $P$ be a projective $A$-module. Given an element $\phi \in P^\ast$
and an element $p \in P$, we define an endomorphism $\phi_p$ as the
composite $P \by \phi A \by p P.$

If $\phi (p) = 0$, then $\phi^2_p = 0$ and $1 + \phi_p$ is a unipotent
automorphism of $P$ and hence is an element of $\SL(P)$.

By a {\it transvection}, we mean an automorphism of $P$ of the form $1
+ \phi_p$, where $\phi(p) = 0$ and either $\phi$ is unimodular in
$P^\ast$ or $p$ is unimodular in $P$. We denote by E$(P)$ the subgroup
of $\SL(P)$ generated by all the transvections of $P$.
\end{define}

\begin{rem}
When $P=A^n$, a transvection is an element of $\SL_n(A)$ of the form
$1+vw^t$, where $v,w\in$ M$_{n\times 1}(A)$ and $wv=0$ and either $v$
or $w$ is unimodular. For example,
$e_{ij}(\lambda)=1+\lambda e_i e_j^t$ is a transvection. Hence
E$_n(A)$ is a subgroup of E$(A^n)$.
\end{rem}
\medskip

The following lemma is proved in \cite{B-R}.

\begin{lemma}\label{lem15}
Let $A$ be a Noetherian ring, $I \subset A$ ideal of $A$ and $P$ a
projective $A$-module. Then any transvection of $P/IP$ can be lifted
to a (unipotent) automorphism of $P$. 
\end{lemma}

\begin{proof}
Let $\phi^\prime \in (P/IP)^\ast$ and $p^\prime \in P/IP$ be such that
$\phi^\prime(p^\prime) = 0.$ Assume that $p^\prime$ is unimodular. 
	Let $p \in P$ (resp. $\theta \in P^\ast$) be a lift of $p^\prime$
(resp. $\phi^\prime$). Then, we have $\theta(p) = a$ for some $a \in
I$. Since $p^\prime$ is unimodular, there exists a $\psi \in P^\ast $
such that $ \psi(p) = 1 + b$ for some $b \in I$ (as $P$ is
projective). Set $\phi = (1 + b)\theta - a \psi$. Then $\phi$ is a lift
of $\phi^\prime$ and $\phi(p) = 0$. 
	Consequently, $1 + \phi_p$ is an
automorphism of $P$ lifting $1 + \phi^\prime_{p^\prime}$. 

Now, assume that $\phi^\prime$ is unimodular. Then, there exists 
$q \in P$ such that $\theta(q) = 1 + b$ for some $b \in I$. Set 
$p_1 = (1+b)p - aq$. Then $\theta(p_1) = 0$. Consequently, 
$1+\theta_{p_1}$ is an automorphism of $P$ lifting 
$1+\phi^\prime_{p^\prime}$. $\hfill \gj$
\end{proof}


\chapter{ Some Addition and Subtraction Principles}

In \cite{MK2}, Mohan Kumar proved the following theorems 

\begin{theorem}
({\rm Addition principle}) Let $A$ be a reduced affine ring of
dimension $n$ over $k$, where $k$ is algebraically closed. Let $I$ and
$J$ be two comaximal ideals of height $n$ which are generated by $n$
elements. Then $I\cap J$ is also generated by $n$ elements.
\end{theorem}

\begin{theorem}
({\rm Subtraction principle}) Let $A$ be a reduced affine ring of
dimension $n$ over $k$, where $k$ is algebraically closed.
Let $I$ and $J$ be two comaximal ideals of height $n$. Assume
that $I$ and $I\cap J$ are generated by $n$ elements. Then $J$ is also
generated by $n$ elements.
\end{theorem}

\begin{theorem}\label{100}
Let $A$ be a reduced affine ring of dimension $n$ over an
algebraically closed field. Let $P$ be a projective $A$-module of rank
$n$. If $P$ maps onto an ideal $J$ of height $n$ which is generated by
$n$ elements, then $P$ has a unimodular element.
\end{theorem}

In this chapter, we prove some addition and subtraction
principles. These are modeled upon those proved by Mohan Kumar. 
Roughly, the idea is to consider ideals $J$ together with sets of
generators of $J/J^2$ and formulate the addition and subtraction
principles using this data. 

\begin{theorem}\label{theo1.3}
{\bf (Addition Principle)} 
	Let $A$ be a Noetherian ring of dimension
$n \geq 2$. Let $J_1$ and $J_2$ be two comaximal ideals of height $n$
 and $J_3 = J_1 \cap J_2$. Suppose
$J_1 = (a_1,\ldots,a_n)$ and $J_2 = (b_1,\ldots,b_n)$.
Then $J_3 = (c_1,\ldots,c_n)$, where $a_i-c_i \in J_1^2$ and
$b_i-c_i \in J_2^2$.
\end{theorem}

\begin{proof}
We have $J_1 = (a_1,\ldots,a_n)$ and $J_2 = (b_1,\ldots,b_n)$. Since
$J_1 + J_2 = A$, we have that $[\ol a_1,\ldots,\ol a_n]$ is a
unimodular row over $A/J_2$. Since $\dim (A/J_2) = 0$, by
(\ref{Action}), there exists an elementary matrix $\ol \sigma \in
{\rm E}_n(A/J_2)$ such that $[\ol a_1,\ldots,\ol a_n] \ol \sigma =
[1,0,\ldots,0]$. 

Let $\sigma \in {\rm E}_n(A)$ be a lift of $\ol \sigma$. Let
$[a_1,\ldots,a_n]\sigma = [\wt a_1,\ldots,\wt a_n]$. Then $\wt a_1 =
1$ modulo $J_2$ and $\wt a_2,\ldots,\wt a_n \in J_2$. Hence  adding
 suitable  multiples of $a_n$ to $a_1,\ldots,a_{n-1}$,
we can assume that (1) $a_1 = 1$ modulo $J_2$, 
	(2) if $K = (a_1,\ldots,a_{n-1})$, then $\hh K = n-1$ and 
	(3) $K + J_2 = A$. Let $ S = 1 + K$. Then $S \cap J_2 \neq
\varnothing$. Hence $[b_1,\ldots,b_n] \in A_S^n $ is a unimodular row.

\paragraph{Claim :}
$[b_1,\ldots,b_n]$ can be taken to $[0,\ldots,0,1]$ by an element of 
$\SL_n(A_S)$.
\vspace*{.2in}

Assume the claim. Then, there exists an element $s \in S$ and an
automorphism $\Gamma$ of $A^n_s$ of determinant 1 such that 
$[b_1,\ldots,b_n]\Gamma = [0,\ldots,0,1]$. Since $S \cap J_2 \neq
\varnothing$, without loss of generality, we may assume that $s\in
J_2$. Hence, we have $(J_3)_s = (a_1,\ldots,a_n)_s.$
  
Let $s = 1 + t$, for some $t\in K$. Then $(J_3)_t =
(b_1,\ldots,b_n)_t.$
Since $t \in K$, we have that $[a_1,\ldots,a_{n-1}] \in A_t^{n-1}$
 is unimodular row.
	Hence, by (\ref{com}), 
$[a_1,\ldots,a_n]$ can be taken to
$[0,\ldots,0,1]$ by an elementary transformation $\Delta$ of $A^n_t$.
Hence, we have $[a_1,\ldots,a_n]\Delta_s(\Gamma^{-1})_t =
[b_1,\ldots,b_n]$. 

Let $\Phi = \Gamma_t \Delta_s(\Gamma^{-1})_t$. Then
$[a_1,\ldots,a_n](\Gamma^{-1})_t\Phi = [b_1,\ldots,b_n]$. Since
$\Delta_s$ is an elementary automorphism, $\Phi$ is isotopic to
identity automorphism of $A^n_{st}$. 
	Hence, by (\ref{depu}), there
exists a splitting $\Phi =(\Phi_1)_t (\Phi_2)_s$, where $\Phi_2$ is
an automorphism of $A^n_t$ which is identity modulo the ideal $(s)$
and $\Phi_1$ is an automorphism of $A^n_s$ which is identity modulo
the ideal $(t)$.
Let
$$[a_1,\ldots,a_n]\Gamma^{-1}\Phi_1 =
	[a_1^\prime,\ldots,a_n^\prime] : A^n_s \surj (J_3)_s,$$
$$[b_1,\ldots,b_n]\Phi_2^{-1} = [b_1^\prime,\ldots,b_n^\prime] :
	A^n_t \surj (J_3)_t.$$ 
These two surjections patch up to give a
surjection $\psi = [g_1,\ldots,g_n] : A^n \surj J_3$. 
Since $s$ is unit modulo $J_1$, the homomorphism
$A\ra A/J_1$ factors through $A_s$. Similarly, the homomorphism $A \ra
A/J_2$ factors through $A_t$. Now, since $\phi_1$ is identity modulo
the ideal
$(t)\subset J_1$ and $\phi_2$ is identity modulo $J_2$, it follows
that  $[g_1,\ldots,g_n] \otimes A/J_1$ and $[g_1,\ldots,g_n]
\otimes A/J_2$ differ from
$[a_1,\ldots,a_n] \otimes A/J_1$ and $[b_1,\ldots,b_n] \otimes A/J_2$
by an element of $\SL_n(A/J_1)$ and $\SL_n(A/J_2)$ respectively. 
	Since $\dim (A/J_i) = 0$ for $i=1,2$, $\SL_n(A/J_i) =$ E$_n(A/J_i)$.
Hence, using (\ref{lem15}), we can alter $[g_1,\ldots,g_n]$ by
an element of $\SL_n(A)$, to get a surjection $\theta : A^n \surj J_3$,
say $J_3 = (c_1,\ldots,c_n)$, such that $a_i-c_i \in J_1^2$ and
$b_i-c_i \in J_2^2$. This proves the theorem.

\paragraph{Proof of the claim.}
First, we assume that $n\geq 3$. Let ``bar'' denote modulo $K$. Then
$[\ol b_1,\ldots,\ol b_n]$ is a unimodular row in $\ol A^n_S$.
Since $\dim (A/K) = 1$, by (\ref{Action}), $[\ol b_1,\ldots,\ol
b_n]$ can be taken to $[0,\ldots,0,1]$ by an elementary
transformation, say $ \ol \psi \in $ E$_n(\ol A_S)$.  
	Taking a lift
$\psi \in$ E$_n(A_S)$ of $\ol \psi$, by (\ref{lem15}), we see that
$[b_1,\ldots,b_n]$ can be taken to $[d_1,\ldots,d_{n-1},1+d_n]$ by an
elementary transformation, where $d_i \in K_S$. Since $1+d_n$ is a unit
in $A_S$, $[d_1,\ldots,d_{n-1},1+d_n]$ can be taken to
$[0,\ldots,0,1+d_n]$ by an elementary transformation. Now,
$[0,\ldots,0,1+d_n]$ can be taken to $[0,\ldots,0,1]$ by an elementary
automorphism of $A_S^n$ by (\ref{com}). This proves the claim.

When $n =2$. Given $[b_1,b_2]$ is a unimodular row in $A_S^2$. Let
$a_1,a_2 \in A_S$ be chosen so that $a_1b_1+a_2 b_2 = 1$. Consider the
matrix $\gamma = \left[\begin{array}{cc} a_1& -b_2\\a_2& b_1
\end{array} \right]$.\\ 
Then $[b_1,b_2]\gamma = [1,0]$ and det $(\gamma) = 1$. Hence, the claim
is proved.  
$\hfill \gj$
\end{proof}

\begin{theorem}\label{theo1.4}
{\bf (Subtraction Principle)} 
	Let $A$ be a Noetherian ring of
dimension $n \geq 2$. Let $J$ and $J_1$ be two comaximal ideals of height
$n$. Let $J_2 = J \cap J_1$. Assume that $J_2 = (a_1,\ldots,a_n)$ and
$J_1 = (b_1,\ldots,b_n)$ with $a_i - b_i \in J^2_1$. Then $J =
(c_1,\ldots,c_n)$ with $a_i - c_i \in J^2$.
\end{theorem}

\begin{proof}
Let $\sigma \in {\rm E}_n(A)$. Suppose
$[b_1,\ldots,b_n]\sigma = [\wt b_1,\ldots,\wt b_n]$ and
$[a_1,\ldots,a_n]\sigma = [\wt a_1,\ldots,\wt a_n]$. Then, since 
$b_i - a_i \in J_1^2$, we have  $\wt b_i - \wt a_i \in J_1^2$. 
Therefore, without loss of generality, we can perform elementary
transformations on $[b_1,\ldots,b_n]$.

We have $(a_1,\ldots,a_n) = J\cap (b_1,\ldots,b_n)$. Let ``bar''
denote modulo $J$. Then $[\ol b_1,\ldots,\ol b_n]$ is a unimodular row
over $A/J$. Since $\dim (A/J) = 0$, by (\ref{Action}), there exists
an elementary transformation $\ol \sigma \in {\rm E}_n(A/J)$ such that
$[\ol b_1,\ldots,\ol b_n]\ol \sigma = [1,0,\ldots,0]$. 

After changing by  elementary transformation, we can assume, as
in (\ref{theo1.3}),
that $b_1 = 1$ modulo $J$, $b_i \in J, i=2,\ldots,n$ and $\hh K =
n-1$, where $K = (b_1,\ldots,b_{n-1})$. Then $K + J = A$. Let 
$S = 1 + K$. Consider the natural mapping from $A \ra A_S$. Since 
$S \cap J \neq \varnothing$, we have 
$(a_1,\ldots,a_n)_S = (b_1,\ldots,b_n)_S$.
\paragraph{Claim} 
There exists $\tau \in \GL_n(A_S)$ such that
$[a_1,\ldots,a_n]\tau = [b_1,\ldots,b_n]$. \vspace*{.2in}

Assume the claim. Then, there exists an element $b =1+a \in S$, $a \in
K$ and  $\tau \in \GL_n(A_b)$ such
that $[a_1,\ldots,a_n]\tau = [b_1, \ldots ,b_n]$. Moreover, since
$S\cap J \neq \varnothing$, we can assume that $b \in J$.

Let $\beta : (A_b)^n \surj J_b = A_b$ be defined by $\beta (e_1) = 1$ and
$\beta (e_i) = 0, i = 2,\ldots,n$ and let
$\alpha : (A_a)^n \surj J_a = (a_1, \ldots ,a_n)$ be defined by $\alpha
(e_i) = a_i$.
Since $[b_1,\ldots,b_{n-1}] \in (A_{ab})^{n-1}$ is a unimodular row, 
$[b_1,\ldots,b_n]$ can be taken to $[1,0,\ldots,0]$ by an elementary
transformation $\Delta \in {\rm E}_n(A_{ab})$.

Define $\delta : (A_{ab})^n \surj J_{ab} =A_{ab} = (J_1)_{ab}$ by
$\delta (e_i) = b_i$.  Hence, we have $\alpha_b \tau_a = \delta$
and $\delta \Delta = \beta_a$.  From these two relations, we get
$\alpha_b \tau_a \Delta = \beta_a$.  
	Let $\wt \Delta = \tau_a
\Delta \tau_a^{-1}$. Then, we have $\alpha_b \wt \Delta \tau_a =
\beta_a$. Hence $\alpha_b \wt\Delta = \beta_a \tau_a^{-1} =
(\beta \tau^{-1})_a$. 

Since $\Delta$ is an elementary automorphism, we have that $\wt
\Delta$ is isotopic to identity. Hence $\wt \Delta = (\wt
\Delta_1)_b (\wt \Delta_2)_a$, by (\ref{depu}),
 where $\wt
\Delta_1$ is an automorphism of $(A_a)^n$ which is identity modulo the
ideal $(b)$ and $\wt \Delta_2$ is an automorphism of $(A_b)^n$ which
is identity modulo the ideal $(a)$. Hence, we have $\alpha_b(\wt
\Delta_1)_b(\wt \Delta_2)_a = (\beta \tau^{-1})_a$ and hence, we
get $ (\alpha \wt \Delta_1)_b = (\beta\tau^{-1} \wt
\Delta_2^{-1})_a.$ The surjections
$$\alpha\wt\Delta_1  =
	[c_1^\prime,\ldots,c_n^\prime] : (A_a)^n \surj J_a$$
$$\beta\tau^{-1} \wt \Delta_2^{-1} =
	[d_1^\prime,\ldots,d_n^\prime] : (A_b)^n \surj J_b$$ 
patch up to give $J = (c_1, \ldots ,c_n)$ such that $c_i = c_i^\prime$
in $A_a$ and $c_i = d_i^\prime$ in $A_b$. Now, we show that 
$c_i - a_i \in J^2$.
Since $b = 1 + a \in J$, the map $A \ra
A/(b)$ factors through $A_a$.  Since 
$\wt \Delta_1 = \id$ (mod $(b)$), from the equation  
$\alpha \wt \Delta_1 = [c_1^\prime, \ldots,c_n^\prime]$, 
it follows by going modulo $J$, that $c_i-a_i\in J^2$. 

\paragraph{Proof of the claim.}
To prove the claim, replace $A$ by $A_S$. Then 
$K = (b_1,\ldots,b_{n-1})$ is an ideal of height $n-1$ such that 
$K \subset J(A)$. Given $J = (a_1,\ldots,a_n) = (b_1,\ldots,b_n)$ such
that $a_i - b_i \in J^2$ and  $\hh J = n$. To show that there exists
$\tau \in \GL_n(A)$ such that 
$[a_1,\ldots,a_n]\tau = [b_1,\ldots,b_n]$.

Let $b_i = a_i + d_i, \, d_i \in J^2$. Then 
	$d_i = \sum_{j=1}^n
\lambda_{ij}^\prime a_j, \, \lambda_{ij}^\prime \in J$. 
	Hence,
there exists $\sigma \in {\rm M}_n(A)$ such that $\sigma = \id$ (mod
$J$), and $[a_1,\ldots,a_n]\sigma = [b_1,\ldots,b_n]$. Similarly,
there exists $\theta \in {\rm M}_n(A)$ such that $\theta = \id$ (mod
$J$) and $[b_1,\ldots,b_n]\theta = [a_1,\ldots,a_n]$.

Let $\sigma = (\lambda_{ij})$ and $\theta = (\mu_{ij})$.  
	Then, we
have $[b_1,\ldots,b_n]\theta\sigma = [b_1,\ldots,b_n]$.  
	Let $\gamma = \sum_{j=1}^n \lambda_{nj}\mu_{jn}$.  
	From the above equation, we get 
	$b_n = c_1 b_1 +\ldots+c_{n-1}b_{n-1} +
	\gamma b_n$ for some $c_1,\ldots,c_{n-1} \in A$. 
	Hence, we have
	$(1-\gamma)b_n \in (b_1,\ldots,b_{n-1}) = K$.  
	Since $\hh K = n-1$
and $K$ is generated by $n-1$ elements, any minimal prime ideal
of $K$ is also of height $n-1$. Hence $b_n$ does not belong to any
minimal prime ideal of $K$.  Hence $(1-\gamma) \in \sqrt K$. This
implies that $(\gamma) + \sqrt K = A$ and so $(\gamma) + K = A$. But
$K \subset J(A)$, hence $(\gamma) = A$. 
	This shows that $\gamma \in A$
is a unit. Hence $[\lambda_{n1},\ldots,\lambda_{nn}]$ is a unimodular
row. In fact $[\lambda_{n1},\ldots,\lambda_{nn}] \in \Um_n(A,J)$,
i.e. it is a lift of the unimodular row $[0,\ldots,0,1]$ in
$\Um_n(A/J)$.

Assume that $n \geq 3$. Let ``bar'' denote modulo $K$. Since $\dim
(A/K) = 1$, by (\ref{Action}), the unimodular row $[\ol
\lambda_{n1},\ldots,\ol \lambda_{nn}]$ can be taken to
$[1,0,\ldots,0]$ by an elementary transformation $\ol \sigma_1$.
Since $K \subset J(A)$, $1+K$ are units in $A$. 
	We first show that
 $[\lambda_{n1},\ldots,\lambda_{nn}]$ can be taken to
$[1,0,\ldots,0]$ by elementary transformation. To see this, first take an
elementary lift, say $\sigma_1\in$ E$_n(A)$ of $\ol\sigma_1$. Assume
$[\lambda_{n1},\ldots,\lambda_{nn}] \sigma_1 = [1+u_1,\ldots,u_n]$,
where $u_i \in K$. Since $1+u_1$ is a unit, $[1+u_1,\ldots,u_n]$ can
be taken to $[1,0,\ldots,0]$ by an elementary transformation. Hence, the
row $[\lambda_{n1},\ldots,\lambda_{nn}]$ is completable to an
elementary matrix over $A$. 
	Let $\Delta \in {\rm E}_n(A)$ be such that
$$ \Delta = \left [
		\begin{array}{ccc}
		\Delta_1 & 	& \ast \\
			 &      & \vdots \\
		\lambda_{n1}    & \ldots & \lambda_{nn}
		\end{array}
	\right]_{n\times n}.$$
Consider the map from E$_n(A) \ra {\rm E}_n(A/J)$.  Let ``tilde'' denote
modulo $J$. Then
$$ \wt \Delta = \left [
	\begin{array}{ccc}
	\wt \Delta_1 &	&  \ast \\
		     &  &  \vdots \\  
	0	     &\dots & 1 
	\end{array}
\right]_{n\times n},$$
where $\wt \Delta_1 \in \SL_{n-1}(A/J) = {\rm E}_{n-1}(A/J)$. 
Let $\delta \in {\rm E}_{n-1}(A)$ be a lift of $\wt \Delta_1$. 
Then, the inverse $\delta^{-1} \in {\rm E}_{n-1}(A)$. Then
$$\left [
	\begin{array}{ccc}
	\delta^{-1} &	& 0 \\ 
		    &	& \vdots \\
	0	    &\dots & 1
	\end{array}\right] 
\Delta = \left [
	\begin{array}{ccc}
	\theta'  &	& \ast \\ 
		&	& \vdots \\
	\lambda_{n1}& \dots &\lambda_{nn}
	\end{array}
\right],$$
where $\theta' \in {\rm M}_{n-1}(A)$ is such that $\theta' = \id$ (mod
$J$). Hence, after changing by elementary transformation, we can 
assume that $\Delta \in {\rm E}_n(A)$ is such that $\Delta_1 = \id$
(mod $J$). Let $[a_1,\ldots,a_n] \Delta =
[a_1^\prime,\ldots,a_{n-1}^\prime,b_n]$, where $a_i - a_i^\prime \in
J^2$.  Then
$(a_1,\ldots,a_n)=(a_1^\prime,\ldots,a_{n-1}^\prime,b_n)=
(b_1,..,b_n)$.  Let $a_i^\prime = c_i + d_i b_n,c_i \in K$ and
$d_i\in A$.  Consider the matrix
$$\Gamma = \left [
	\begin{array}{cccc} 
	1 	& \dots &   0    &    -d_1 \\ 
	\vdots  &       & \vdots & \vdots \\
	0       & \dots &  1     & -d_{n-1}\\ 
	0       & \dots &  0     &  1 
	\end{array}
\right]_{n\times n}.$$ 
Then $\Gamma$ is an elementary matrix and
$[a_1^\prime,\ldots,a_{n-1}^\prime,b_n]\Gamma =
[c_1,\ldots,c_{n-1},b_n]$,  
where $a_i^\prime - c_i \in (b_n)$.  Hence
$[a_1,\ldots,a_n]\Gamma\Delta = [c_1,\ldots,c_{n-1},b_n]$ and so
$(c_1,\ldots,c_{n-1},b_n) = (b_1,\ldots,b_n)$.  
	Since $c_i - a_i \in
J^2 + Ab_n$, we have $c_i - b_i \in J^2 + Ab_n$.
Let ``bar'' denote modulo $(b_n)$. Since $K + Ab_n
= J$, we have $\ol J = (\ol
b_1,\ldots,\ol b_{n-1}) = (\ol c_1,\ldots,\ol c_{n-1})$ and $\ol b_i -
\ol c_i \in \ol J^2$.  Hence, there exists a 
$\psi \in {\rm M}_{n-1} (\ol A)$ such that 
$  \psi =\id$ (mod $\ol J$) and $[\ol b_1,\ldots,\ol
b_{n-1}]\psi = [\ol c_1,\ldots,\ol c_{n-1}]$. 

Let $\psi = (\ol {s_{ij}})$, where $\ol {s_{ij}} \in A/(b_n)$.  Then, we
have
$$h_j = \sum_{i=1}^{n-1}b_is_{ij}-c_j \in
	(b_n),{\rm for}1\leq j \leq n-1\;.$$ 
Let $h_j = f_jb_n$.  Since $ \psi =\id$ (mod $J$), we have
$s_{ii} = 1 + t_{ii} b_n + d_{ii}$  for $1 \leq i \leq n-1$, where
$t_{ii} \in A, d_{ii} \in K$ and $s_{ij} = t_{ij}b_n + d_{ij}$ for
$1\leq i,j\leq n-1$, $i\neq j$, where $t_{ij} \in A,d_{ij}\in
K$. Hence, we have the following relations
$$\sum_{i=1}^{n-1}b_i(\delta_{ij}+d_{ij}) - c_j = b_n
	(f_j - \sum_{i=1}^{n-1}b_i t_{ij}),$$
where  $\delta_{ij}$ is the {\it Kronecher delta function}. 
Let us denote by 
$$g_j = f_j - \sum_{i=1}^{n-1}b_i t_{ij}, ~1\leq j\leq n-1.$$
Consider the matrix
$$\alpha = \left [
	\begin{array}{cccc}
	{1+d_{11}} & \dots &{d_{1,n-1}} & -g_1 \\
	\vdots     &       & \vdots     & \vdots\\ 
	d_{n-1,1} & \ldots &1+d_{n-1,n-1} & -g_{n-1}\\ 
	0 &  \ldots & 0 &1 
	\end{array}
\right]_{n\times n}.$$
Then, we have $[b_1,\ldots,b_n]\alpha = [c_1,\ldots,c_{n-1},b_n]$ and
$\alpha \in \GL_n(A)$, since det $(\alpha) = 1+x$ for some $x\in
K$. But $K\subset J(A)$, hence $1+x$ is a unit in $A$. Thus, the
claim is proved.

When $n = 2$, then the claim follows from (\ref{lem3}). Hence, the
theorem is proved.  

$\hfill \gj$
\end{proof}


\begin{rem}
In fact, one can prove by the method of (\cite{MK2}, p. 248)  
that the subtraction principle (\ref{theo1.4})
implies the addition principle (\ref{theo1.3}) as follows.

Let $I$ and $J$ be two comaximal ideals of height $n$. Let
$I=(a_1,\ldots,a_n)$ and $J=(b_1,\ldots,b_n)$. We want to show that
$I\cap J = (c_1,\ldots, c_n)$ with $a_i-c_i \in I^2$ and $b_i-c_i \in
J^2$.  We can find $x_1,\ldots,x_n \in I\cap J$ which generate $I\cap
J$ modulo $(I\cap J)^2$ such that $a_i-x_i \in I^2$ and $b_i-x_i \in
J^2$. Using (\ref{cor14}), we may further assume that
 $(x_1,\ldots,x_n) = I\cap J\cap K$, where
$K$ is a height $n$ ideal, comaximal with $I\cap J$. 
	Since $I = (a_1,\ldots,a_n)$ and $a_i-x_i\in I^2$, 
	by the subtraction principle, we have 
$J\cap K=(d_1,\ldots,d_n)$ such that
 $x_i-d_i \in (J\cap K)^2$.  Since $J=(b_1,\ldots,b_n)$ and
$b_i-x_i\in J^2$, $b_i-d_i\in J^2$,
 again by the subtraction principle, $K=(g_1,\ldots,g_n)$ such that
 $d_i-g_i \in K^2$. Hence $x_i-g_i \in K^2$. Applying subtraction
principle to the ideal $I\cap J$ and $K$, we get $I\cap J=(c_1,\ldots,c_n)$
such that $x_i-c_i \in (I\cap J)^2$. Hence $a_i-c_i \in I^2$ and
$b_i-c_i \in J^2$. This proves the addition principle.
\medskip

Now, we state the subtraction principle in the general case
(\cite{B-RS3}, Theorem 3.3). Notice that when $P$ is free, this reduces to
(\ref{theo1.4}).
\end{rem}

\begin{theorem}\label{manoj1}
{\bf (Subtraction Principle)} 
Let $A$ be a Noetherian ring of
dimension $n \geq 2$. Let $P$ be a projective $A$-module of rank $n$
with trivial determinant. Let $\chi : A \by \sim \wedge^n(P)$ be an
isomorphism.  Let $J,J'$  be two ideals of $A$. Let
``bar'' denote reduction modulo $J'$. Assume
	
(i) $\hh J \geq n$, $\hh J' = n$ and $J+J'=A$.

(ii) $\alpha : P \surj J \cap J'$ and 
	$\beta : A^n \surj J' $
be two surjections. 

(iii) $\ol \alpha : \ol P \surj J'/{J'}^2$  and 
$\ol \beta : \ol {A^n} \surj J'/{J'}^2$ be
surjections induced from $\alpha$ and $\beta$ respectively.  
	
(iv) There exists an isomorphism $\delta : \ol {A^n} \by \sim
\overline P$ such that  $ \overline \alpha \delta = \overline
\beta$, and $\wedge^n(\delta) = \overline \chi$.  

Then, there exists a
surjection $\theta : P \surj J$ such that $\theta \otimes A/J = \alpha
\otimes A/J$.  
\end{theorem}

Taking $J = A$ in the above theorem, we obtain the following:

\begin{corollary}\label{manoj2}
Let $A$ be a Noetherian ring of dimension $n \geq 2$. Let $P$ be a
projective $A$-module of rank $n$ with trivial determinant.  Let
$\chi : A \by \sim \wedge^n(P)$ be an isomorphism. 
	Let $J'
\subset A$ be an ideal of height $n$. Let ``bar''
denote reduction modulo $J^\prime$. Assume 

(i) $\alpha : P \surj J^\prime$
and $\beta : A^n \surj J^\prime$ be two surjections.

(ii) $\ol \alpha : \ol P \surj
J^\prime/J^{\prime 2}$ and $\ol \beta : \ol {A^n} \surj
J^\prime/J^{\prime 2}$ be surjections induced from $\alpha$ and
$\beta$ respectively.  

(iii) There exists an isomorphism
$\delta : \ol{A^n} \by \sim \ol P$ such that $ \ol \alpha \delta
= \ol \beta\;$ and $ \wedge^n(\delta) = \ol \chi$.  

Then $P$ has a unimodular element.
\end{corollary}


\chapter{The Euler Class Group of a Noetherian Ring}

For the rest of this thesis, we assume that all rings considered
contain the field $\Q$ of rational numbers. We make this assumption as
we need to apply (\ref{prop1.2}) to show that the ``Euler class'' of
a projective module is well defined. In general, one can define the
Euler class group of $A$ with respect to any rank one projective
$A$-module $L$. However, we'll define it
with respect to $A$ only. 

Let $A$ be a Noetherian ring with $\dim A = n \geq 2$. We define the
{\it Euler class group of A}, denoted by $E(A)$, as follows:  

Let $J \subset A$ be an ideal of height $n$ such that $J/J^2$ is
generated by $n$ elements.  Let $\alpha$ and $\beta$ be two
surjections from $(A/J)^n$ to $J/J^2$. We say that $\alpha$ and
$\beta$ are {\it related} if there exists an automorphism $\sigma$ 
of $(A/J)^n$ of determinant $1$ such that $\alpha \sigma=\beta$. 
	It is easy to see that
this is an equivalence relation on the set of generators of
$J/J^2$. If $\alpha : (A/J)^n \surj J/J^2$ is a surjection,   
then by $[\alpha]$, we denote the equivalence class of
$\alpha$. We call such an equivalence class $[\alpha]$ a {\it local
orientation} of $J$. 

Since $\dim (A/J) = 0$ and $n \geq 2$ , we have $\SL_n(A/J) = {\rm
E}_n(A/J)$ and therefore, by (\ref{lem15}), the canonical map from
$\SL_n(A)$ to $\SL_n(A/J)$ is surjective. Hence, if a surjection $\alpha
: (A/J)^n \surj J/J^2$ can be lifted to a surjection $\theta : A^n
\surj J$, and $\alpha$ is equivalent to 
	$\beta : (A/J)^n \surj J/J^2$,
then $\beta$ can also be lifted to a surjection from $A^n$ to
$J$. For, let $\alpha \sigma = \beta$ for some $\sigma \in 
\SL_n(A/J)$. Then, there exists $\wt \sigma \in \SL_n(A)$ which is
a lift of $\sigma$ by (\ref{lem15}). Then $\theta \wt \sigma :
A^n \surj J$ is a lift of $\beta$. 

A local orientation $[\alpha]$ of $J$ is called a {\it global
orientation} of $J$ if the surjection $\alpha : (A/J)^n \surj J/J^2$
can be lifted to a surjection $\theta : A^n \surj J$. 

We shall also, from now on, identify a surjection $\alpha$ with
the equivalence class $[\alpha]$ to which $\alpha$ belongs. 

Let $\M \in A$ be a maximal ideal of height $n$ and $\n$ be a
$\M$-primary ideal such that $\n/\n^2$ is generated by $n$
elements. Let $w_{\n}$ be a local orientation of $\n$. Let $G$ be the
free abelian group on the set of pairs $(\n,w_{\n})$, where $\n$ is a
$\M$-primary ideal and $w_{\n}$ is a local orientation of $\n$. 

Let $J = \cap\n_i$ be the intersection of finitely many ideals
$\n_i$, where $\n_i$ is $\M_i$-primary, $\M_i \subset A$ being
distinct maximal ideals of height $n$. Assume that $J/J^2$ is
generated by $n$ elements. Let $w_J$ be a local orientation of
$J$. Then $w_J$ gives rise, in a natural way, to a local orientation
$w_{\n_i} $ of $\n_i$. We associate to the pair $(J,w_J)$, the element
$\sum(\n_i,w_{\n_i})$ of $G$. By abuse of notation, we denote the
element $\sum(\n_i,w_{\n_i})$ by $(J,w_J)$. 

Let $H$ be the subgroup of $G$ generated by set of pairs $(J,w_J)$,
where $J$ is an ideal of height $n$ which is generated by $n$ elements
and $w_J$ is a global orientation
of $J$.  We define the {\it Euler class group} of $A$, $E(A) =
G/H$. Thus $E(A)$ can be thought of as the quotient of the group of
local orientations by the subgroup generated by global
orientations. 
\medskip

One of the aims of this chapter is to prove theorem (\ref{theo2.2}) which
states that if $(J,w_J)$ is zero in $E(A)$, i.e. $(J,w_J)\in H$, then
$J$ is generated by $n$ elements and $w_J$ is a global orientation of
$J$. This is proved as follows: first assume that $J = J_1\cap J_2$,
where $J_1$ and $J_2$ are two comaximal ideals of height $n$ which
are generated by $n$ elements. Assume $w_J$ be a local orientation of
$J$ which is induced by generators of $J_1$ and $J_2$. Then $(J,w_J)
=0$ in $E(A)$. 
	By (\ref{theo1.3}), $J$ is generated by $n$ elements and 
$w_J$ is a global orientation of $J$. Now assume $J_2 =
J\cap J_1$, where $J$ and $J_1$ are comaximal ideals of height
$n$. Assume $J_1 = (a_1,\ldots,a_n)$ and $J_2 = (b_1,\ldots,b_n)$ such
that
$a_i-b_i \in J_1^2$.  Assume $w_J$ is a local orientation of $J$
which is induced by the generators of $J_2$. Then $(J,w_J)=0$ in
$E(A)$. By 
(\ref{theo1.4}) $J$ is generated by $n$ elements and $w_J$ is a
global orientation of $J$. Hence, the addition and subtraction
principles are actually special cases of (\ref{theo2.2}). Using
these two special cases and a formal
group theoretic lemma (\ref{result1}), we prove the theorem.
Using this, we show (\ref{cor2.4}) that $E(A)$ detects the
obstruction for a projective module of trivial determinant to have a
unimodular element.

\begin{lemma}\label{result1}
Let $F$ be the free abelian group with basis $(e_i)_{i\in I}$. Let
$\sim$ be an equivalence relation on $(e_i)_{i\in I}$.  Define $x \in
F$ to be ``{\rm reduced}'' if $x = e_1+\ldots+e_r$ and $e_i\neq e_j$
for $i\neq j$.  For $x \in F$ with $x = e_1+\ldots+e_r$, define ``{\rm
support}'' of $x$ to be the set $\{ e_1,\ldots,e_r\}$ and denote it by
{\rm supp} $(x)$.  
	Define $x\in F$ to be ``{\rm nicely reduced}'' if
$x = e_1+\ldots+e_r$ and $e_i\neq e_j$ for $i\neq j$ and such that no
$e_i$ belongs to the equivalence class of other $e_j$ for $i,j =
1,\ldots,r$ and $i\neq j$.  Let $S\subset F$ be such that :

(1) Every element of $S$ is nicely
reduced.

(2) Let $x,y\in F$ be nicely reduced such that $x+y$ is also nicely
reduced. Then, if any two of $x,y$ and $x+y$ belongs to $S$ so does
the third one.

(3) Let $x\in F$, $x\notin S$ and $x$ is nicely reduced and let
$J\subset I$ be a finite set. Then, there exists $y \in F$ satisfying
the following properties:

$(i)$ $y$  is nicely reduced, $(ii)$  $x+y \in S$ and $(iii)$ $y + e_j$
is nicely reduced $\forall$ $j \in J$.

Let $H$ be the subgroup of $F$ generated by $S$. Then, if $x\in H$ is
nicely reduced, then $x\in S$.
\end{lemma}

{\bf Remark} Let $x,y$ be elements of $F$ with positive coefficients
and $z=x+y$ be nicely reduced. Then $x$ and $y$ are nicely reduced.

\begin{proof}
Let $$y_1+\ldots+y_r+x = z_1+\ldots+z_s\eqno(\ast),$$ 
where $y_i,z_j \in S$, $1\leq i\leq r,\; 1\leq j\leq s$.  If
$z_1+\ldots+z_s$ is nicely reduced, then by previous remark
$y_1+\ldots+y_r$ is also nicely reduced. Hence $z_1+\ldots+z_s \in S$
and $y_1+...+y_r\in S$, by assumption $(2)$. Then $x\in S$, by assumption
$(2)$.

Now, assume that $z_1+\ldots+z_s$ is not nicely reduced. Given an
equality of the type $(\ast)$, we associate a non-negative integer
$n(\ast)$ in the following manner: For a basis element $e_i$ of $F$,
we associate a number $n(e_i(\ast))$ as follows: 
	$n(e_i(\ast))+1$ is
the cardinality of the set $\{t | e_i+ z_t$ is not nicely reduced
for $1 \leq t \leq s\}$ and let $n(\ast) = \sum n(e_i(\ast))$ for the
equation $(\ast)$, where the sum is over those $e_i$'s which belong to
the set $\bigcup_1^s\,$supp$(z_t)$. We note that $n(\ast) = 0$ if and
only if $z_1+\ldots+z_s$ is nicely reduced.

Since $z_1+\ldots+z_s$ is not nicely reduced
(i.e. $n(\ast)$ is positive), there exist $z_k,z_l,\;1\leq k,l
\leq s,\; k\neq l$ such that $z_k+z_l$ is not nicely reduced. Without
loss of generality, we can assume that $k=1,l=2$.
Let $z_1 = e_1+w_1$, $z_2 = e_1^\prime + w_2$ and $e_1 \sim
e_1^\prime$. Since $x$ is nicely reduced, at least one of
$e_1,e_1^\prime \in$ supp $(y_i)$ for some $1\leq i\leq r$. 
Without loss of generality, we can assume that $e_1 \in$ supp $(y_1)$
and assume that $y_1 = e_1 + u_1$. The equation $(\ast)$ can be
written as
$$u_1+y_2+\ldots+y_r+x = w_1+z_2+\ldots+z_s\eqno(\ast_1)$$ 

If $e_1 \in S$, then by assumption $(2)$ $u_1,w_1 \in S$. Then, we
see that $n(\ast_1) < n(\ast)$. Hence, by induction, we are through.

Now, we assume that $e_1\notin S$. Then, by assumption $(2)\; w_1,u_1
\notin S$. Let $J$ be the set $\{i\in I |\;e_i \in \bigcup_1^s{\rm
supp} \;(z_t)\}$. Then $J \subset I$ is a finite set and $w_1\in F$ is
nicely reduced such that $w_1\notin S$.  By assumption $(3)$, there
exists $\theta \in F$ such that 
	$(i)\; \theta$ is nicely reduced,
	$(ii)\; w_1 + \theta \in S$ and $(iii)\; \theta + e_j$ is nicely
reduced $\forall j\in J$. Now, we claim that $\theta + u_1 \in S$.

{\it Proof of the claim.} Let $J^\prime$ be the set $\{i\in I
|e_i \in \bigcup_1^s{\rm supp} (z_t)\bigcup {\rm supp}
(\theta)\}$.  Then $J^\prime \subset I$ is a finite set. Then, by
assumption $(3)$, there exists $\theta^\prime \in F$ such that $(i)\;
\theta^\prime$ is nicely reduced, $(ii)\; e_1 + \theta^\prime \in S$
and $(iii)\; \theta^\prime + e_j$ is nicely reduced $\forall j\in
J^\prime$.

We have $w_1+\theta \in S$, $e_1+\theta^\prime \in S$. As 
$w_1+\theta +e_1+\theta^\prime$ is nicely reduced, by
assumption $(2)$, $w_1+\theta +e_1+\theta^\prime \in S$.
	Hence $e_1+w_1+\theta+\theta^\prime \in S$ and $e_1+w_1 =
z_1\in S$. Then, by assumption $(2)$, $\theta+\theta^\prime \in S$. We
have $e_1+u_1 = y_1\in S$ and $e_1+u_1+\theta+\theta^\prime\in S$, as
it is nicely reduced. Since $e_1+\theta^\prime \in S$, we have $\theta
+ u_1 \in S$. Thus, the claim is proved.

Now, the equation $(\ast_1)$ can be written as  
$$(u_1+\theta)+y_2+\ldots+y_r+x =
	(w_1+\theta)+z_2+\ldots+z_s\eqno(\ast_2),$$ 
where $\theta +u_1 \in S$ and $\theta + w_1 \in S$. Hence, we see that
$n(\ast_2) < n(\ast_1)$, since $n(e_1(\ast_2)) < n(e_1(\ast_1))$ and
$n(e_i(\ast_2))= n(e_i(\ast_1))$ or 0 for $i\neq 1$ according as $e_i
\in \bigcup_1^s\,{\rm supp}\, (z_i)$ or $e_i \in {\rm
supp}\,(\theta)$.  Hence, by induction, the lemma follows. 
$\hfill \gj$
\end{proof}

\begin{theorem}\label{theo2.2}
Let $A$ be a Noetherian ring of dimension $n \geq 2$. Let $J \subset
A$ be an ideal of height $n$ such that $J/J^2$ is generated by $n$
elements and let $w_J : (A/J)^n \surj J/J^2$ be a local orientation
of $J$. Suppose that the image of $(J,w_J)$ is zero in the Euler class
group $E(A)$ of $A$. Then $J$ is generated by $n$ elements and
 $w_J$ is a global orientation of $J$.
\end{theorem}

\begin{proof}
Let $F$ be the free abelian group on the set of pairs $(\n,w_{\n})$
such that $\n/\n^2$ is generated by $n$ elements (where $\n$ is
$\M$-primary ideal and $\M$ is a maximal ideal of height
$n$). Define an equivalence relation on the set of pairs $(\n,w_{\n})$
by $(\n,w_{\n})\sim (\n_1,w_{\n_1})$ if $\sqrt {\n} = \sqrt {\n_1}$,
i.e. $\n,\n_1$ both are $\M$-primary ideals of $A$. Let $J\subset A$
be an ideal of height $n$ such that $J/J^2$ is generated by n elements
and  $J = \cap\n_i$ be a reduced primary decomposition of $J$.
Then, denote $(J,w_J) = \sum(\n_i,w_i)$, where $w_i$ is induced from
$w_J$.

Let $S = \{(J,w_J)\in F|J = (a_1,\ldots,a_n)\,{\rm and} \,w_J 
	= (\ol a_1,\ldots,\ol a_n)\}$. 
We check that the conditions $1,2$ and $3$
of lemma (\ref{result1}) hold. 

(1) If $(J,w_J) \in S$, then it is nicely reduced. Since, if $J =
\bigcap_{i=1}^r\n_i$, then each $\n_i$ is comaximal with the other
$\n_j,\; j\neq i$ and $(J,w_J) = (\n_1,w_1)+\ldots+(\n_r,w_r)$.

(2) If $(J,w_J)$ and $(J^\prime,w_{J^\prime})$ are nicely reduced
elements of $F$ such that $(J,w_J)+ (J^\prime,w_{J^\prime})$ is also
nicely reduced (i.e. $J+J^\prime = A$), then, by the addition principle
(\ref{theo1.3}) and subtraction principle (\ref{theo1.4}), it
follows that if any two of $(J,w_J),(J^\prime,w_{J^\prime})$ and
$(J,w_J)+(J^\prime,w_{J^\prime})$ belong to $S$, then so does the
third. 

(3) Similarly, by (\ref{cor14}), condition 3 of lemma
(\ref{result1}) holds.  
Now, applying (\ref{result1}), the theorem is proved.  $\hfill
\gj$
\end{proof}

\begin{lemma}\label{lem1.1}
Let $A$ be a Noetherian ring and let $P$ be a projective $A$-module of rank
$n$. Let $\lambda : P \surj J_0$ and $\mu : P \surj J_1 $ be
surjections, where $J_0, J_1 \subset A$ are ideals of height
$n$. Then, there exists an ideal $I$ of $A[T]$ of height $n$ and a
surjection $\alpha (T) : P[T] \surj I$ such that 
	$I(0) = J_0 , \alpha (0) = \lambda$ and 
	$I(1) = J_1, \alpha(1) = \mu$, where for $a \in A,
I(a) = \{F(a):F(T) \in I\}.$
\end{lemma}

\begin{proof}
Let $\alpha (T) = T \mu (T) + (1-T)\lambda (T)$, where
$\lambda(T) = \lambda \otimes A[T]$ and $\mu(T) = \mu
\otimes A[T]$. Then $\alpha(0) =\lambda$ and $\alpha(1) = \mu$.
\medskip
 
\noindent{\bf Claim} $(\alpha(T)(P[T])+(T(1-T)))= (J_0A[T],T) \cap
 (J_1A[T],1-T)$. \medskip

Clearly LHS $\subset$ RHS. Now, let $G = Tf + g = (1-T)f_1 + g_1 \in$
RHS, where $f,f_1 \in A[T], \;g \in J_0A[T], \;g_1 \in J_1A[T]$. 
	Then $T(f + f_1) = f_1 + g_1 - g$. Write 
	$G = (1-T) f_1 + g_1 =
	T(1-T)(f+f_1) + (1-T)g +Tg_1$. 
We want to show that $(1-T)g+Tg_1 \in$ LHS.
Now, there exist $p(T),q(T)\in P[T]$ such that $\lambda(T)(p(T))= g$
and $\mu(T)(q(T))= g_1$. 
	Hence $\alpha(T)((1-T)p(T))= T(1-T)
\mu(T)(p(T)) + (1-T)^2 g \in$ LHS and so $(1-T)^2 g\in$ LHS. But
$(1-T)^2 g = (1-T)g -T(1-T)g$. Hence $(1-T)g\in $ LHS. Similarly,
taking $Tq(T)$, we can show that $Tg_1\in$ LHS. This proves the claim.

Now, replacing $\alpha(T)$ by $\alpha(T) + T(1-T) \beta(T)$ for a
suitable $\beta(T) \in P[T]^\ast$, we may assume, by (\ref{cor13}),
that $\alpha(P[T]) = I$ has height $n$. This proves the lemma.
$\hfill \gj$
\end{proof}

\begin{lemma}
Let $A$ be a ring and $J\subset A$ an ideal. Let $B = A_{1+J}$. Then
$J B$ is contained in the Jacobson radical of $B$. 
\end{lemma}

\begin{lemma}\label{555}
Let $A$ be a ring and let $\p_1\subsetneqq \p_2\subsetneqq \p_3$ be a
chain of prime ideals of $A[T]$. Then, we can not have $\p_1\cap A =
\p_2\cap A= \p_3\cap A$.
\end{lemma}

\begin{proof}
Let us assume contrary. By going modulo $\p_1\cap A$, we can assume
that $A$ is a domain and $\p_1\cap A = \p_2\cap A= \p_3\cap A =
0$. Let $S=A-\{0\}$. Then $S^{-1}A[T]$, being principal ideal domain,
is of dimension $1$. But $S^{-1}\p_1 \subsetneqq S^{-1}\p_2
\subsetneqq S^{-1}\p_3$. This is a contradiction. $\hfill \gj$ 
\end{proof}

\begin{lemma}\label{666}
Let $A$ be a Noetherian ring and let $I\subset A[T]$ be an ideal of height
$k$. Then $\hh (I\cap A) \geq k-1$.
\end{lemma}

\begin{proof}
First, we assume that $I=\p$ is a prime ideal. Then, we claim that 
$\hh \p = \hh (\p\cap A)$ if $\p = (\p\cap A)[T]$ and 
$\hh \p = \hh (\p\cap A)+1$ if $\p \supsetneqq (\p\cap A)[T]$. 

Any prime chain $\q_0 \subsetneqq \ldots \subsetneqq \q_r \subsetneqq
(\p\cap A)$ in $A$ extends to a prime chain $\q_0[T] \subsetneqq
\ldots \subsetneqq \q_r[T] \subsetneqq(\p\cap A)[T] \subset \p$ in
$A[T]$. Hence, $\hh \p \geq$ the given values. 
	Now, let $\hh (\p\cap A) =
r$. Then, by the dimension theorem, $(\p\cap A)$ is minimal over an
ideal $\G =(a_1,\ldots,a_r)$. Then $(\p\cap A)[T]$ is
minimal over $\G[T]$, so $\hh ((\p\cap A)[T]) \leq r$. Thus, we have
$\hh \p =\hh (\p\cap A)$ in the case $\p = (\p\cap A)[T]$. 

Now, assume $(\p\cap A)[T] \subsetneqq \p$, say $f\in \p-(\p\cap
A)[T]$. We will be done if we can show that $\p$ is a minimal prime
over $\G[T]+f A[T]$, for then $\hh \p\leq r+1$. Let $\p'$ be a prime
between these. Then $\G\subset\ (p'\cap A)\subset (\p\cap A)$, so
$(\p'\cap A) = (\p\cap A)$, since $(\p\cap A)$ is minimal prime over
$\G$.  In particular, $(\p\cap A)[T]\subsetneqq \p'\subset\p$. By
(\ref{555}), we have $\p=\p'$.

Now, we prove the lemma for any ideal $I\subset A[T]$.  Let $\sqrt I =
\bigcap_1^r \p_i$, where $\p_1,\ldots,\p_r$ are minimal primes over
$I$.  Then $\sqrt {(I\cap A)} = \bigcap_1^r (\p_i\cap A)$. 
	The  prime ideals minimal over
$I\cap A$ occur among $\p_1\cap A,\ldots,\p_r\cap A$. Choose $\p_i$
such that  $\hh (I\cap
A)= \hh (\p_i\cap A)$. Then  $\hh (I\cap
A)= \hh (\p_i\cap A)\geq \hh \p_i -1 \geq
\hh I-1$. This proves the lemma. 
$\hfill \gj$
\end{proof}
\vspace*{.2in}

Roughly, the aim of the next proposition (\ref{prop1.2}) 
is to show that if
$I\subset A[T]$ is an ideal of height $n$ which is the surjective
image of an extended projective module $P[T]$ of rank $n$, then, there
exists an ideal $K$ of $A$ of height $\geq n$ such that $I$ is
comaximal with $KA[T]$ and $I\cap KA[T]$ is generated by $n$
elements. We construct $K$ as follows: 
	we choose $K$ such that
$I(0)\cap K$ is generated by $n$ elements. Further, we choose $K$ to
be comaximal with $I\cap A$. This is achieved using
(\ref{cor13}). Then, using patching argument, we show that $I\cap
KA[T]$ is generated by $n$ elements.

\begin{proposition}\label{prop1.2}
Let $A$ be a Noetherian ring of dimension $n \geq 2$ such that
$(n-1)!$ is invertible in $A$. Let $P$ be a projective $A$-module of
rank $n$ with trivial determinant.  Let $\chi : A \by \sim
\wedge^n(P)$ be an isomorphism. Suppose that $\alpha (T) : P[T] \surj
I$ is a surjection, where $I \subset A[T]$ is an ideal of height
$n$. Then, there exists a homomorphism $\phi : A^n \ra P$, an ideal $K
\subset A$ of height $\geq n$ which is comaximal with $I \cap A$ and a
surjection $\rho(T) : (A[T])^n \surj I \cap KA[T]$ such that: 

$(i)~\phi\otimes A/N$ is an isomorphism, where $N=(I\cap A)$ and 
$\wedge^n(\phi) = u\chi$, where $u = 1$ modulo $I \cap A$. 

$(ii)~ (\alpha(0) \phi)(A^n) = I(0)\cap K.$ 

$(iii)~ \alpha (T) \phi (T) \otimes A[T]/I = \rho (T) \otimes A[T]/I.$

$(iv)~ \rho (0) \otimes A/K = \rho (1) \otimes A/K.$
\end{proposition}

\begin{proof}
First, we show the existence of $\phi$ satisfying $(i)$ and $(ii)$. 

Since $\hh I = n$, we have $\hh N \geq n-1$, by (\ref{666}),
and hence $\dim (A/N) \leq 1$. Since $P$ has trivial determinant,
by Serre's theorem (\ref{Serre}),
there exists an isomorphism $\eta : (A/N)^n \iso P/NP$.
We can alter $\eta$ by an automorphism of
$(A/N)^n$ to obtain an isomorphism 
	$\ol \delta : (A/N)^n \iso P/N P$ 
such that $\wedge^n(\ol \delta) = \ol \chi$, where ``bar'' denotes 
reduction modulo $N$. Let $\delta : A^n \ra P$ be a lift of 
$\ol \delta$. Since $NA_{1+N} \subset J(A_{1+N})$, by (\ref{manoj3}),
$\delta_{1+N} : (A_{1+N})^n\ra P_{1+N}$ is an isomorphism. 

Let $J = I(0)$, where $I(0) = \{F(0)|F(T)\in I\}$ and $\beta =
\alpha(0) : P \surj J$. The equality $\delta(A^n) + NP = P$ shows
that $(\beta \delta)(A^n) + NJ = J$. Since $NJ \subset J^2$, by
(\ref{lem11}), there exists $c \in NJ$ such that $(\beta
\delta)(A^n) + (c) = J$. Therefore, applying (\ref{cor13}) to
$(\beta \delta,c)$, we see that there exists $\gamma \in (A^n)^\ast$
such that the ideal $(\beta \delta + c \gamma)(A^n)$ has height at
least minimum of $n$ and $\hh J$. Since $(\beta \delta + c
\gamma)(A^n) +(c) = J$ and $c \in J^2$, by (\ref{lem11}), $(\beta
\delta + c \gamma)(A^n) = J \cap K$, where $K$ is either $A$ or an
ideal of height $n$ which is comaximal with $(c)$, and hence with $N$
and $J$. 

The next step of the proof is to show that there exists a map $\phi :
A^n \ra P$ which lifts $\ol \delta$ and such that $\beta \phi (A^n)=J\cap
K$. This is achieved by altering $\delta$ by an element of
$\Hom(A^n,NP)$.

Since $c \in NJ, c = \sum a_i d_i$, where $a_i \in N$ and $d_i
\in J$. Any element of $(A^n)^\ast$ of the form $d \gamma$, where $d
\in J$, has its image contained in $J$. Now, since $d_i \in J$ and
$\beta : P \surj J$ is surjective, 
there exists $\nu_i : A^n \ra P$ such that 
$\beta \nu_i = d_i \gamma$. Let $\nu = \sum a_i \nu_i$. Then
	$c \gamma = \Sigma a_i d_i \gamma 
	= \sum a_i \beta \nu_i = \beta \nu$, 
where $\nu = 0$ modulo $N$. Let $\phi = (\delta +
\nu)$. Then $\phi$ is also a lift of $\ol \delta$ and hence 
$\wedge^n \phi = u \chi$, where $u = 1$ modulo $N$. Moreover $\phi$
has the property that

$\beta \phi (A^n) = (\beta \delta + \beta \nu)(A^n) =
(\beta \delta + \beta \sum a_i \nu_i)(A^n) = (\beta
\delta + \sum a_i \,\beta \nu_i)(A^n)$

$= (\beta \delta + \sum a_i d_i \gamma)(A^n) =
(\beta \delta + c \gamma)(A^n) = J \cap K.$  
This proves $(i)$ and $(ii)$. 

Since $K + N = A$, we have $I + KA[T] = A[T]$. Let $I^\prime = I \cap
KA[T]$. Then $I^\prime(0) = J \cap K$ and $I^\prime/{I^\prime}^2 =
I/I^2 \oplus KA[T]/K^2A[T]$. 

Since $N A_{1+N}\subset J(A_{1+N})$, by (\ref{manoj3}), 
$\phi_{1+N} : (A_{1+N})^n \iso P_{1+N}$
is an isomorphism. Further,
$I^\prime_{1+N} = I_{1+N}$, as $K\cap (1+N) \neq \varnothing$. 
	Therefore, the
map $(\alpha(T)\phi(T))_{1+N} : (A_{1+N}[T])^n \ra
I^\prime_{1+N}$ is surjective, where $\phi(T) = \phi \otimes A[T]$. 
Hence, there exists $a \in N$ such that the map 
	$(\alpha(T) \phi(T))_{1+a} : (A_{1+a}[T])^n  \ra
I^\prime_{1+a}$ is surjective. We can assume that $1+a \in K$, 
as $N+K = A$.
Since $a \in N \subset I$, we have $I^\prime_a = KA_a[T]$.
Therefore, we get a surjection $(\beta \phi) \otimes A_a[T] :
(A_a[T])^n \surj I^\prime_a$. 

The elements $(\beta \phi) \otimes A_{a(1+aA)}[T]$ and
$(\alpha(T) \phi(T))_{a(1+aA)}$ are surjections from \\
$(A_{a(1+aA)}[T])^n \surj A_{a(1+aA)}[T]$ and as $\alpha(0) = \beta$,
they are equal modulo $(T)$. Note that $\dim A_{a(1+aA)} \leq n-1$
(for if $\m$ is any maximal ideal of $A$, then either $a\in \m$ or
$(1+aA)\cap \m \neq \varnothing$).  The kernels of the surjections
$(\beta\phi)\otimes A_{a(1+aA)}$ and 
	$(\alpha(T) \phi(T))_{a(1+aA)}$ 
are stably free modules given by unimodular rows.
These  are extended from
$A_{a(1+aA)}$ by (\ref{Ravi}), since $(n-1)$! is invertible in $A$.
By (\ref{lem9}), there exists an $\sigma(T)\in 
\GL_n(A_{a(1+aA)}[T])$ such that $\sigma(0) =\id$ and $(\alpha(T)
\phi(T))_{a(1+aA)} \sigma(T) = (\beta \phi) \otimes
A_{a(1+aA)}[T]$. 

Let $(1+aa')=(1+a)(1+aa{''})$, 
where $a{''}\in A$ is chosen so that the following properties hold:
 
(1) det $(\sigma(T))$ is a unit belonging to
$A_{a(1+aa')}[T]$ and 

(2) $(\alpha(T)
\;\phi(T))_{a(1+aa')} \;\sigma(T) = (\beta \phi) \otimes
A_{a(1+aa')}[T]$. 

Let $b = (1+aa')$. Then $\sigma(T)$ $\in 
\GL_n(A_{ab}[T])$ with $\sigma (0) = \id$.  Since $\sigma(0)= \id$,
by lemma (\ref{lem10}), we see that
$\sigma(T)= \tau(T)_a \theta(T)_b$, where $\tau(T)$ is an
$A_b[T]$-automorphism of $(A_b[T])^n$ such that $\tau(0) =\id$  and
$\tau =\id$  modulo $(a)$ and $\theta(T)$ is an $A_a[T]$-automorphism
of $(A_a[T])^n$ such that $\theta(0) =\id$  and $\theta =\id$  modulo
$(b)$.

We have $((\alpha(T) \phi(T))_b \tau(T))_a = 
(((\beta \phi) \otimes A_a[T]) (\theta(T))^{-1})_b$. Hence,
the surjections $(\alpha(T) \phi(T))_b.\tau(T) : (A_b[T])^n \surj
I^\prime_b$ and  $((\beta \phi) \otimes A_a[T]) (\theta(T))^{-1} :
(A_a[T])^n \surj I^\prime_a$ patch to yield a surjection $\rho(T) :
(A[T])^n \surj I^\prime$ such that $\rho(0) = \beta \phi$.

Since $\theta(T) = \id$ modulo the
ideal $(b)$ and $b\in K$,
 it follows from the construction of $\rho(T)$ that
$\rho(T)(e_i)-(\beta\phi\otimes A[T])(e_i) \in K^2A[T]~ \forall i$ 
(where $e_i$ are the
coordinate functions of $A[T]^n$). Hence
$\rho(0) \otimes A/K = \rho(1) \otimes A/K$.
Further, using the fact that $\tau(T) =\id$ modulo the ideal
$(a)$, we see that
$(\alpha(T) \phi(T)) \otimes A[T]/I = \rho(T)\otimes A[T]/I$.
This proves $(iii)$ and $(iv)$ and hence, the proposition is proved.
\Com  $$\xymatrix{ 
	(A[T])^n \ar@{->}[rr]\ar@{->}[dd]\ar@{->>}[dr]^{\rho(T)} &
	& (A_b[T])^n \ar@{->}'[d][dd] 
	\ar@{->>}[dr]^{(\alpha(T)\phi(T))_b\tau(T)}   \\ 
	& I^\prime \ar@{->}[rr]\ar@{->}[dd] & 
	& I^\prime_b \ar@{->}[dd] 	\\ 
	(A_a[T])^n \ar@{->}'[r][rr] 
	\ar@{->>}[dr]_{((\beta\phi)\otimes A_a[T])
	(\theta(T))^{-1}} & & (A_{ab}[T])^n \ar@{->}[dr]	\\ 
	& I^\prime_a \ar@{->}[rr] & & I^\prime_{ab} 
	}$$ 

$\hfill \gj$
\end{proof}

\begin{rem}\label{300}
Now, we discuss (\ref{prop1.2}) in the context of the Euler class
group $E(A)$. Let $I \subset A[T]$ be an ideal of height $n$, where
$A$ is of dimension $n$. Suppose $I$ is a surjective image of a projective
$A[T]$-module $P[T]$, where $P$ is a projective $A$-module of rank $n$
having trivial determinant. Further, assume that $I(0)$ and $I(1)$ are
ideals of height $n$. Now, tensoring the surjection from $P[T]$ to $I$  and
$\phi : A^n \ra P$ given in (\ref{prop1.2}) with $A[T]/I$ and
composing, we get a `local
orientation' $w(T)$ of $I$, i.e. a surjection 
$w(T) : (A[T]/I)^n \surj I/I^2$, which in turn gives rise to local
orientations $w(0) : (A/I(0))^n \surj I(0)/I(0)^2$ and 
$w(1) : (A/I(1))^n \surj I(1)/I(1)^2$ of $I(0)$ and $I(1)$
respectively. 

The gist of (\ref{prop1.2}) is that there exists an ideal $K \subset
A$ of height $n$ and a local orientation $w_K$ of $K$, which is the
class of the surjection $\rho(0) \otimes A/K : (A/K)^n \surj K/K^2$,
such that
$$(I(0),w(0)) + (K,w_K) =0 = (I(1),w(1)) + (K,w_K) $$
in $E(A)$. Therefore, $(I(0),w(0)) = (I(1),w(1))$ in $E(A)$. 

In case when $K = A$, we get $\rho(T) : (A[T])^n \surj I$, hence $I$ is
generated by $n$ elements. Hence $w(0)$ and $w(1)$ are global
orientations of $I(0)$ and $I(1)$ respectively. So 
$(I(0),w(0)) = 0 = (I(1),w(1))$ in $E(A)$. 

Let $P$ be a projective $A$-module of rank $n$ with trivial
determinant. Let $\chi : A \iso \wedge^n(P)$ be an isomorphism.
We call $\chi$ an {\it orientation of} $P$ and $\chi(1)$ a generator
of $\wedge^n(P)$. We write $\chi$ for $\chi(1)$. To the pair
$(P,\chi)$, we associate an element $e(P,\chi)$ of $E(A)$ as
follows:

Let $\lambda : P \surj J_0$ be a surjection, where $J_0 \subset A$ is
an ideal of height $n$. Let ``bar'' denote reduction modulo $J_0$. We
obtain an induced surjection 
$\ol \lambda : P/J_0P \surj J_0/J_0^2$. We choose an isomorphism 
$\ol \gamma : (A/J_0)^n \iso P/J_0P$ such that 
$\wedge^n(\ol \gamma) = \ol \chi$. 

Let $w_{J_0}$ be a local orientation of $J_0$ given by $\ol \lambda
\ol \gamma : (A/J_0)^n \surj J_0/J_0^2$.  Let $e(P,\chi)$ be the
image in $E(A)$ of the element $(J_0,w_{J_0})$ of $G$. We say that
$(J_0,w_{J_0})$ is {\it obtained} from the pair $(\lambda,\chi)$. We
show that the assignment sending the pair $(P,\chi)$ to the element
$e(P,\chi)$ of $E(A)$ is well defined. 

Let $\mu : P \surj J_1$ be another surjection, where $J_1 \subset A$
is an ideal of height $n$. Then, by (\ref{lem1.1}), there exists an
ideal $I$ of $A[T]$ of height $n$ and a surjection $\alpha(T) : P[T]
\surj I$ such that $\alpha(0) = \lambda$ $\alpha(1) = \mu$, $I(0) =
J_0$ and $I(1) = J_1$. 

Then, from the above discussion, we have $(J_0,w_{J_0}) =
(J_1,w_{J_1})$ in $E(A)$,
where $w_{J_0} = \ol \lambda\ol \gamma$ and 
$w_{J_1} = \ol \mu \ol \gamma$. Hence $e(P,\chi)$ does not depend on
the choice of the surjection. 

Now, let $\lambda : P \surj J_0$ be a surjection, where $J_0 \subset
A$ is an ideal of height $n$. If $\ol \delta : (A/J_0)^n \iso P/J_0P$
is another isomorphism such that $\wedge^n(\ol \delta) = \ol \chi$,
then $\ol\delta$ and $\ol\gamma$ differ by an element of
$\SL_n(A/J_0)$. Hence, there exists an $\ol \sigma \in 
\SL_n(A/J_0)$ such that $\ol \delta = \ol \sigma \ol \gamma$. 
This shows that $e(P,\chi)$ does not depend on the choice of
$\ol\gamma$ and proves that $e(P,\chi)$ is well defined.  We
define the {\it Euler class } of $(P,\chi)$ to be $e(P,\chi)$.
$\hfill \gj$
\end{rem}

\begin{corollary}\label{cor2.3}
Let $A$ be a Noetherian ring of dimension $n \geq 2$. Let $P$ be a
projective $A$-module of rank $n$ with trivial determinant and $\chi$
be an orientation of $P$. Let $J \subset A$ be an ideal of height $n$
such that $J/J^2$ is generated by $n$ elements. Let $w_J$ be a local
orientation of $J$. Suppose that $e(P,\chi) = (J,w_J)$ in
$E(A)$. Then, there exists a surjection $\alpha : P \surj J$ such that
$(J,w_J)$ is obtained from $(\alpha,\chi)$.
\end{corollary}

\begin{proof}
We can regard $w_J$ as a surjection : $(A/J)^n \surj J/J^2$. We choose
an isomorphism $\lambda : P/JP \iso (A/J)^n$ such that
$\wedge^n(\lambda) = (\ol \chi)^{-1}$, where ``bar'' denotes modulo $J$.
Consider the surjection $\ol \alpha = w_J \lambda : P/JP \surj
J/J^2$. By (\ref{cor14}), there exists an ideal $J^\prime \subset A$
and a surjection $\beta : P \surj J \cap J^\prime$ such that: 

$(i) J + J^\prime = A, (ii)  \beta \otimes A/J = \ol \alpha,$
and $(iii)$ height $(J^\prime) \geq n$. 

If $J^\prime = A$, then $\beta : P \surj J$ is such that $\beta \otimes
A/J = \ol \alpha$. Hence $\beta$ satisfies the required
property. Otherwise, if $\hh J^\prime = n$, then, we have 
$e(P,\chi) = (J,w_J) + (J^\prime,w_{J^\prime})$ in $E(A)$ (where
$w_{J'}$ is obtained using $P$).  
	By the assumption of the theorem, 
$e(P,\chi) = (J,w_J)$ in $E(A)$. Hence, we have
$(J^\prime,w_{J^\prime}) = 0$ in $E(A)$. Therefore, by
(\ref{theo2.2}), there exists a surjection 
$\gamma : A^n \surj J^\prime$ such that 
$w_{J^\prime} = \gamma \otimes A/J^\prime$. Now, applying the
subtraction principle (\ref{manoj1}), we get a surjection 
$\alpha : P \surj J$ such that $(J,w_J)$ is obtained from the pair
$(\alpha,\chi)$.  This proves the corollary.  $\hfill \gj$
\end{proof}

\begin{corollary}\label{cor2.4}
Let $A$ be a Noetherian ring of dimension $n \geq 2$. Let $P$ be a
projective $A$-module of rank $n$ with trivial determinant and let $\chi$
be an orientation of $P$. Then $e(P,\chi) = 0$ if and only if $P$ has
a unimodular element. In particular, if the determinant of $P$ is
trivial and $P$ has a unimodular element, then every generic section
ideal $J$ of $P$ (i.e. an ideal $J$ of height $n$ which is a
surjective image of $P$) is generated by $n$ elements.
\end{corollary}

\begin{proof}
Let $\alpha : P \surj J$ be a surjection, where $J$ is an ideal of
height $n$. Let $e(P,\chi) = (J,w_J)$ in $E(A)$, where $(J,w_J)$ is
obtained from the pair $(\alpha,\chi)$.  First, assume that $e(P,\chi)
= 0$. Then $(J,w_J) = 0$ in $E(A)$. Hence, by (\ref{theo2.2}), there
exists a surjection $\beta : A^n \surj J$ such that $w_J = \beta
\otimes A/J$.  Now, applying (\ref{manoj2}), we see that $P$ has a
unimodular element.

Now, we assume that $P$ has a unimodular element, i.e.  $P = Q \oplus
A$.  Then $\alpha = (\theta,a)$ as an element of $P^\ast = Q^\ast
\oplus A$. By performing an elementary automorphism of $P$,
i.e. replacing $\theta$ by $\theta+a\theta'$, we may assume by
(\ref{cor13}), that $\hh \theta(Q) = n-1$. Let $K = \theta(Q)$.
Note that, since determinant of $Q$ is trivial, without loss of
generality, we may assume that $\chi$ is induced by an isomorphism
$\chi^\prime : A \iso \wedge^{n-1}(Q)$.

Since $\dim (A/K) \leq 1$, there exists an isomorphism 
	$\gamma : (A/K)^{n-1} \iso Q/KQ$ such that 
	$\wedge^{n-1}(\gamma) = \chi^\prime$ modulo $K$.  
The surjection $(\theta \otimes A/K)\gamma :
	(A/K)^{n-1} \surj K/K^2$ 
can be lifted to a map $\delta : A^{n-1} \ra K$ such that 
	$\delta(A^{n-1}) + K^2 = K$. 
Let $\delta (A^{n-1}) = K^\prime$. Then, since 
$K^\prime + K^2 = K$, by (\ref{lem11}), $K = K^\prime + (e)$, 
with $e \in K^2$ and $e^2 -e \in K^\prime$.
Therefore, by (\ref{Mohan}), $J = K+(a) = K^\prime + (b)$, where $b
= e + (1-e)a$.  Now, consider the surjection $(\delta,b) : A^n \surj
J$. As $e\in K^2$, we have that $w_J$ is obtained by tensoring the
surjection $(\delta,b)$ with $A/J$. Hence, by definition, $e(P,\chi) =
0$ in $E(A)$. This proves the corollary.  $\hfill \gj$
\end{proof}

\begin{lemma}\label{lem3.5}
Let $A$ be a Noetherian ring of dimension $n \geq 2$. Let $J \subset
A$ be an ideal of height $n$ such that $J/J^2$ is generated by $n$
elements. Let $w_J$ be a local orientation of $J$. Suppose that
$(J,w_J) \neq 0$ in $ E(A)$. Then, there exists an ideal $J_1$ of
height $n$ which is comaximal with $J$ and a local orientation
$w_{J_1}$ of $J_1$ such that $(J,w_J) + (J_1,w_{J_1}) = 0$ in
$E(A)$. Further, given any element $f \in A$ such that $\hh fA = 1$,
$J_1$ can be chosen with the additional property that it is comaximal
with $(f)$.
\end{lemma}

\begin{proof}
Let $\alpha : (A/J)^n \surj J/J^2$ be a surjection corresponding to
$w_J$. Then, by (\ref{cor14}), there exists an ideal $J_1$ of height
$\geq n$ which is comaximal with $fJ$ and a surjection 
$\beta : A^n \surj J\cap J_1$ such that $\beta \otimes A/J = \alpha$.
Since $(J,w_J) \neq 0$ in $E(A), J_1$ is a proper ideal of height
$n$. Let $w_{J_1}$ be the local orientation of $J_1$ induced by
$\beta$. Then $(J,w_J) + (J_1,w_{J_1}) = 0$ in $E(A)$.  $\hfill
\gj$
\end{proof}

\begin{lemma}\label{200}
Let $A$ be a Noetherian ring of dimension $n\geq 2$. Then, any element
of the Euler class group $E(A)$ is of the form $(J,w_J)$, where $J$ is
an ideal of $A$ of height $n$ such that $J/J^2$ is generated by $n$
elements and $w_J$ is a local orientation of $J$.
\end{lemma}

\begin{proof}
First, we show that if $(J,w_J) \in E(A)$, $(J,w_J)\neq 0$ and $f \in
A$ such that $\hh f A = 1$, then $- (J,w_J) = (J_1,w_{J_1}) \in E(A)$
with $J_1 + f J = A$. By (\ref{lem3.5}), there exists an ideal $J_1$
of height $n$ which is comaximal with $f J$ and a local orientation
$w_{J_1}$ of $J_1$ such that $(J,w_J) + (J_1,w_{J_1}) = 0$ in
$E(A)$. Hence $- (J,w_J) = (J_1,w_{J_1})$.  
	Therefore, any element $z\in
E(A)$ is of the form $z = \sum_1^r (J_i,w_{J_i})$. It is enough to
show that any element $z \in E(A)$ of the form $z = (J_1,w_{J_1}) +
(J_2,w_{J_2})$ can be written as $z = (J',w_{J'})$ in $E(A)$ for some
ideal $J'$ of height $n$ and a local orientation $w_{J'}$ of
$J'$. Then, by induction, the result will follow.  

Without loss of generality, we may assume that both $(J_1,w_{J_1})$ and
$(J_2,w_{J_2})$ are not zero in $E(A)$. 
Choosing $f\in J_1\cap J_2$ and $\hh (f)=1$ and applying $(\ref{lem3.5})$,
we see that there exists an ideal $J_0$ of height $n$ which
is comaximal with $J_1$ and $J_2$ and a local orientation $w_{J_0}$ of
$J_0$ such that $(J_1,w_{J_1}) + (J_0,w_{J_0}) = 0$ in $E(A)$. 
	Hence,
we have $z = - (J_0,w_{J_0}) + (J_2,w_{J_2})$. Now, $-(J_0,w_{J_0}) =
(K_1,w_{K_1})$ in $E(A)$, where $K_1$ is comaximal with $J_0$ and
$J_2$. We have $z = (K_1,w_{K_1}) + (J_2,w_{J_2})$, where
$K_1 + J_2 = A$. Therefore, we have
$z= (J',w_{J'})$ in $E(A)$, where $J' = K_1 \cap J_2$ and $w_{J'}$ is
a local orientation of $J'$ induced from $w_{K_1}$ and $w_{J_2}$. This
proves the lemma. $\hfill \gj$
\end{proof}
\medskip

Let $A$ be a Noetherian ring of dimension $n \geq 2$. Let $\G$ be the
nil radical of $A$. Let ``bar'' denote modulo $\G$. Let $G(A)$ be the
free abelian group on the set $(\n,w_{\n})$, where $\n$ is $\M$-primary
ideal of height $n$ such that $\n/\n^2$ is generated by $n$ elements
and $w_{\n}$ is a local orientation of $\n$. Similarly, we define
$G(\ol A)$. If $\n$ is $\M$-primary
ideal, then $\ol \n = (\n + \G)/\G$ is also $\ol \M$-primary and if
$\n/\n^2$ is generated by the images of 
$(a_1,\ldots,a_n)$, then $\ol \n/(\ol \n)^2$
is also generated by the images of $(a_1,\ldots,a_n)$. Hence, if 
$(\n,w_{\n}) \in G(A)$, then  $(\ol \n,w_{\ol \n}) \in G(\ol A)$. 

Let $J\subset A$ be an ideal of height $n$ with primary decomposition
as $J = \cap \n_i$, where $\n_i$ is $\M_i$-primary, $\M_i$ a maximal
ideal of $A$. Then $\ol J = (J + N)/N = J/J\cap N \subset \ol A$ is an
ideal of height $n$ with primary decomposition as $\ol J = \cap \ol
\n_i$. The following
diagram is commutative 
\Com
$$\xymatrix{ 
	J/J^2 \ar [r] \ar [d] & \ooplus \n_i/ \n_i^2 \ar [d]  \\
	\ol J/ \ol {J^2} \ar [r] & \ooplus \ol \n_i/\ol \n_i^2 
	}$$ 
Hence, any surjection $w_J : (A/J)^n \surj J/J^2$ induces a surjection
$\ol w_J : (\ol A/\ol J)^n \surj \ol J/\ol J^2 = (J + N)/ (J^2 +
N)$. From the above discussion, it follows that the assignment sending
$(J,w_J)$ to $(\ol J,\ol w_J)$ gives rise to a group homomorphism
$\Phi : E(A) \ra E(\ol A)$.
\medskip

As a consequence of (\ref{theo2.2}), we have the following:

\begin{corollary}\label{cor2.5}
The homomorphism $\Phi : E(A) \ra E(\ol A)$ is an isomorphism.
\end{corollary}

\begin{proof}
Let $\ol w_{\ol J} : (\ol A/\ol J)^n \surj \ol J/\ol J^2$ be a
surjection. Let $J \supset \G$ be an ideal of height $n$ such that $J/\G
= \ol J$. Then $\ol w_{\ol J}$ can be considered as a surjective map
from $(A/J)^n$ to $J/(J^2+\G)$. Let $\alpha : A^n \ra J$ be an
$A$-linear map which is a lift of $\ol w_{\ol J}$. Let 
$\alpha(A^n) = (f_1,\ldots,f_n)$. Then 
$(f_1,\ldots,f_n) + J^2 + \G = J$. By (\ref{lem11}), there exists
an element $e \in J^2$ such that $(f_1,\ldots,f_n) + \G + Ae = J$, 
and $e(1 - e) \in ((f_1,\ldots,f_n) + \G)$. 
Then $\ol e\in A/((f_1,\ldots,f_n)+\G)$ is an idempotent.

Since $\G$ is a nilpotent ideal and idempotent elements can be lifted
modulo a nilpotent ideal, we can lift $\ol e\in
A/((f_1,\ldots,f_n)+\G)$ to an idempotent element of
$A/(f_1,\ldots,f_n)$. Let $f \in A$ be such that $\ol f \in
A/(f_1,\ldots,f_n)$ is a lift of $\ol e\in A/((f_1,\ldots,f_n)+\G)$,
i.e. $f - f^2 \in (f_1,\ldots,f_n)$ and 
	$f - e \in ((f_1,\ldots,f_n) + \G)$. 
Let $J_1 = (f_1,\ldots,f_n,f)$. Then $(f_1,\ldots,f_n) + J_1^2 =
J_1$ and $J_1 + \G =(f_1,\ldots,f_n,f) +\G = (f_1,\ldots,f_n,e)+\G =
J$ (as $f-e\in ((f_1,\ldots,f_n)+\G)$).  
	Let ``bar'' denote modulo
$\G$. Then $\ol J_1 = \ol J$ and $\ol J_1/\ol J_1^{2} = (\ol
f_1,\ldots,\ol f_n) $. Hence $(\ol f_1,\ldots,\ol f_n)= \ol J/\ol
J^2 = J/(J^2+\G)$.  This implies that $\ol w_{\ol J}$ is induced
from $\ol \alpha$. Hence, the map $\Phi $ is surjective. 

Now, we prove that the map $\Phi$ is injective. By (\ref{200}), every
element of $E(A)$ is of the form $(J,w_J)$. Hence, it is enough to
prove that for $(J,w_J) \in E(A)$ ( where $J \subset A$ is an ideal of
height $n$ and $w_J : (A/J)^n \surj J/J^2$ is a local orientation of
$J$), if the image of $(J,w_J) = 0$ in $E(\ol A)$, then $(J,w_J) = 0$
in $E(A)$. Assume that the image of $(J,w_J) = (\ol J,\ol w_{\ol J})=0$
 in $E(\ol A)$. Then,
by (\ref{theo2.2}), $\ol w_{\ol J}$ is a global orientation of $J$,
i.e. there exists a surjection $\gamma : (A/\G)^n \surj (J+\G)/\G$
such that $\ol w_{\ol J} = \gamma \otimes \ol A/\ol J$.

We are given surjections $\alpha : A^n \surj J/J^2$ (which is obtained
by $w_J$ by composing with the natural map $A^n \surj (A/J)^n$) and
$\beta : A^n \surj (J+\G)/\G = J/(J \cap \G)$ (which is obtained by
$\gamma$ by composing with natural map $A^n \surj (A/\G)^n$) such that
they induce the same surjective map from $A^n$ to $J/(J^2+(J\cap
\G))$.

Since $J/(J^2\cap \G)$ is the fiber product of $J/J^2$ and 
$J/(J\cap \G)$ over $J/(J^2,J\cap \G)$, 
	$\alpha$ and $\beta$ patch to
yield a map $\delta : A^n \ra J/(J^2\cap \G)$.
\Com
$$\xymatrix{ 
	J/(J^2\cap \G) \ar [r] \ar [d] & J/J^2 \ar [d]\\ 
	J/(J\cap \G) \ar [r] & J/(J^2,J\cap \G) 
	}$$ 
Let $\theta : A^n \ra J$ be a lift of $\delta$. Then $\theta$ is a
lift of $\alpha$ and $\beta$. Hence, we have

$(i)\; \theta(A^n) + J^2 = J$, and 
$(ii)\; \theta(A^n) + (J\cap \G) = J$.
Since $\G$ is nilpotent,  $\sqrt{\theta(A^n)}= \sqrt J$ 
are same by $(ii)$.
Hence $\theta (A^n) = J$ by $(i)$ and (\ref{manoj10}).

Since $\theta$ is a lift of $\alpha$, we get $(J,w_J) = 0$ in
$E(A)$. Hence $\Phi$ is injective and hence  an
isomorphism. This proves the corollary. $\hfill \gj$
\end{proof}


\chapter{Some Results on $E(A)$}
If $A$ is an affine domain of dimension $n$ over an algebraically
closed field and $P$ is a projective $A$-module of rank $n$ and
trivial determinant, then it follows from a
result of Mohan Kumar (\ref{100}) that if $P$ maps onto an
ideal $J$ of height $n$ which is generated by $n$ elements, then $P$
has a unimodular element and hence all its generic section ideals
(i.e. ideals of height $n$ which are surjective image of $P$) are
generated by $n$ elements (\ref{cor2.4}).  
	But this is not
necessarily true if the base field is not algebraically closed. For
example, all the {\it reduced} generic section ideals of the tangent
bundle of the real 2-sphere are generated by $2$ elements
(\cite{B-RS2}, (5.6-i)). 
There are however {\it non-reduced} generic section
ideals of the tangent bundle which are not generated by $2$ elements
(\cite{B-RS1}, (5.2)).

This phenomenon is explained by the result (\ref{lem3.9}) of this
chapter, which asserts that for any $n$-dimensional real affine
domain, a projective module of rank $n$ with trivial determinant, all
of whose generic section ideals are generated by $n$ elements, has a
unimodular element. To prove this result, we first prove some
lemmas.

Let $A$ be a Noetherian ring of dimension $n \geq 2$. Let $J \subset
A$ be an ideal of height $n$ and $w_J : (A/J)^n \surj J/J^2$ be a
local orientation of $J$. Let $\ol b \in A/J$ be a unit. Let 
$\sigma : (A/J)^n \iso (A/J)^n$ be an automorphism with det
$(\sigma) = \ol b$. Then $w_J \sigma$ is another local orientation
of $J$, which we denote by $\ol b w_J$.

\begin{lemma}\label{lem3.0}
Let $A, J$ be as above. Let $w_J$ and $\wt w_J$ be two local
orientations of $J$. Then $\wt w_J = \ol b w_J$ for some unit 
$\ol b \in A/J$.
\end{lemma}

\begin{proof}
We have two surjections $w_J : (A/J)^n \surj J/J^2$ and 
$\wt w_J : (A/J)^n \surj J/J^2$.  We will define a map 
$\psi : (A/J)^n \ra (A/J)^n$ such that $w_J \psi = \wt w_J$.  

Let $\{e_i, i = 1,\ldots,n\}$ be a basis of $(A/J)^n$. Given $\wt w_J
(e_i) = \ol {a_i}, w_J (e_i) = \ol {b_i}$. Let $\ol {a_i} = \sum 
\ol {c_{ij}} \ol {b_j}$. Define $\psi (e_i) = \sum_j 
c_{ij}e_j.$ Then $ w_J  \psi = \wt w_J$.  Now, by (\ref{lem2})
$\psi$ is an isomorphism.  Let det $(\psi) = \ol b$.
Then $\wt w_J = \ol b w_J$.  This proves the
lemma.  $\hfill \gj$
\end{proof}

\begin{lemma}\label{lem3.2}
Let $A$ be a ring and let $J \subset A$ be an ideal which is generated by
two elements $a_1, a_2$. Let $a \in A$ be a unit modulo $J$ and $b \in
A$ be such that $ab = 1$ modulo $J$. Suppose that the unimodular row
$(b,a_2,-a_1)$ is completable to a matrix in $\SL_3(A)$. Then, there
exists a matrix $\tau \in {\rm M}_2(A)$ with {\rm det} $(\tau) = a$ modulo
$J$ such that $[a_1,a_2]\tau^t = [b_1,b_2]$, where $b_1,b_2$ generate
$J$.
\end{lemma}

\begin{proof}
Choose a completion $\sigma \in \SL_3(A)$ of the unimodular row
$(b,a_2,-a_1)$. Suppose that second and third rows of $\sigma$ are
$(d,\lambda_{11},\lambda_{12})$ and $(e,\lambda_{21},\lambda_{22})$
respectively.

Let $\gamma : A^3 \surj J$ be a surjection given by 
$\gamma(e_1) = 0, \gamma(e_2) = a_1, \gamma(e_3) = a_2.$ The
vectors $(b,a_2,-a_1), (d,\lambda_{11},\lambda_{12})$ and
$(e,\lambda_{21},\lambda_{22})$ generate $A^3$. Hence, their images
under $\gamma$ generate $J$. Hence 
$J = (b_1,b_2)$, where $b_1 = a_1\lambda_{11} + a_2\lambda_{12}$
and  $b_2 = a_1\lambda_{21} + a_2\lambda_{22}$.  
Let $\tau = (\lambda_{ij}) \in {\rm M}_2(A)$. Since $\sigma \in
\SL_3(A)$ and $a_1,a_2 \in J$, we get det$(\tau) = a$ modulo $J$.
Further, $[a_1,a_2]\tau^t = [b_1,b_2]$.  This proves the lemma. $ \hfill
\gj$
\end{proof}

\begin{lemma}\label{lem3.3}
Let $A$ be a Noetherian ring of dimension $n \geq 2$, $J \subset A$ an
ideal of height $n$ and $w_J : (A/J)^n \surj J/J^2$ a
surjection. Suppose that $w_J$ can be lifted to a surjection $\alpha :
A^n \surj J$.  Let $a \in A$ be a unit modulo $J$. Let $\theta$ be an
automorphism of $(A/J)^n$ with determinant $\ol {a^2}$. Then, the
surjection $w_J \theta : (A/J)^n \surj J/J^2$ can be lifted to a
surjection $\gamma : A^n \surj J$.
\end{lemma}

\begin{proof}
Let $P = A^{n-2}$ and $\alpha : P \oplus A^2 \surj J$ be map defined
by $(a_3,\ldots,a_n,a_1,a_2)$ such that $w_J = \alpha\otimes A/J$.
 Let $J^\prime = (a_3,\ldots,a_n)$ and
let ``tilde'' denote reduction modulo $J^\prime$. Then 
$\wt \alpha : \wt A^2 \surj \wt J$ is defined by 
$\wt \alpha (0,1,0) = \wt {a_1}, \wt \alpha (0,0,1) = \wt {a_2}$. 

Since $a\in A$ is a unit modulo $J$, there exists an element $b\in A$
such that  
$a b = 1 (\mbox{mod}J)$. Then $(\wt {b^2},\wt{a_2},- \wt {a_1})$
is a unimodular row, which is completable to an invertible matrix in
$\SL_3(\wt A)$, by (\ref{Suslin}). Hence, by (\ref{lem3.2}), there
exists a matrix $\wt \tau  \in {\rm M}_2(\wt A)$ such that 
$[\wt {a_1}, \wt {a_2}]\wt \tau = [\wt {b_1},\wt {b_2}]$, where 
$\wt J = (\wt {b_1}, \wt {b_2})$ and det 
$(\wt\tau) = \wt {a^2} (\mbox{mod}J)$.

Define a surjection $\gamma^\prime : P\oplus A^2 \surj J$ by setting
$\gamma^\prime = \alpha$, on $P$ and $\gamma^\prime(0,1,0) = b_1,
\gamma^\prime (0,0,1) = b_2$. Define 
$\theta^\prime : (A/J)^n \iso (A/J)^n$ by $\left(\begin{matrix} I_{n-2} & 0\\
0 & \ol \tau \end{matrix} \right)$.

Then, det $\theta^\prime = \ol a^2$. It follows that
$w_J\theta^\prime = \gamma^\prime \otimes A/J$. Hence $w_J \theta'$
can be lifted to a surjection $A^n\surj J$.

Since $\dim (A/J) = 0$, we have $\SL_n(A/J) = \mbox{E}_n(A/J)$ and
the canonical map from $\SL_n(A) \surj \SL_n(A/J)$ is
surjective. Now, since det $\theta=$ det $\theta'=\ol {a^2}$, it follows
that  $w_J \theta$ can also be lifted to a surjection
$A^n\surj J$.
$\hfill \gj$
\end{proof}

\begin{lemma}\label{lem3.4}
Let $A$ be a Noetherian ring of dimension $n \geq 2$, $J \subset A$ an
ideal of height $n$ and $w_J$  a local orientation of $J$. Let $\ol
a \in A/J$ be a unit. Then $(J,w_J) = (J,\ol {a^2}w_J)$ in $E(A)$.
\end{lemma}

\begin{proof}
If $(J,w_J) = 0$ in $E(A)$, then, by (\ref{theo2.2}), $w_J$ can be
lifted to a surjection from $A^n \ra J$. By (\ref{lem3.3}),
$\ol{a^2}w_J$ can also be lifted to a surjection from $A^n \ra
J$. Hence $(J,\ol {a^2}w_J) = 0$ in $E(A)$. Hence, the result follows
in this case.

Now, assume that $(J,w_J) \neq 0$ in $E(A)$, where $w_J : (A/J)^n \surj
J/J^2$ is a surjection.  By (\ref{cor14}), there exists an ideal
$J_1$ of height $n$ which is comaximal with $J$ and a surjection
$\alpha : A^n \surj J\cap J_1$ such that $\alpha\otimes A/J = w_J$
(If $J_1 = A$, then $(J,w_J) = 0$ in $E(A)$). Let $w_{J_1} =
\alpha \otimes A/J_1$.

Let $x + y = 1, x \in J, y\in J_1$. If we set $b = a^2(1-x)+x$,
then $b = a^2$ (mod $J$) and $b = 1$ (mod $J_1$).  Applying
(\ref{lem3.3}), there exists a surjection 
$\gamma : A^n \surj J\cap J_1$ such that $\;\gamma \otimes A/J = \ol
{a^2} w_J,\;\;\gamma \otimes A/J_1 = w_{J_1}$.

From the surjection $\alpha$, we get $(J,w_J) + (J_1,w_{J_1}) = 0$ in
$E(A)$ and from the surjection $\gamma$, we get 
$(J,\ol {a^2}w_J) + (J_1,w_{J_1}) = 0$ in $E(A)$.  Thus 
$(J,w_J) = (J,\ol {a^2} w_J)$ in $E(A)$.  This proves the
lemma. $\hfill \gj$
\end{proof}

\begin{lemma}\label{0.0}
Let $A$ be a Noetherian ring of dimension $n\geq 1$ and let $J\subset A$
be an ideal of height $n$. Let $f\neq 0 \in A$ such that $J A_f$ is a
proper ideal of $A_f$.  Assume $J A_f = (a_1,\ldots,a_n)$, where
$a_i\in J$. Then, there exists $\sigma\in \SL_n(A_f)$
such that $[a_1,\ldots,a_n]\sigma = [b_1,\ldots,b_n]$, where $b_i\in
J$ and $\hh (b_1,\ldots,b_n) = n$.
\end{lemma}

\begin{proof}
Let $I$ be the set $\{ \sigma \in \SL_n(A_f) : [a_1,\ldots,a_n]
\sigma = [b_1,\ldots,b_n],\;b_i \in J\}$. Then $I \neq \varnothing$,
since  $\id \in I$. For $\sigma \in I$, if $[a_1,\ldots,a_n] \sigma =
[b_1,\ldots,b_n] \in A^n$, $b_i\in J$,  let $N(\sigma)$ denote 
$\hh (b_1,\ldots,b_n)$. 
	Then, it is enough to prove that there exists
$\sigma \in I$ such that $N(\sigma) = n$. This is proved by showing
that for any $\sigma \in I$ with $N(\sigma) < n$, there exists
$\sigma_1 \in I$ such that $N(\sigma_1) > N(\sigma)$.

Let $\sigma \in I$ be such that $N(\sigma) < n$. Let $[a_1,\ldots,a_n]
\sigma = [b_1,\ldots,b_n] \in A^n$, $b_i\in J$. 
Then, by (\ref{Evans}), there exists
$[c_1,\ldots,c_{n-1}] \in A^{n-1}$ such that $\hh J'_{b_n} \geq n-1$,
where $J' = (b_1+c_1 b_n, \ldots, b_{n-1}+c_{n-1} b_n)$. 
The transformation $\tau$ sending $[b_1,\ldots,b_n]$
to $[b_1+c_1 b_n, \ldots, b_{n-1}+c_{n-1} b_n,b_n]$ is
elementary. 
	Hence $\sigma \tau \in I$. Note that $N(\sigma) = N(\sigma
\tau)$. Hence, if necessary, we can replace $\sigma$ by $\sigma \tau$
and assume that if a prime
ideal $\p$ of $A$ contains $(b_1,\ldots, b_{n-1})$ and does not
contain $b_n$, then, we have $\hh \p \geq n-1$. Now, we claim that
$N(\sigma) = \hh (b_1,\ldots, b_{n-1})$. 

We have $N(\sigma) \leq n-1$. Since $N(\sigma) = 
\hh(b_1,\ldots,b_n)$, we have $\hh (b_1,\ldots, b_{n-1}) \leq
N(\sigma) \leq n-1$. Let $\p$ be a minimal prime ideal of $(b_1,
\ldots, b_{n-1})$ such that $\hh \p = \hh (b_1,\ldots,b_{n-1})$. 
If $b_n \notin \p$, then $\hh \p) \geq n-1$. Hence, we have the
inequalities $n-1 \leq \hh (b_1,\ldots, b_{n-1}) \leq N(\sigma)
\leq n-1$. This implies that $N(\sigma) = \hh (b_1,\ldots,
b_{n-1}) = n-1$. If $b_n \in \p$, then $\hh (b_1,\ldots,b_{n-1}) = \hh
 (b_1,\ldots,b_n) = N(\sigma)$. This proves the claim.

Let $K$ denote the set of minimal prime ideals of
$(b_1,\ldots,b_{n-1})$. Let $K_1 = \{ \p \in K : b_n \in \p \}$ and
let $K_2 = K - K_1$. Note that $K_1 \neq \varnothing$. For, if $K_1 =
\varnothing$, no minimal prime ideal of $(b_1,\ldots, b_{n-1})$
contains $b_n$. Then $\hh (b_1,\ldots,b_n) > \hh
(b_1,\ldots,b_{n-1})$. But, $\hh (b_1,\ldots,b_n) = N(\sigma) = \hh
(b_1,\ldots,b_{n-1})$. This proves that $K_1 \neq \varnothing$.
Now, we claim that $f \in \p$ for all $\p \in K_1$.

For, if $f \notin \p$ for some $\p \in K_1$, then, since $\p$ is minimal
prime ideal of $(b_1,\ldots,b_{n-1})$, we have $\hh \p \leq
n-1$. Since $b_n \in \p \in K_1$, we have $\p$ is minimal prime
ideal of $(b_1,\ldots,b_n)$ and $f\notin \p$. Hence $\hh \p A_f = 
\hh \p \leq n-1$. But $(b_1,\ldots,b_n) A_f \subset \p A_f$ and 
$\hh (b_1,\ldots, b_n) A_f = \hh (a_1,\ldots,a_n) A_f = n$. This is
a contradiction. This proves that $f \in \p$ for all $\p \in
K_1$.

If $\p \in K$, then $\hh \p \leq n-1$, since $\p$ is a minimal prime
ideal of $(b_1,\ldots, b_{n-1})$. If $\p \in K_2$, then $b_n \notin \p$
and hence $\hh \p \geq n-1$. Therefore, $\hh \p = n-1$ for all $\p
\in K_2$. Since $\bigcap_{\p \in K_2}\p \not\subset \bigcup_{\p \in K_1}
\p$, choose $x \in \bigcap_{\p \in K_2}\p$ such that $x \notin
\bigcup_{\p \in K_1} \p$. Since $f \in \p$ for all $\p \in K_1$, we
have $x f \in \bigcap_{\p \in K} \p$. This implies that $(x f)^r \in
(b_1,\ldots, b_{n-1})$ for some positive integer $r$.

Let $(x f)^r = \sum_1^{n-1} d_i b_i$. Consider the matrix $\theta \in$
E$_n(A)$ which takes $[c_1,\ldots,c_n]$ to
$[c_1,\ldots,c_{n-1},c_n+ \sum_1^{n-1} d_i c_i]$.
Let $\theta_1$ be the matrix : diagonal $(1,\ldots,1,f^r)$.
Then $\theta_1 \in \GL_n(A_f)$. Let $\wt
\theta = \theta_1\theta_f \theta_1^{-1}$. Then $\wt \theta \in
\SL_n(A_f)$. Let $\sigma_1 = \sigma \wt \theta$. Then 

$[a_1,\ldots,a_n] \sigma_1 = [b_1,\ldots,b_n] \wt \theta
	= [b_1,\ldots,b_n]\theta_1\theta_f \theta_1^{-1}$

$=[b_1,\ldots,b_{n-1},f^r b_n]\theta_f \theta_1^{-1}$
$= [b_1,\ldots,b_{n-1},f^r b_n + \sum_1^{n-1} d_i b_i]
	\theta_1^{-1}$ 

$=[b_1,\ldots,b_{n-1},f^r b_n + f^r x^r]\theta_1^{-1}$ 
$= [b_1,\ldots,b_{n-1},b_n + x^r]$.

Now, we claim that no minimal prime ideal of $(b_1,\ldots,b_{n-1})$
contains $b_n + x^r$. For, if $\p \in K_1$, then
$b_n \in \p$, but $x \notin \p$. Hence $b_n + x^r \notin \p$ for all
$\p \in K_1$. If $\p \in K_2$, then $b_n \notin \p$, but $x \in
\p$. Hence $b_n + x^r \notin \p$ for all $\p \in K_2$. This proves the
claim. Hence $\hh (b_1,\ldots,b_{n-1}, b_n + x^r) >
\hh (b_1,\ldots,b_{n-1})$. Note that $\sigma_1 \in I$ and $N(\sigma_1) =
\hh (b_1,\ldots,b_{n-1},b_n+ x^r) > \hh
(b_1,\ldots,b_{n-1}) =  N(\sigma)$. This completes the proof of the
lemma. $\hfill \gj$
\end{proof}

\begin{lemma}\label{lem3.6}
Let $A$ be an affine domain over a field $k$ of dimension $n \geq 2$
and let $f$ be a non-zero element of $A$. Let $J \subset A$ be an ideal of
height $n$ such that $J/J^2$ is generated by $n$ elements. Suppose
that $(J,w_J) \neq 0$ in $E(A)$, but the image of $(J,w_J) = 0$ in
$E(A_f)$. Then, there exists an ideal $J_2$ of height $n$ such that
$(J_2)_f = A_f$ and $(J,w_J) = (J_2,w_{J_2})$ in $E(A)$.
\end{lemma}

\begin{proof}
Since $(J,w_J) \neq 0$ in $E(A)$ and $(J,w_J) = 0$ in $E(A_f)$, $f\in
A$ is not a unit. By (\ref{lem3.5}), we can choose an ideal $J_1$ of
height $n$ which is comaximal with $J$ and $(f)$ such that $(J,w_J) +
(J_1,w_{J_1}) = 0$ in $E(A)$. We have $(J_1,w_{J_1}) \neq 0$ in
$E(A)$. Since the image of $(J,w_J) = 0$ in $E(A_f)$, it follows that
the image of $(J_1,w_{J_1}) = 0$ in $E(A_f)$. By (\ref{theo2.2}),
$w_{J_1}$ induces a global orientation of $(J_1)_f$. 
	Hence $(J_1)_f =
(b_1,\ldots,b_n)$ and $w_{J_1} \otimes A_f$ is induced by the set of
generators $b_1,\ldots,b_n$ of $(J_1)_f$ mod $(J_1)_f^2$. Choose $k$
large enough such that $f^{2k}b_i \in J_1, 1\leq i \leq n$. Since
$f$ is a unit modulo $J_1$, by (\ref{lem3.4}), $(J_1,w_{J_1}) =
(J_1,\ol {f^{2kn}}w_{J_1})$ in $E(A)$.  
	Hence, without loss of
generality, we can assume that $b_i \in J_1$.  Now, by lemma
(\ref{0.0}), we get $\sigma \in \SL_n(A_f)$ such that if
$[b_1,\ldots,b_n]\sigma = [c_1,\ldots,c_n]$, then $c_i \in J_1$ and
$c_1,\ldots,c_n$ generate an ideal of height $n$ in $A$.
We claim that $(c_1,\ldots,c_n)=J_1\cap J_2$, where $(J_2)_f = A_f$. 

Let $(c_1,\ldots,c_n)=\G_1\cap\ldots \cap \G_r\cap\G_{r+1}\ldots\cap
\G_t$ be a reduced primary decomposition, where $\G_i$ is
$\m_i$-primary ideal. Assume that $f\in \m_i$ for $r+1\leq i\leq t$
and $f\notin \m_i$ for $1\leq i\leq r$. Then $(c_1,\ldots,c_n)_f
=(J_1)_f= \bigcap_1^r (\G_i)_f$. We observe that $J_1=\bigcap_1^r
\G_i$. This follows easily from the fact that $f$ is a unit modulo
$J_1$ and modulo $\G_i,1\leq i\leq r$. Write $J_2 = \bigcap_{r+1}^t
\G_i$. Then, we have $(c_1,\ldots,c_n)=J_1 \cap J_2$ and $(J_2)_f =
A_f$. 

Note that $A/J_1 \iso A_f/(J_1)_f$. The image of $\sigma\in
\SL_n(A_f)$ in $\SL_n(A_f/(J_1)_f)$ gives rise to an element in
$\SL_n(A/J_1)$ and hence the $n$ generators of $(J_1)_f/(J_1)_f^2$ gives
rise to $n$ generators of $J_1/J_1^2$. 

Now, $J_1 \neq (c_1,\ldots,c_n)$, since $(J_1,w_{J_1}) =(J_1,\ol
{f^{2kn}} w_{J_1}) \neq 0$ in $E(A)$.  Hence $J_2$ is a proper ideal
of height $n$. Since $(f) + J_1 = A$ and $(J_2)_f=A_f$, we have $J_1 +
J_2 = A$. Since det $(\sigma) = 1$, $\ol {f^{2kn}}w_{J_1}$ is given by
the set of generators $c_1,\ldots,c_n$ of $J_1$ modulo
$J_1^2$. 
	Therefore, $(J_1,\ol {f^{2kn}} w_{J_1}) + (J_2,w_{J_2}) = 0$
in $E(A)$, where $w_{J_2}$ is given by the set of generators
$c_1,\ldots,c_n$ of $J_2$ modulo $J_2^2$. Since $(J,w_J) +
(J_1,w_{J_1}) = 0 $ in $E(A)$, it follows that $(J,w_J) =
(J_2,w_{J_2})$ in $E(A)$.  This proves the lemma. $\hfill \gj$
\end{proof}

\begin{rem}
The hypothesis $A$ is an affine domain ensures that $\dim A_f =n$ and
hence $E(A_f)$ is defined.
\end{rem}

\begin{lemma}\label{lem3.7}
Let $A$ be an affine domain over a field $k$ of dimension $n \geq 2$
and let $P$ be a projective $A$-module of rank $n$ having trivial
determinant. Let $f \in A$ be a non-zero element. Suppose that the
projective $A_f$-module $P_f$ has a unimodular element. Then, there
exists a surjection $\alpha : P \surj J$, where $J \subset A$ is an
ideal of height $n$ such that $J_f = A_f$.
\end{lemma}

\begin{proof}
Suppose $P$ has a unimodular element. Then, by 
(\ref{cor2.4}), $e(P,\chi) =
0$ in $E(A)$. Let $J$ be an ideal of height $n$ and generated by $n$
elements such that $f \in J$. Then, for some $w_J$, 
 $(J,w_J) = 0 = e(P,\chi)$ in
$E(A)$. By (\ref{cor2.3}), there exists a surjection $\alpha :
P \surj J$ and $J_f = A_f$.

Now, assume that $P$ has no unimodular element. Let $e(P,\chi) =
(J,w_J)$, where $J$ is an ideal of height $n$ such that $J/J^2$ is
generated by $n$ elements. Now $(J,w_J) \neq 0$ in $E(A)$, but its
image in $E(A_f) = 0$. Therefore, by (\ref{lem3.6}), there exists an
ideal $J_2$ of height $n$ such that $({J_2})_f = A_f$ and $(J,w_J) =
(J_2,w_{J_2})$ in $E(A)$. Then $e(P,\chi) = (J_2,w_{J_2})$. Hence, by
(\ref{cor2.3}), there exists a surjection $\alpha : P \surj J_2$
such that $(J_2,w_{J_2})$ is obtained from $(\alpha,\chi)$.  This
proves the lemma. $\hfill \gj$
\end{proof}

\begin{defin} 
Let $A$ be a Noetherian ring of dimension $n$ and let $P$ be a projective
$A$-module of rank $n$. Let $\alpha : P \surj J$ be a surjection. We
say that $\alpha$ is a {\it generic} surjection, if $J$ has height
$n$. In this case $J$ is said to be a {\it generic surjection ideal}
or {\it generic section ideal} of $P$.
\end{defin}

\begin{lemma}\label{lem3.8}
Let $A$ be an affine domain over a field $k$ of dimension $n \geq 2$
and let $P$ be a projective $A$-module of rank $n$ having trivial
determinant. Let $f \in A$ be a non-zero element. Assume that every
generic surjection ideal of $P$ is generated by $n$ elements. Then,
every generic surjection ideal of $P_f$ is generated by $n$ elements.
\end{lemma}

\begin{proof}
Let $\beta : P_f \surj \wt J$ be a generic surjection. Let
$J' = \wt J \cap A$. Then $J'A_f = \wt J$ and $(f)+J'=A$. 
Let $\chi$ be a generator of $\wedge^n(P)$, and let
$(J^\prime_f,w_{J^\prime_f})$ be obtained from $(\beta,\chi_f)$. Since
$\ol f \in A/J'$ is a unit, by (\ref{lem3.4}), we may replace
$w_{J^\prime_f}$ by $\ol {f^{2m}}w_{J^\prime_f}$ for some large
integer $m$ and assume that $w_{J^\prime_f}$ is given by a set of
generators of $J^\prime/{J^\prime}^2$ which induce $w_{J^\prime}$. The
element $e(P,\chi) - (J^\prime,w_{J^\prime})$ of $E(A)$ is zero in
$E(A_f)$. 
If $e(P,\chi) - (J^\prime,w_{J^\prime}) = 0$ in $E(A)$, then $e(P,\chi)
= (J',w_{J'})$. Hence, there exists a surjection from $P$ to $J'$, by 
(\ref{cor2.3}). By assumption, $J'$ is generated by $n$
elements. Hence $\wt J$ is generated by $n$ elements.
Therefore, assume otherwise. By (\ref{200}),
$ e(P,\chi) - (J^\prime,w_{J^\prime}) = (J_2,w_{J_2})$ in
$E(A)$.
	By (\ref{lem3.6}), we can assume that 
$({J_2})_f = A_f$. Since $J^\prime + Af = A$, we have 
$J^\prime + J_2 = A$.  Hence $e(P,\chi) = (J^\prime,w_{J^\prime}) +
(J_2,w_{J_2}) = (J' \cap J_2,w_{J' \cap J_2})$ in $E(A)$, where 
$w_{J' \cap J_2}$ is obtained from $w_{J'}$ and $w_{J_2}$. 
	By (\ref{cor2.3}), there exists a
surjection $\gamma : P \surj J^\prime \cap J_2$. By hypothesis,
$J^\prime \cap J_2$ is generated by $n$ elements, as it is a generic
surjection ideal of $P$. Hence 
$J^\prime_f = (J^\prime \cap J_2)_f$ is generated by $n$
elements. This proves the lemma. $\hfill \gj$
\end{proof}

\begin{theorem}\label{lem3.9}
Let $A$ be an affine domain over $\R$ of dimension $n \geq 2$ and let
$P$ be
a projective $A$-module of rank $n$ having trivial determinant.
Assume that for every generic surjection $\alpha : P \surj J$, the
generic surjection ideal $J$ is generated by $n$ elements. Then $P$
has a unimodular element.
\end{theorem}

\begin{proof}
To any generic surjection $\alpha : P \surj J$, we associate an
integer $N(P,\alpha)$, which is equal to the number of real maximal
ideals containing $J$ (if $\M$ is a maximal ideal of $A$
containing $J$, then it is called {\it real} if the quotient field
$A/\M$ is isomorphic to $\R$, otherwise it is called a {\it complex}
maximal ideal and in this case $A/\M$ is isomorphic to $\C$). Let
$t(P)$ = min $N(P,\alpha)$, where $\alpha$ varies over all generic
surjections of $P$.
\paragraph{Case 1.} 
Suppose that $t(P) = 0$. Let $\alpha : P \surj J$ be a generic
surjection with $N(P,\alpha) = 0$. This means that $J$ is contained
only in complex maximal ideals. By assumption, $J$ is generated by $n$
elements. These $n$ elements give rise to $\wt w_J$, a local orientation
of $J$, such that the
element $(J,\wt w_J) = 0$ in $E(A)$. Let $\chi$ be a generator of
$\wedge^n(P)$ and $e(P,\chi) = (J,w_J)$ in $E(A)$. Then, by
(\ref{lem3.0}), $(J,w_J) = (J,\ol u \wt w_J)$ in $E(A)$, where
$\ol u \in A/J$ is a unit. Since $J$ is contained only in complex
maximal ideals, $\ol u$ is a square. It follows now from
(\ref{lem3.4}), that $e(P,\chi) = (J,w_J) = (J,\ol u\wt w_J) =
(J,\wt w_J) = 0$ in $E(A)$. Therefore, by (\ref{cor2.4}), $P$ has a
unimodular element.

\paragraph{Case 2.}
Suppose that $t(P) = 1$. Let $\alpha : P \surj J$ be a generic
surjection with $N(P,\alpha) = 1$. This means that $J$ is contained
only in one real maximal ideal. By assumption, $J$ is generated by $n$
elements. These $n$ elements give rise to  $\wt w_J$, a local 
orientation of $J$,
such that the element $(J,\wt
w_J) = 0 = (J,-\wt w_J)$ in $E(A)$. Let $\chi$ be a generator of
$\wedge^n(P)$ and $e(P,\chi) = (J,w_J)$ in $E(A)$. Let $(J,w_J) =
(J,\ol u\wt w_J)$ in $E(A)$, where $\ol u\in A/J$ is a unit. 
Then, since $J$ is contained only in
one real maximal ideal, it follows, as in case 1, that either $\ol u \in
A/J$ is a square or $-\ol u$ is a square. Therefore, it follows that
either $(J,w_J) = (J,\wt w_J)$ or $(J,w_J) = (J,-\wt w_J)$ in
$E(A)$. In any case, $(J,w_J) = 0$ in $E(A)$ and hence, by
(\ref{cor2.4}), $P$ has a unimodular element.

\paragraph{Case 3.}
Now, we show that under the assumption of the theorem $t(P) \leq 1$ and
hence the theorem will be proved. Suppose $N(P,\alpha) = r \geq
2$. Let $\m_1,\ldots,\m_r$ be the real maximal ideals containing
$J$. Let $f\in A$ be chosen so that $f$ belongs to only the real
maximal ideals $\m_2,\ldots,\m_r$ (Such an $f$ exists, for choose a set
of generators $h_1,\ldots,h_k$ of $\m_2 \cap \ldots \cap \m_r$. Take
$f = h_1^2 + \ldots +h_k^2$).  Then $N(P_f,\alpha_f) = 1$ and hence
$t(P_f) \leq 1$. Since, for every generic surjection $\alpha : P \surj
J, J$ is generated by $n$ elements, it follows from
(\ref{lem3.8}), that for every generic surjection $\beta : P_f \surj
J^\prime_f, J^\prime_f$ is generated by $n$ elements. Hence, by cases
1 and 2, $P_f$ has a unimodular element. Therefore, by
(\ref{lem3.7}), there exists a surjection $\gamma : P \surj J_1$,
where $J_1$ is an ideal of height $n$ such that $({J_1})_f = A_f$. Since
$\m_2,\ldots,\m_r$ are the only real maximal ideals containing $f$, it
follows that $N(P,\gamma) = r-1$. Repeating this process, we see that
$t(P) \leq 1$. This proves the theorem.  $\hfill \gj$
\end{proof}


\chapter{The Weak Euler Class Group of a Noetherian Ring}

Let $A$ be a Noetherian ring of dimension $n \geq 2$. We define the
{\it weak Euler class group} $E_0(A)$ of $A$ as follows:

Let $S$ be the set of ideals $\n \subset A$ such that $\n/\n^2$ is
generated by $n$ elements (where $\n$ is $\M$-primary ideal for some
maximal ideal $\M$ of height $n$). Let $G$ be the free abelian group on
the set $S$.

Let $J = \cap\n_i$ be the intersection of finitely many ideals $\n_i$,
where $\n_i$ is $\M_i$-primary and $\M_i$ being distinct maximal
ideals of height $n$. Assume that $J/J^2$ is generated by $n$
elements. We associate to $J$, the element $\sum\n_i$ of $G$. By
abuse of notation, we denote this element of $G$ by $(J)$.  
	Let $H$ be the
subgroup of $G$ generated by elements of the type $(J)$, where
$J\subset A$ is an ideal of height $n$ generated by $n$ elements.
\vspace*{.1in} 

\begin{define} The weak Euler class group of $A$
is defined as $E_0(A) = G/H$.
\end{define}

Let $P$ be a projective $A$-module of rank $n$ with trivial
determinant and let $\lambda : P \surj J_0$ be a surjection, where $J_0
\subset A$ is an ideal of height $n$. We define $e(P) = (J_0)$ in
$E_0(A)$. We show that this assignment is well defined.

Let $\mu : P \surj J_1$ be another surjection, where $J_1$ is an ideal
of height $n$. Then, by (\ref{lem1.1}), there exists a surjection
$\alpha(T) : P[T] \surj I$, where $I \subset A[T]$ is an ideal of
height $n$ with $\alpha(0) = \lambda$ and $\alpha(1) = \mu$. Now,
using (\ref{prop1.2}), there exists an ideal $K$ of height $n$
comaximal with $I\cap A$ such that $I\cap KA[T]$ is generated by $n$
elements. Therefore $J_0 \cap K$ and $J_1\cap K$ are generated by $n$
elements. Hence $(J_0) = (J_1)$ in $E_0(A)$.

We note that there is a canonical surjective homomorphism from $E(A)$
to $E_0(A)$ obtained by forgetting  orientations.

The aim of this chapter is to prove theorem (\ref{theo}), i.e. if 
$A$ is Noetherian ring of even dimension $n$, then $(J)$ is zero in 
$E_0(A)$ if and only if $J$ is a surjective image of a stably free 
$A$-module of rank $n$. This is proved along the same lines as
(\ref{theo2.2}): first we prove some
addition and subtraction principles (\ref{prop4.1}), and then using
the group theoretic lemma (\ref{result1}), we prove the theorem.

\begin{lemma}\label{lem6}
Let $A$ be a ring and let $P$ be a projective $A$-module of rank $n$. Let
$\alpha$ be any element of $P^{\ast}$. Let $p_0,p_1,\ldots,p_n$
be $n+1$ elements of $P$. Let $w_i \in \wedge^n(P)$ be defined as
follows : $w_0 = \alpha(p_0)(p_1 \wedge p_2 \wedge \ldots \wedge
p_n)$, $ w_i = \alpha(p_i)(p_0 \wedge \ldots \wedge p_{i-1} \wedge
p_{i+1} \wedge \ldots \wedge p_n), \;\;1 \leq i \leq n.$ Then
$\sum_{i=0}^n (-1)^i w_i = 0.$
\end{lemma}

\begin{proof}
Let $e$ denote the element $(1,0) \in A \oplus P$. The map
$x \mapsto e \wedge x$ is an isomorphism from $\wedge^n(P) \by \theta
\wedge^{n+1}(A \oplus P).$
Let $w = \sum_{i=0}^n (-1)^i w_i.$ Now, consider the map 
$\gamma : P \ra A \oplus P$ defined by
$\gamma(p) = (\alpha(p),p).$ We obtain an induced map
$\wedge^{n+1}\gamma : \wedge^{n+1}P \ra \wedge^{n+1}(A \oplus P)$.
We get $\wedge^{n+1} \gamma (p_0\wedge\ldots\wedge p_n) = e \wedge w +
p_0\wedge\ldots\wedge p_n.$ But $\wedge^{n+1}(P) = 0$ hence 
$e \wedge w = 0$. But, the map $\theta$ is an isomorphism, 
hence $ w = 0$.
 
$\hfill \gj$
\end{proof}

\begin{lemma}\label{lem7}
Let $A$ be a Noetherian ring and let $P$ be a projective $A$-module 
of rank $n$. Suppose that we are given the following short 
exact sequence 
$$0 \rightarrow P_1 \rightarrow A\oplus P \longby {(b,-\alpha)}  
A \rightarrow 0.$$ 
Let $(a_0,p_0)\in  A\oplus P$ be such that 
$a_0b - \alpha(p_0) = 1.$
Let $q_i = (a_i,p_i) \in P_1$ , $1 \leq i \leq n.$ Then, 

(1) The map $\delta : \wedge^n(P_1) \rightarrow \wedge^n(P)$ given by 
$\delta (q_1 \wedge \ldots \wedge q_n) 
= a_0(p_1 \wedge \ldots \wedge p_n) +
\sum_{i=1}^n (-1)^i a_i (p_0 \wedge \ldots \wedge p_{i-1} \wedge
p_{i+1} \wedge \ldots \wedge p_n)$ \\
is an isomorphism. 

(2) $\delta (bq_1 \wedge \ldots \wedge q_n) = p_1 \wedge \ldots \wedge p_n.$
\end{lemma}

\begin{proof}
Let $e=(1,0), f=(a_0,p_0)$ in $A\oplus P$. Then
$A\oplus P = Af\oplus P_1$ and as in (\ref{lem6}),
$f\wedge q_1\wedge \ldots\wedge q_n = e\wedge w$ in 
$\wedge^{n+1}(A\oplus P)$,  where 
$w = a_0(p_1\wedge \ldots\wedge p_n) + \sum_{i=1}^n (-1)^i
a_i(p_0\wedge \ldots\wedge p_{i-1}\wedge p_{i+1}\wedge \ldots \wedge p_n).$

Since the map $x\mapsto e\wedge x$ is an isomorphism from $\wedge^n(P)$ to
$\wedge^{n+1}(A \oplus P),$ result (1) follows. 
Since $q_i=(a_i,p_i) \in P_1$, we have $ba_i= \alpha(p_i), 
1\leq i\leq n.$ Moreover, $ba_0 = 1+\alpha(p_0)$.
Therefore, (2) follows from (\ref{lem6}).
$\hfill \gj$
\end{proof}

\begin{lemma}\label{lem8}            
Let $A$ be a Noetherian ring and let $P$ be a projective $A$-module of rank
$n$. Suppose that we are given the following exact sequence 
$$0 \rightarrow P_1 \rightarrow A\oplus P \longby {(b,-\alpha)} A
	\rightarrow 0.$$ 
Then, $(i)$The map $\beta : P_1 \ra A$ given by $\beta (q) = c$, where
$q = (c,p)$, has the property that $\beta (P_1) = \alpha (P).$ 
$(ii)$ The map $\Phi : P \ra P_1$ given by $\Phi (p) = (\alpha
(p),bp)$ has the property that $\beta \Phi = \alpha$ and $\delta
\wedge^n (\Phi)$, where $\delta$ is as in (\ref{lem7}), is scalar
multiplication by $b^{n-1}.$
\end{lemma}

\begin{proof}
Let $c\in \beta(P_1),$. Then, there exists $q=(c,p) \in P_1$ such that
$\beta(q) = c$. Since $q\in P_1$, we have $b c = \alpha(p)$.  Also,
there exists $q_0=(a_0,p_0) \in A\oplus P$ such that $a_0
b-\alpha(p_0) =1$. Now, $\alpha(a_0 p-c p_0)=c$, hence $c\in
\alpha(P)$.  Conversely, let $c=\alpha(p),\;p\in P$. Then $b
c=\alpha(b p)$.  This shows that $(c,b p)\in P_1$, and hence $c\in
\beta(P_1)$. This proves the first part.

The map 
$\delta : \wedge^n(P_1) \rightarrow     \wedge^n(P)$ is given by 
$$\delta (q_1 \wedge \ldots \wedge q_n) = a_0(p_1 \wedge \ldots \wedge
	p_n) + \sum_{i=1}^n (-1)^i a_i (p_0 \wedge \ldots \wedge p_{i-1}
	\wedge p_{i+1} \wedge \ldots \wedge p_n),$$ 
where $q_i=(a_i,p_i) \in P_1, 1\leq i\leq n,$ and $(a_0,p_0) \in
A\oplus P.$ We have 

$\delta \wedge^n(\Phi) (p_1\wedge \ldots\wedge p_n)  =
	\delta((\alpha(p_1),bp_1)\wedge \ldots\wedge 
	(\alpha(p_n),bp_n)) $

$= a_0b^n(p_1 \wedge \ldots \wedge p_n)+\sum_{i=1}^n (-1)^i
	\alpha(p_i) b^{n-1} (p_0 \wedge \ldots\widehat{\wedge p_i} \ldots
	\wedge p_n) $

$ =  b^{n-1} ((1+\alpha(p_0))(p_1\wedge \ldots\wedge p_n)
	+\sum_{i=1}^n (-1)^i \alpha(p_i)
 	(p_0 \wedge \ldots\widehat{\wedge p_i} \ldots \wedge p_n)) $

$= b^{n-1}(p_1\wedge \ldots\wedge p_n+\alpha(p_0)(p_1\wedge
	\ldots\wedge p_n)+ \sum_{i=1}^n (-1)^i \alpha(p_i) (p_0 \wedge
	\ldots\widehat{\wedge p_i} \ldots \wedge p_n)) $

$ = b^{n-1}(p_1\wedge \ldots\wedge p_n),  \mbox{by
	(\ref{lem6})}$.  
$\hfill \gj$
\end{proof}

\begin{lemma}\label{lem3.1}
Let $A$ be a Noetherian ring of dimension $n \geq 2$. Let $P$ be a
projective $A$-module of rank $n$ with trivial determinant and let $\chi$
be an orientation of $P$. Let $\alpha : P \surj J$ be a surjection,
where $J \subset A$ be an ideal of height $n$ and let $(J,w_J)$ be
obtained from $(\alpha,\chi)$. Let $a,b \in A$ be such that $ab = 1$
modulo $J$ and let $P_1$ be the kernel of the surjection 
$(b,-\alpha) : A \oplus P \surj A$. Let $\beta : P_1 \surj J$ be as in
(\ref{lem8}) and let $\chi_1$ be the orientation of $P_1$ given by
$\delta^{-1} \chi : A \by \sim \wedge^n(P_1)$ (where $\delta$ is as in
(\ref{lem7})). Then $(J,\ol {a^{n-1}}w_J)$ is obtained from
$(\beta,\chi_1)$.
\end{lemma}

\begin{proof}
We have an exact sequence \;\;
	$0 \ra P_1 \ra A \oplus P \longby {(b,-\alpha)} A \ra 0$.
The map $\beta : P_1 \surj J$ is defined by 
$\beta (q) = c$, where $q = (c,p)$. By (\ref{lem8}), we have
$\beta (P_1) = \alpha (P) = J$. Since $P_1$ is stably isomorphic to
$P$, determinant of $P_1$ is trivial. Let 
$\delta : \wedge^n(P_1) \iso \wedge^n(P)$ be an isomorphism defined 
in (\ref{lem7}). Let 
$\chi_1 = \delta^{-1}\chi : A \iso \wedge^n(P_1)$ be an orientation
of $P_1$. Then, by (\ref{lem8}), the map $\Phi : P \ra P_1$ given by
$\Phi(p) = (\alpha(p),bp)$ has the property that 
$\beta\Phi = \alpha$ and $\delta\wedge^n(\Phi)$ is a scalar
multiplication by $b^{n-1}$. By (\ref{lem2}), the map 
$\Phi \otimes A/J : P/JP \iso P_1/JP_1$ is an isomorphism. Let 
$\ol \gamma : (A/J)^n \iso P/JP$ be an isomorphism such that
$\wedge^n(\ol \gamma) = \ol \chi$ and $w_J = \ol \alpha \ol \gamma$.
Consider the commutative diagram 
\Com
$$\xymatrix{ 
	P_1/JP_1 \ar@{->>}[r]^{\ol \beta}\ar@{<-}[d]_{\ol
	\Phi\ol \gamma} & J/J^2	\\ 
	(A/J)^n\ar@{->>}[ru]_{w_J}  & 
	}$$ 
Since $\wedge^n(\ol \Phi\ol\gamma) = \ol\delta^{-1}\ol
{b^{n-1}}\ol\chi = \ol {b^{n-1}}\ol \chi_1$.  
Hence $\ol\chi_1 = \ol {a^{n-1}}\wedge^n(\ol\Phi\ol\gamma)$, 
since $ab = 1$ modulo $J$. Let $\theta$ be an automorphism of
$(A/J)^n$ of determinant $\ol {a^{n-1}}$. Now, consider the isomorphism 
$\ol \Phi \ol \gamma\theta : (A/J)^n \iso P_1/JP_1$. 
Then $(J,\wt w_J)$ is obtained from $(\beta,\chi_1)$, where 
$\wt w_J = \ol \beta\ol \Phi \ol \gamma\theta = 
\ol \alpha \ol \gamma\theta =  w_J\theta=\ol {a^{n-1}} w_J$. 
Hence $(J,\ol {a^{n-1}}w_J)$ is obtained from
$(\beta,\chi_1)$. This proves the lemma.  
$\hfill \gj$
\end{proof}

\begin{lemma}\label{lemma3.7}
Let $A$ be a Noetherian ring of even dimension $n$. Let $P$ be a
stably free $A$-module of rank $n$ and let $\chi$ be a generator of
$\wedge^n(P)$. Suppose that $e(P,\chi) = (J,w_J)$ in $E(A)$, where $J$
is an ideal of height $n$ and $w_J$ is a local orientation of $J$. Then,
there exists an ideal $J^\prime$ of height $n$ which is generated by
$n$ elements and a local orientation $w_{J^\prime}$ of $J^\prime$ such
that $(J,w_J) = (J^\prime,w_{J^\prime})$ in $E(A)$. Moreover,
$J^\prime$ can be chosen to be comaximal with any given ideal of
height $n$.
\end{lemma}

\begin{proof}
By Bass Cancellation theorem (\ref{Bass}), we have
$P\oplus A \simeq A^{n+1}$. We have $P =
A^{n+1}/(a_0,\ldots,a_n)$ for some unimodular row $(a_0,\ldots,a_n)$
in $A^{n+1}$. We can assume, by
(\ref{Evans}), that $J^\prime = (a_1,\ldots,a_n)$ is an ideal of
height $n$. Further, we can assume that $J'$ is
comaximal with any given ideal of height $n$.

Let $\ol e_i$ be the image of the basis vector $e_i$ of $A^{n+1}$ in
$P$. Then, there exists a surjective map $\psi : P \surj J^\prime$
defined by $\psi(\ol e_0) = 0$, $\psi (\ol e_i) = a_{i+1}$ if $i$ is
odd and $\psi (\ol e_i) = -a_{i-1}$ if $i$ is even. The map $\psi$ is
well defined, since we have $\sum_0^na_i\ol e_i = 0$ in $P$ and
$\psi(\sum_0^na_i\ol e_i) = 0$. Computing $e(P,\chi)$ using
$\psi$, we see that $(J,w_J) = (J^\prime,w_{J^\prime})$.

$\hfill \gj$
\end{proof}

\begin{lemma}\label{lemma3.6}
Let $A$ be a Noetherian ring of even dimension $n$. Let $P$ be a
projective $A$-module of rank $n$ having trivial determinant and
let $\chi_P$ be a generator of $\wedge^n(P)$. Let $e(P,\chi_P) = (J,w_J)$
in $E(A)$, where $J$ is an ideal of height $n$ and $w_J$ is a local
orientation of $J$. Suppose $\wt w_J$ is another local orientation of
$J$. Then, there exists a projective $A$-module $P^\prime$ of rank $n$
with $[P^\prime] = [P]$ in $K_0(A)$ and a generator $\chi_{P^\prime}$
of $\wedge^n(P^\prime)$ such that 
$e(P^\prime,\chi_{P^\prime}) = (J,\wt w_J)$ in $E(A)$.
\end{lemma}

\begin{proof}
By (\ref{lem3.0}), $\wt w_J = \ol b w_J$ for some unit $\ol b\in
A/J$.   By (\ref{lem3.1}), there exists a projective $A$-module
$P^\prime$ of rank $n$ with $[P^\prime] = [P]$ in $K_0(A)$ and a
generator $\chi_{P^\prime}$ of $\wedge^n(P^\prime)$ such that
$e(P^\prime,\chi_{P^\prime}) = (J,\ol {b^{n-1}}w_J)$ in
$E(A)$. Applying (\ref{lem3.4}), we get the result.  $\hfill
\gj$
\end{proof}

\begin{proposition}\label{prop4.1}
Let $A$ be a Noetherian ring of even dimension $n$. Let $J_1$ and $J_2
$ be two comaximal ideals of $A$ of heights $n$ and $J_3 = J_1 \cap
J_2$.  Then 
 
	(i) {\rm (Addition Principle)} If $J_1$ and $J_2$
are surjective images of stably free $A$-modules of rank $n$, then so
is the $J_3$.

	(ii) {\rm (Subtraction Principle)} If $J_1$ and $J_3$
are surjective images of stably free $A$-modules of rank $n$, then so
is $J_2$.
\end{proposition}

\begin{proof}
(i) Suppose that $J_1$ and $J_2$ are surjective images of stably
free $A$-modules. Hence, there exist surjections $\psi_1 : P_1 \surj
J_1$ and $\psi_2 : P_2 \surj J_2$, where $P_1$ and $P_2$ are stably
free $A$-modules of rank $n$. We choose orientations $\chi_1$ and
$\chi_2$ of $P_1$ and $P_2$ respectively. Then $e(P_1,\chi_1) =
(J_1,w_{J_1})$ and $e(P_2,\chi_2) = (J_2,w_{J_2})$ in $E(A)$ for some
local orientations $w_{J_1}$ and $w_{J_2}$ of $J_1$ and $J_2$
respectively.  By (\ref{lemma3.7}), we can choose two comaximal ideals
$J_1^\prime$ and $J_2^\prime$ of height $n$ which are generated by $n$
elements such that
$$(J_1,w_{J_1}) = (J_1^\prime,w_{J_1^\prime})\quad{\rm and}\quad
	(J_2,w_{J_2}) = (J_2^\prime, w_{J_2^\prime}) \eqno(1)$$ 
in $E(A)$. Let
$J_3^\prime = J_1^\prime \cap J_2^\prime$. Let
$$(J_1,w_{J_1}) + (J_2,w_{J_2}) = (J_3,w_{J_3}) \eqno(2)$$
$$(J_1^\prime,w_{J_1^\prime}) + (J_2^\prime,w_{J_2^\prime}) =
	(J_3^\prime,w_{J_3^\prime}) \eqno (3)$$ 
	in $E(A)$. Then, we have
$$(J_3,w_{J_3}) = (J_3^\prime,w_{J_3^\prime}) \eqno (4)$$ 
	in $E(A)$.  
Since $J_1^\prime$ and $J_2^\prime$ are generated by $n$
elements, by (\ref{theo1.4}), $J_3^\prime$ is also generated by $n$
elements. Therefore, applying (\ref{lemma3.6}) with $P$ free, there
exists a stably free $A$-module $P_3$ of rank $n$ and an orientation
$\chi_3$ of $P_3$ such that $e(P_3,\chi_3) =
(J_3^\prime,w_{J_3^\prime})$. Hence, by (\ref{cor2.3}), there exists
a surjection from $P_3$ to $J_3$.

(ii) Assume that $J_1$ and $J_3$ are surjective images of stably free
$A$-modules of rank $n$. Let $\psi_3 : P_3 \surj J_3$ be a surjection,
where $P_3$ is a stably free $A$-module of rank $n$. Let $\chi_3$ be
an orientation of $P_3$. Then
$$e(P_3,\chi_3) = (J_3,w_{J_3}) \eqno (5)$$ 
	in $E(A)$ for some local orientation $w_{J_3}$ of $J_3$.  Let
$w_{J_1}$ and $w_{J_2}$ be local orientations of $J_1$ and $J_2$
respectively, obtained from $w_{J_3}$. Then
$$(J_1,w_{J_1}) + (J_2,w_{J_2}) = (J_3,w_{J_3}) \eqno (6)$$ 
	in $E(A)$. 
Since $J_1$ is a surjective image of a stably free $A$-module
of rank $n$, by (\ref{lemma3.6}), there exists a stably free
$A$-module $P_1$ of rank $n$ such that $e(P_1,\chi_1) =
(J_1,w_{J_1})$. Now, since $P_1$ is stably free, by (\ref{lemma3.7}),
there exists an ideal $J_1^\prime$ of height $n$ which is generated by
$n$ elements and  comaximal with $J_2$ such that
$$(J_1,w_{J_1}) = e(P_1,\chi_1) =
	(J_1^\prime,w_{J_1^\prime}). \eqno(7)$$ 
	Let $J_1^\prime \cap J_2 =
J_4$. Then $w_{J_1^\prime}$ and $w_{J_2}$ induce a local
orientation $w_{J_4}$ of $J_4$ such that
$$(J_1^\prime,w_{J_1^\prime}) + (J_2,w_{J_2}) =
	(J_4,w_{J_4}). \eqno (8)$$ 
	By (6,7), we have
$(J_3,w_{J_3}) = (J_4,w_{J_4})$ in $E(A)$. By (5), we
have $e(P_3,\chi_3) = (J_4,w_{J_4})$.  Let $J_1^\prime =
(a_1,\ldots,a_n)$ and let
$$(J_1^\prime,[\ol a_1,\ldots,\ol a_n]) + (J_2,w_{J_2}) =
(J_4,\wt w_{J_4}).$$ 
	Then, since $(J_1^\prime,[\ol a_1,\ldots,\ol a_n]) = 0$ in $E(A)$, 
we have $(J_2,w_{J_2}) = (J_4,\wt w_{J_4})$ in $E(A)$.  
By (\ref{lemma3.6}), there exists a projective $A$-module $Q$ of 
rank $n$ which is stably isomorphic to $P_3$ and hence stably free, 
and an orientation $\chi_Q$ of $Q$ such that
$$e(Q,\chi_Q) = (J_4,\wt w_{J_4}).$$ 
	Therefore $e(Q,\chi_Q) = (J_2,w_{J_2})$. By (\ref{cor2.3}), there
exists a surjection from $Q$ to $J_2$ and hence the proposition is
proved.  $\hfill \gj$
\end{proof}

\begin{theorem}\label{theo}
Let $A$ be a Noetherian ring of even dimension $n$. Let $J \subset A$
be an ideal of height $n$ such that $J/J^2$ is generated by $n$
elements. Then $(J) = 0$ in $E_0(A)$ if and only if $J$ is a
surjective image of a stably free $A$-module of rank $n$.
\end{theorem}

\begin{proof}
We will apply lemma (\ref{result1}) to prove this theorem. 

Let $F$ be the free abelian group on the set $(\n)$, where $\n$ is
$\M$-primary ideal of height $n$ such that $\n/\n^2$ is generated by
$n$ elements.  Define an equivalence relation on the set $(\n)$ by
$(\n)\sim (\n_1)$ if $\sqrt {\n} = \sqrt {\n_1}$, i.e. $\n$ and $\n_1$
both are $\M$-primary ideals of $A$. If $J\subset A$ is an ideal of
height $n$ and $J = \cap\n_i$ is a reduced primary decomposition of
$J$, then denote $(J)$ the element $\sum(\n_i)$ of $F$. Let $S$ be
the set

$\{(J)\in F| J$ is surjective image of a
stably free $A$-module of rank $n\}$. Then

(1) Every element of $S$ is nicely reduced.

(2) Let $x,y\in F$ be nicely reduced such that $x+y$ is also nicely
reduced. Then if any two of $x,y$ and $x+y$ belongs to $S$, then so
does the third. This follows from (\ref{prop4.1}).

(3) Let $x\in F$ be nicely reduced and $x\notin S$ and let
$(\n_i)$, for $i=1,\ldots,r$, be finitely many elements of
$F$. Since $x$ is nicely reduced, we have $x = (J)$ for some
height $n$ ideal $J$. Applying (\ref{cor14}), there exists an ideal
$J^\prime$ of height $n$ which is comaximal with
$J,\n_i,i=1,\ldots,r$ such that $J\cap J^\prime$ is generated by
$n$ elements. Let $y = (J^\prime)$. Then $x+y\in S$.

Let $H^\prime$ be the subgroup of $F$ generated by $S$. Then, by
(\ref{result1}), if $x\in H^\prime$ is nicely reduced, then $x\in
S$. $\hfill (\ast)$

Let $H$ be the subgroup of $F$ generated by  $(J)\in
F$ where $J$ is generated by $n$ elements. We claim that $H =
H^\prime$. 

Clearly $H\subset H^\prime$. For other inclusion, it is enough to show
that $S \subset H$. For this, let $(J)\in
S$. Then $J$ is surjective image of a stably free $A$-module $P$ of
rank $n$. Applying (\ref{lemma3.7}), there
exists an ideal $J^\prime$ of height $n$ which is generated by $n$
elements such that $J'$ is surjective image of $P$.
By (\ref{prop1.2}), there exists an ideal $K$ of height $n$
comaximal with $J$ and $J'$ such that $J\cap K$ and $J'\cap K$ are
generated by $n$ elements. Hence $(J)\in H$ and $H=H'$.

Suppose $(J) = 0$ in $E_0(A)$. Then $(J)\in H = H^\prime$ is a
nicely reduced element. Hence, by $(\ast)$, we get that $J$ is
surjective image of a stably free $A$-module of rank $n$.
$\hfill \gj$
\end{proof}

\begin{proposition}\label{prop4.7}
Let $A$ be a Noetherian ring of dimension $n$. Let $P$ and $P_1$ be
two projective $A$-modules of rank $n$ such that $[P] = [P_1]$ in
$K_0(A)$. Then, there exists an ideal $J \subset A$ of height $\geq n$
such that $J$ is a surjective image of both $P$ and $P_1$.
\end{proposition}

\begin{proof}
Since $\dim A = n$ and $[P] = [P_1]$ in $K_0(A)$, we have $P\oplus Q
\simeq P_1\oplus Q$ for some projective $A$-module $Q$. We may assume
that $Q$ is free, by replacing $Q$ by $Q\oplus Q' \simeq A^t$. Now, it
follows that
$P\oplus A \iso P_1 \oplus A$ by the 
Bass Cancellation theorem (\ref{Bass}). 
Therefore, there exists a short exact
sequence
$$0 \ra P_1 \ra A \oplus P \longby {(b,-\alpha)} A \ra 0.$$ 
Further, without loss of generality, we may replace $\alpha$ by
$\alpha + b\gamma$ by a transvection, where $\gamma \in P^\ast$,
because this will not change the isomorphism class of ker
$((b,-\alpha)) = P_1$, i.e. if ker $((b,-(\alpha+b\gamma))) = P_2$,
then $P_1 \iso P_2$. Therefore, using (\ref{cor13}), we may assume
that the ideal $\alpha(P) = J$ is such that height $(J) \geq n$. By
(\ref{lem8} $(i)$), $J$ is also a surjective image of $P_1$. This
proves the proposition.  $\hfill \gj$
\end{proof}

\begin{corollary}\label{lem4.3}
Let $A$ be a Noetherian ring of even dimension $n$. Let $P$ be a
projective $A$-module of rank $n$ with trivial determinant. Then $e(P) =
0$ in $E_0(A)$ if and only if $[P] = [Q \oplus A]$ in $K_0(A)$ for
some projective $A$-module $Q$ of rank $n-1$.
\end{corollary}

\begin{proof}
First, assume that $[P] = [Q \oplus A]$ in $K_0(A)$ for
some projective $A$-module $Q$ of rank $n-1$. Then, by
(\ref{prop4.7}), $e(P)=e(Q\oplus A)$. By (\ref{theo}), $e(Q\oplus
A)=0$ in $E_0(A)$. Hence $e(P)=0$.

Now, we assume that $e(P) = 0$ in $E_0(A)$.
Let $\psi : P \surj J$ be a surjection, where $J$ is an
ideal of height $n$. Let $e(P,\chi) = (J,w_J)$, where $\chi$ is a
generator of $\wedge^n(P)$ and $w_J$ is a local orientation of $J$.
Since $e(P) = (J) = 0$
in $E_0(A)$, it follows from (\ref{theo}) that $J$ is a surjective
image of a stably free $A$-module of rank $n$. It follows now from
(\ref{lemma3.6}), that there exists a stably free $A$-module $\wt P$
of rank $n$ and a generator $\wt \chi$ of $\wedge^n(\wt P)$ such that
$e(\wt P,\wt \chi) = (J,w_J)$. Since $\wt P$ is a
stably free $A$-module of rank $n$, by (\ref{lemma3.7}), there
exists an ideal $J_1$ of height $n$ which is generated by $n$ elements
and a local orientation $w_{J_1}$ of $J_1$ such that $(J,w_J) =
(J_1,w_{J_1})$. Hence, we have $e(P,\chi) =
(J,w_J) = (J_1,w_{J_1})$. Let $J_1 = (b_1,\ldots,b_n)$. Then, by
(\ref{lemma3.6}), there exists a projective $A$-module $P^\prime$ of
rank $n$ with $[P^\prime] = [P]$ in $K_0(A)$ and a generator
$\chi_{P^\prime}$ of $\wedge^n(P^\prime)$ such that
$e(P^\prime,\chi_{P^\prime}) = (J_1,[\ol b_1,\ldots,\ol b_n]) = 0$ in
$E(A)$. But, then by (\ref{cor2.4}), $P^\prime$ has a unimodular
element. Hence $P^\prime = Q \oplus A$.  But $[P] = [P^\prime]$ in
$K_0(A)$. This proves the corollary.  $\hfill \gj$
\end{proof}

\begin{corollary}\label{prop4.4}
Let $A$ be a Noetherian ring of even dimension $n$. Let $P$ be a
projective $A$-module of rank $n$ with trivial determinant.  Suppose
that $e(P) = (J)$ in $E_0(A)$, where $J \subset A$ is an ideal of
height $n$. Then, there exists a projective $A$-module $Q$ of rank
$n$ such that $[Q] = [P]$ in $K_0(A)$ and $J$ is a surjective image
of $Q$.
\end{corollary}

\begin{proof}
By (\ref{cor14}), there exists a surjection $\psi : P \surj J\cap
J_1$, where $J_1$ is a height $n$ ideal such that $J + J_1 = A$. Since
$e(P) = (J) = (J \cap J_1)$ in $E_0(A)$, it follows that $(J_1) = 0$
in $E_0(A)$.
Using $\psi$ we have  
$$ e(P,\chi) = (J,w_J) + (J_1,w_{J_1})$$ 
in $E(A)$, where $\chi$ is a generator of $\wedge^n(P)$.  Since $(J_1)
= 0$ in $E_0(A)$, it follows from (\ref{theo}) that $J_1$ is the
surjective image of a stably free $A$-module of rank $n$. Therefore, by
(\ref{lemma3.6}), there exists a stably free $A$-module $P_1$ of
rank $n$ such that $e(P_1,\chi_1) = (J_1,w_{J_1})$, where $\chi_1$ is
a generator of $\wedge^n(P_1)$ and $w_{J_1}$ is a local orientation of
$J_1$.  By (\ref{lemma3.7}), we can choose an ideal $J_2$ of height
$n$ which is generated by $n$ elements and is comaximal with $J$ such
that $(J_1,w_{J_1}) = (J_2,w_{J_2})$ for some local orientation
$w_{J_2}$ of $J_2$. Therefore
$$e(P,\chi) = (J,w_J) + (J_2,w_{J_2}) = (J\cap J_2,w_{J\cap J_2}),$$
where $w_{J\cap J_2}$ is a local orientation of $J\cap J_2$ induced
from $w_J$ and $w_{J_2}$. Therefore, by (\ref{cor2.3}), there exists
a surjection from $P$ to $J\cap J_2$. Since $J_2$ is generated by $n$
elements, we can choose a set of generators $b_1,\ldots,b_n$ of
$J_2$. Let
$$(J,w_J) + (J_2,[\ol b_1,\ldots,\ol b_n]) = (J\cap J_2,\wt w_{J\cap
	J_2}).$$ 
By (\ref{lemma3.6}), there exists a projective $A$-module $Q$ 
with $[Q] = [P]$ in $K_0(A)$ such that 
$$e(Q,w_Q) = (J\cap J_2,\wt w_{J\cap J_2}) = (J,w_J).$$ 
	Hence, by (\ref{cor2.3}), there exists a surjection from $Q$ to
$J$. This proves the corollary.  $\hfill \gj$
\end{proof}

\begin{proposition}\label{prop4.5}
Let $A$ be a Noetherian ring of even dimension $n$ and let $J \subset
A$ be an ideal of height $n$ such that $J/J^2$ is generated by $n$
elements. Let $w_J : (A/J)^n \surj J/J^2$ be a surjection. Suppose
that the element $(J,w_J)$ of $E(A)$ belongs to the kernel of the
canonical homomorphism $E(A) \surj E_0(A)$. Then, there exists a
stably free $A$-module $P_1$ of rank $n$ and a generator $\chi_1$ of
$\wedge^n(P_1)$ such that $e(P_1,\chi_1) = (J,w_J)$ in $E(A)$.
\end{proposition}

\begin{proof}
Since $(J,w_J)\in E(A)$ belongs to the kernel of the canonical
homomorphism $E(A) \ra E_0(A)$, it follows that $(J) = 0$ in
$E_0(A)$. Hence, by (\ref{theo}), there exists a surjection $\alpha
: P \surj J$, where $P$ is stably free $A$-module of rank $n$.  Let
$\chi$ be a generator of $\wedge^n(P)$. Suppose that $(J,\wt w_J)$ is
obtained from $(\alpha,\chi)$. By (\ref{lem3.0}), there exists $a
\in A$ such that $\ol a \in A/J$ is a unit and $w_J = \ol a \wt
w_J$. By (\ref{lem3.1}), there exists a projective $A$-module $P_1$
of rank $n$ with $[P_1] = [P]$ in $K_0(A)$ and a generator $\chi_1$ of
$\wedge^n P_1$ such that $e(P_1,\chi_1) = (J,\ol {a^{n-1}}\wt w_J)$
in $E(A)$. Since $n$ is even, by (\ref{lem3.4}), we have $(J,\ol
{a^{n-1}}\wt w_J) = (J,\ol a \wt w_J)$ in $E(A)$. Hence
$e(P_1,\chi_1) = (J,w_J)$ in $E(A)$.  $\hfill \gj$
\end{proof}

\begin{corollary}\label{cor4.8}
Let $A$ be a  Noetherian ring of even dimension $n$. Let
$P$ be a  projective $A$-module of rank $n$ with
trivial determinant. Let $\alpha : P \surj J$ be a surjection, where
$J\subset A$ is an ideal of height $n$. Then $J$ is a surjective image
of a stably free $A$-module of rank $n$ if and only if $[P] = [Q
\oplus A]$ in $K_0(A)$ for some projective
$A$-module $Q$ of rank $n-1$.
\end{corollary}

\begin{proof}
Let $J$ be a surjective image of a stably free $A$-module of rank $n$.
Then, by (\ref{theo}), $(J) = 0$ in $E_0(A)$. Hence $e(P) = (J) = 0$
in $E_0(A)$. Applying (\ref{lem4.3}), the result follows.

The converse also follows from (\ref{theo}) and (\ref{lem4.3}). 
$\hfill \gj$
\end{proof}


\end{document}